\documentclass{amsart}
\usepackage{amsmath,amssymb,amsthm,amsfonts,amscd,mathrsfs,tikz-cd}
\usepackage[all,cmtip]{xy}
\usepackage{graphicx}
\usepackage{subcaption}

\theoremstyle{plain}

\newtheorem{theorem}{Theorem}[section]
\newtheorem{proposition}[theorem]{Proposition}
\newtheorem{lemma}[theorem]{Lemma}
\newtheorem{corollary}[theorem]{Corollary}

\newtheorem*{proposition*}{Proposition}

\newtheorem{apptheorem}{Theorem}[subsection]
\newtheorem{appproposition}[apptheorem]{Proposition}
\newtheorem{applemma}[apptheorem]{Lemma}

\theoremstyle{definition}
\newtheorem{definition}[theorem]{Definition}
\newtheorem{example}[theorem]{Example}

\newtheorem{appdefinition}[apptheorem]{Definition}

\theoremstyle{remark}
\newtheorem{remark}[theorem]{Remark}

\newcommand{\secref}[1]{Section~\ref{#1}}
\newcommand{\appref}[1]{Appendix~\ref{#1}}
\newcommand{\thmref}[1]{Theorem~\ref{#1}}
\newcommand{\propref}[1]{Proposition~\ref{#1}}
\newcommand{\lemref}[1]{Lemma~\ref{#1}}
\newcommand{\corref}[1]{Corollary~\ref{#1}}

\newcommand{\exref}[1]{Example~\ref{#1}}

\newcommand{\remref}[1]{Remark~\ref{#1}}

\newcommand{\defref}[1]{Definition~\ref{#1}}
\newcommand{\figref}[1]{Figure~\ref{#1}}

\numberwithin{subsection}{section}
\newcommand{\R}{\mathbb{R}}
\newcommand{\Z}{\mathbb{Z}}

\begin{document}

\title[Digital Fundamental Group]{A Fundamental Group for Digital Images}

\author{Gregory Lupton}
\author{John Oprea}
\author{Nicholas A. Scoville}

\address{Department of Mathematics, Cleveland State University, Cleveland OH 44115 U.S.A.}

\email{g.lupton@csuohio.edu}
\email{j.oprea@csuohio.edu}

\address{Department of Mathematics and Computer Science, Ursinus College, Collegeville PA 19426 U.S.A.}

\email{nscoville@ursinus.edu}

\date{\today}


\begin{abstract}   We define a fundamental group for digital images. Namely, we construct a functor from digital images to groups, which
closely resembles the ordinary fundamental group from algebraic topology.  Our construction
differs in several basic ways from previously established versions of a fundamental group in the digital setting.   Our development gives a prominent role to
subdivision of digital images.  We show that our fundamental group is preserved by subdivision.
\end{abstract}

\thanks{This work was partially supported by grants from the Simons Foundation: (\#209575 to Gregory Lupton
and \#244393 to John Oprea).}

\maketitle

\section{Introduction}

In digital topology, the basic object of interest  is a  \emph{digital image}: a finite set of integer lattice points in an ambient Euclidean space with a suitable adjacency relation between points.  This is an abstraction of an actual digital image which consists of  pixels (in the plane, or higher dimensional analogues of such).
There is an extensive literature on this topic with many results that bring notions from topology into this setting (e.g. \cite{Ro86, Bo99, Evako2006}).
In many instances, however, these notions from topology have been translated directly into the digital setting in a way that results in digital versions of topological notions  that  are very rigid   and hence have limited applicability.
 In contrast to this existing literature, in \cite{LOS19a, LOS19b} and in this paper, we have started to  build a  general ``digital homotopy theory" that brings the full strength of homotopy theory to the digital setting.  We use less rigid constructions, with a view towards broad applicability and greater depth of development.  A basic ingredient in our approach is \emph{subdivision}, which is a main focus of \cite{LOS19b}.  As part of this broad program, we focus here on the fundamental group.  The fundamental group is not new in digital topology (see \cite{Kong89, Bo99}, for example).  But our approach and development here differs from versions previously used in digital topology.  In fact, we end up with an invariant that differs from that of \cite{Bo99} for basic examples of digital images.  This difference appears to derive from the way in which we define our notion of homotopy, which  differs from that often used in the literature.  These points are explained at greater length below.    Furthermore, the particulars of our development together with our emphasis on subdivision as a basic ingredient allows us to go significantly beyond the kinds of results that have been established for previous versions of the fundamental group.  For example, we show that subdivision of a digital image preserves our fundamental group.

Much of the material we present here, through the definition of the fundamental group and some of its basic properties (such as behaviour with respect to products---\thmref{thm: products}), is independent of the material in  \cite{LOS19a, LOS19b}.  Most of the references we make to these other papers are for non-essential purposes in the context of examples or discussion.  But two of our main results here do use results from those papers.  Namely \thmref{thm: rho induces iso on pi1}, in which we show that subdivision preserves the fundamental group, depends on results about subdivision of maps and homotopies from \cite{LOS19b}.  Then  \thmref{thm: pi1D}, in which we calculate the fundamental group of a certain digital circle to be $\Z$, uses some material from \cite{LOS19a} about a digital version of the winding number.  
These  results from \cite{LOS19a, LOS19b} to which we refer here have lengthy proofs, or involve establishing  a great deal of notation, or both.  We have found it unfeasible to include all results and their proofs in a single paper.   Where possible, we have tried to keep overlap between this paper and \cite{LOS19a, LOS19b} to a minimum.   
The r{\'e}sum{\'e} of basic notions here (\secref{sec: basics}) and the material collected in \appref{Appx: technical} cover vocabulary and concepts that are also covered in \cite{LOS19a, LOS19b}.  However, because we consider the fundamental group here, we use based maps and homotopies, whereas in  \cite{LOS19a, LOS19b} we generally use unbased maps and homotopies.  So some of the basic results we collect here,  whilst superficially the same as ones of \cite{LOS19a, LOS19b}, are actually technically different from their unbased counterparts.  Overall, it seems reasonable to present this material on the fundamental group, along with necessary background on based maps and homotopies, separately from the material of  \cite{LOS19a, LOS19b}.

The paper is organized as follows.  In \secref{sec: basics} we review standard definitions and terminology, and set various conventions.  Items reviewed here include adjacency, products, homotopy, subdivision, and a combination of the latter two which we call subdivision homotopy, all from a based point of view. Our intention is to be brief here, in order to get to the main point---the fundamental group---as soon as possible.  For the same reason---to focus on the main points--- we have postponed to appendices the lengthy proof of a technical result that is necessary for one of our main results, and also some background material on based maps and homotopies.  
The main results are in \secref{sec: pi-one}, where we define our version of the fundamental group and establish some of its basic properties:
\thmref{thm: pi1} establishes our fundamental group;  \thmref{thm: indep basept} shows that it is independent of the choice of basepoint;  \thmref{thm: products} shows that it preserves products, in the same sense as for the ordinary fundamental group of topological spaces.
 In  \thmref{thm: rho induces iso on pi1}, we show that subdivision preserves the fundamental group.  Several consequences flow from this result. For instance, in \thmref{thm: subdn htpy equiv iso pi1} we show that the fundamental group is preserved by by a relation that is much less rigid than the relation of homotopy equivalence usually used in digital topology.   We give some basic calculations of the fundamental group in \corref{cor: subd contr pi1 triv} and \thmref{thm: pi1D}.
In a brief \secref{sec: Future Work}, we indicate some directions for future work.
\appref{sec: technical} contains the proof of a technical result required for the proof of \thmref{thm: rho induces iso on pi1}, In \appref{Appx: technical} we have expanded somewhat on the review of basics offered in \secref{sec: basics} and given the statements of two results from \cite{LOS19b}.  The material in  \appref{Appx: technical} is cited at various points in the main body, as well as in the proof of \appref{sec: technical}. 

\section{Basic Notions}\label{sec: basics}

A \emph{digital image} $X$ means a finite subset $X \subseteq \Z^n$ of the integral lattice in some $n$-dimensional Euclidean space, together with
a particular  adjacency relation inherited from that of $\Z^n$.  Namely, two (not necessarily distinct) points $x = (x_1, \ldots, x_n) \in \Z^n$ and  $y = (y_1, \ldots, y_n) \in \Z^n$  are adjacent if $|x_i-y_i| \leq 1$ for each $i = 1, \dots, n$.
If $x, y \in X \subseteq \Z^n$, we write $x \sim_X y$ to denote that $x$ and $y$ are adjacent in $X \subseteq \Z^n$.
A \emph{based digital image} is a pair $(X, x_0)$ where $X$ is a digital image and $x_0 \in X$ is some point of $X$, which we refer to as the \emph{basepoint} of $X$.

At the risk of some redundancy, we preserve the word \emph{based} in our nomenclature.  Thus we write based homotopy, based homotopy equivalence, and so-on.  On the other hand, we will usually suppress the basepoint $x_0$ from our notation unless it is useful to emphasize the particular basepoint.  Thus, we will denote a based digital image $(X, x_0)$ simply as $X$, with  the understanding that there is some choice of basepoint $x_0$.

For based digital images  $X \subseteq \Z^n$ and $Y \subseteq \Z^m$, a  function $f\colon X \to Y$  is \emph{continuous} if $f(x) \sim_Y f(y)$ whenever $x \sim_Xy$, and is \emph{based} if $f(x_0) = y_0$.
By a \emph{based map} of based digital images, we mean a continuous, based function.  Occasionally, we may encounter a non-continuous function, or a non-based map, of digital images.  But, mostly, we deal with based maps of based digital images.  The \emph{composition of based maps} $f\colon X \to Y$ and $g\colon Y \to Z$ gives a (continuous) based map $g\circ f\colon X \to Z$, as is easily checked from the definitions.

An \emph{isomorphism} of based digital images is a continuous, based bijection $f \colon X \to Y$ that admits a continuous inverse $g \colon Y \to X$, so that we have $f\circ g = \mathrm{id}_Y$ and $g\circ f = \mathrm{id}_X$ (such a $g$ is necessarily based and bijective).  Here, we say that $X$ and $Y$ are \emph{isomorphic} based digital images, and write $X \cong Y$.

We use the notation $I_N$ or $[0, N]$  for the \emph{digital interval of length} $N$. Namely, $I_N \subseteq \Z$ consists of the integers from $0$ to $N$ (inclusive)  in $\Z$ where consecutive
integers are adjacent. Thus, we have $I_1 = [0, 1] = \{0, 1\}$, $I_2 = [0, 2] = \{0, 1, 2\}$, and so-on.  Occasionally, we may use $I_0$ to denote the singleton point $\{0\} \subseteq \Z$.
We will consistently choose $0 \in I_N$ as the basepoint of an interval.

\begin{example}\label{ex:basic digital images}
 As an example in $\Z^2$, consider what we call  \emph{the Diamond}, $D = \{ (1, 0), (0, 1), (-1, 0), (0, -1) \}$, which may be viewed as a digital circle.  Note that pairs of vertices all of whose coordinates differ by $1$, such as  $(1, 0)$ and  $(0, 1)$ here, are adjacent according to our definition.  Otherwise, $D$ would be disconnected.
\begin{figure}[h!]
\centering
   \begin{subfigure}{0.49\linewidth} \centering
    \includegraphics[trim=160 350 80 80,clip,width=\textwidth]{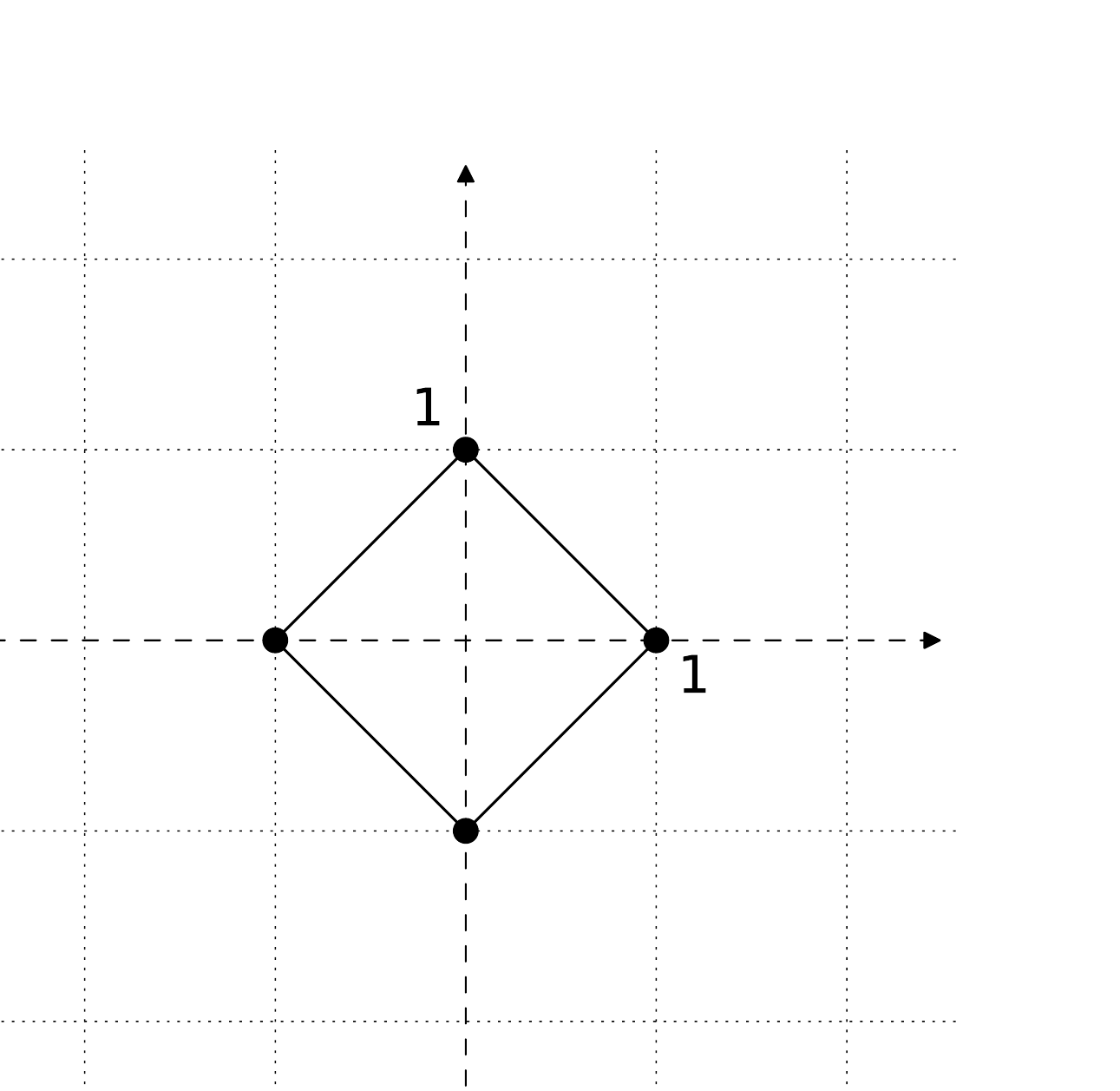}
     \caption{$D$: The Diamond}\label{fig:Circle D}
   \end{subfigure}
   \begin{subfigure}{0.49\linewidth} \centering
    \includegraphics[trim=160 350 80 80,clip,width=\textwidth]{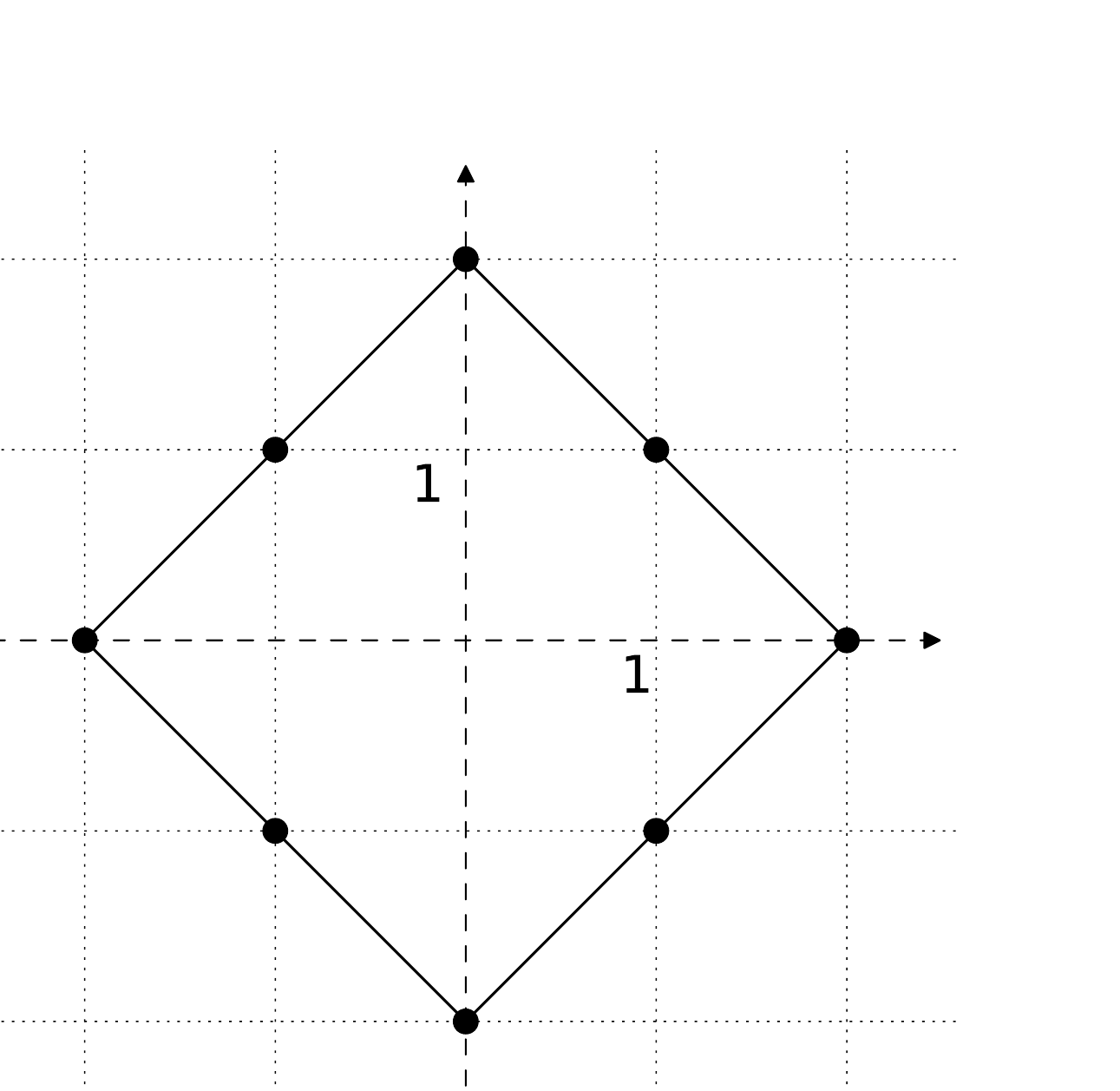}
     \caption{$C$: A larger digital circle}\label{fig:Circle C}
   \end{subfigure}
\caption{Two digital circles.  $D$ and $C$ are not homotopy equivalent, but are subdivision-homotopy equivalent---see \exref{ex: pi1 D and C}} \label{fig:D & C}
\end{figure}
In \figref{fig:D & C} we have included the axes (dashed) and also indicated adjacencies (solid) in the style of a graph.  Note, though, that we have no choice as to which points are adjacent: this is determined by position, or coordinates, and we cannot choose to add or remove edges here.
Also, consider the digital image   $C = \{ (2, 0), (1, 1), (0, 2), (-1, 1), (-2, 0), (-1, -1), (0, -2), (1, -1) \}$ (see \figref{fig:D & C}).  Both $D$ and $C$ may be viewed as  digital circles.  However, the only maps $D \to C$ will be ``homotopically trivial" maps: we cannot ``wrap" a smaller circle around a larger one.  We explain this remark in detail in \exref{ex: pi1 D and C} below.
\end{example}

\begin{definition}\label{def: products}
The product of based digital images $(X, x_0)$ with $X\subseteq \Z^m$ and $(Y, y_0)$ with $Y\subseteq \Z^n$ is $\big(X \times Y, (x_0, y_0)\big)$.  Here,  the Cartesian product  $X \times Y \subseteq \Z^{m} \times \Z^{n} \cong \Z^{m+n}$ has the adjacency relation $(x, y) \sim_{X \times Y} (x', y')$ when $x\sim_X x'$  and $y \sim_Y y'$.
\end{definition}

We use a  ``cylinder object" definition of homotopy.  This is the one commonly used in  the digital topology literature, though with a different notion of adjacency (see Remark \ref{rem: graph product homotopy} below).  In \cite{LOS19a} we give a fuller discussion of homotopy, including a ``path object" definition as well.

\begin{definition}\label{def: Based Homotopy}
Let $f, g\colon X \to Y$ be based maps of based digital images.
We say that $f$ and $g$ are \emph{based homotopic}, and write $f \approx g$,  if, for some $N\geq 1$, there is a (continuous) based map
$$H \colon X \times I_N \to Y,$$
with $H(x, 0) = f(x)$ and $H(x, N) = g(x)$, and $H(x_0, t) = y_0$ for all $t=0, \ldots, N$.  Then $H$ is a \emph{based homotopy} from $f$ to $g$.
\end{definition}

In \defref{def: based h.e.} we give the notion of based homotopy equivalence that flows from this definition of based homotopy.
But the notion of based homotopy equivalence is really too rigid for our purposes.  Instead, we seek to develop a less rigid notion of ``sameness"  for digital images that incorporates \emph{subdivision} (defined in \defref{def: subdivision} below) and seems better suited to homotopy theory in the digital setting.  We call this notion \emph{subdivision-based homotopy equivalence}, and define it in  \defref{def: subdn homotopy equiv}.

\begin{remark}\label{rem: graph product homotopy}
(Based) homotopy of digital maps has been studied by Boxer and others (see, e.g. \cite{Bo99, Bo05}).  Our definition of homotopy above is visually the same as that of these authors.  There is a technical difference, however, in that they take the ``graph product" adjacency relation in the product $X \times I_N$, and not the adjacency relation we use (cf.~remarks after Definition 2.5 of \cite{Bo06}).  The difference is akin to requiring a function of two variables to be separately or jointly continuous.  Therefore, our homotopies must preserve more adjacencies than those of  \cite{Bo99}, and this fact has important consequences.   Note that  various choices of adjacency relation on a product are discussed  in  \cite{Bo18}.
For instance, in \cite{Bo99} (following Th.3.1 there) it is shown that, using the notion of homotopy that derives from the ``graph product,"  the Diamond is contractible.  However, using the notion of homotopy as we have defined it, the contracting homotopy used in \cite{Bo99} fails to be continuous.  In fact, in \cite{LOS19a}, we show that the Diamond $D$ is not contractible.
\end{remark}

The notion of \emph{subdivision of a digital image} plays a fundamental role in our development of ideas in the digital setting.  
Here, we give the minimum amount of information about this sufficient for our purposes.
See \cite{LOS19b} for a full discussion, with illustrative examples, of subdivision of digital images and, especially, subdivision of maps.

\begin{definition}\label{def: subdivision}
Suppose that $X$ is a digital image in $\Z^n$.  For each $k \geq 2$, we have the $k$-\emph{subdivision of $X$}, which is an auxiliary (to $X$) digital image also in $\Z^n$ denoted by $S(X, k)$, with a \emph{canonical map} or \emph{standard projection}
$$\rho_k\colon  S(X, k) \to X$$
that is continuous.  For a real number $x$, denote by $\lfloor x \rfloor$ the greatest integer less-than-or-equal-to $x$.  First, make
the $\Z[1/k]$-lattice in $\R^n$, namely, those points with coordinates each of which is $z/k$ for some integer $z$, and then set
$$S'(X, k) = \left\{ (x_1, \ldots, x_n) \in \left(\Z\left[\frac{1}{k}\right]\right)^n  \mid ( \lfloor x_1 \rfloor, \ldots, \lfloor x_n \rfloor) \in X \right\}.$$
Then set
$$ S(X,k) = \left\{ (kx_1, \ldots, kx_n) \in \Z^n  \mid (x_1, \ldots, x_n) \in S'(X, k) \right\}.$$
The map $\rho_k$ is given by $\rho_k\big( (y_1, \ldots, y_n) \big) = ( \lfloor y_1/k \rfloor, \ldots, \lfloor y_n/k \rfloor)$;  one checks that this map is continuous.

For $x \in X$ an individual point, we write $S(x, k) \subseteq S(X, k)$ for the points $y \in S(X, k)$ that satisfy $\rho_k(y) = x$.  If $x = (x_1, \ldots, x_n)$, then we may describe this set in general as
$$
S(x, k) = \{ (kx_1 + r_1, \ldots, kx_n + r_n) \mid 0 \leq r_i \leq k-1 \}.
$$
\end{definition}

\begin{definition}[Convention on basepoints]\label{def:subdn basepoints}
Suppose $Y\subseteq \Z^n$ is a based digital image with basepoint $y_0 = (y_1, \ldots, y_n)$.  In any subdivision $S(Y, k)$ of $Y$ we may take $\overline{y_0} \in  S(y_0, k)\subseteq S(Y, k)$ as the basepoint, where $\overline{y_0}$ denotes the point whose coordinates are
$$\overline{y_0}= \begin{cases} (2ky_1 + k-1, \ldots, 2ky_n + k-1) & \text{ in } S(Y, 2k)\\ \big( (2k+1)y_1 + k, \ldots, (2k+1)y_n + k\big) & \text{ in } S(Y, 2k+1).\end{cases}$$
In an odd subdivision, $\overline{y_0}$ is the centre point of $S(y_0, 2k+1)$, which is a cubical  lattice in $\Z^n$, each side of which contains $2k+1$ points.  In an even subdivision, $S(y_0, 2k)$ does not have a centre point, as such, since there is no middle point in an interval which contains $2k$ points.  Rather, $\overline{y_0}$ is a corner of the \emph{central clique} of  $S(y_0, 2k)$, which is to say a unit cubical  lattice in $\Z^n$ that may be considered as the center of $S(y_0, 2k)$.  Usually, and unless there is some reason not to do so, we will choose  $\overline{y_0}$ as the basepoint of a subdivision $S(Y, k)$.   One exception to this is  our convention that $0 \in I_N$ is the basepoint of an interval. This means that when we  subdivide an interval, the basepoint of $S(I_N, k) = I_{Nk+k-1}$ will be $0$ rather than $\overline{0} = k$.   However we choose the basepoint in a subdivision, so long as we have the basepoint of $S(Y, k)$ to be some point in $S(y_0, k) \subseteq S(Y, k)$, then the standard projection $\rho_k \colon S(Y, k) \to Y$ is a based map.  In particular, note  that we have $\rho_k(\overline{y_0}) = y_0$.   Furthermore, with our convention that $\overline{y_0}$ is the basepoint of $S(Y, k)$, the partial projections $\rho^c_k\colon S(Y, k) \to S(Y, k-1)$ for each $k \geq 3$ of \defref{def: rho^c} are also based maps.  Indeed, from the formulas of \defref{def: rho^c}, for
$\rho^c_{2k+1}\colon S(Y, 2k+1) \to S(Y, 2k)$, with $k\geq 1$,  we have
$$\begin{aligned}
\rho^c_{2k+1}(\overline{y_0}) &= \rho^c_{2k+1}\big( (2k+1)y_1 + k, \ldots, (2k+1)y_n + k \big)  \\
&= (2ky_1 + k-1, \ldots, 2ky_n + k-1) = \overline{y_0} \in S(Y, 2k).\end{aligned}$$
On the other hand, for
$\rho^c_{2k}\colon S(Y, 2k) \to S(Y, 2k-1)$, with $k\geq 2$,  we have
$$\begin{aligned}
\rho^c_{2k}(\overline{y_0}) &= \rho^c_{2k}( 2ky_1 + k-1, \ldots, 2ky_n + k -1\big)  \\
&= \big((2k-1)y_1 + k-1, \ldots, (2k-1)y_n + k-1) = \overline{y_0} \in S(Y, 2k-1),\end{aligned}$$
with this last identification following from the above definition of $\overline{y_0} \in S(Y, 2k-1)$ because we have $2k-1 = 2(k-1)+1$.   In all cases, that is, we have
$\rho^c_k(\overline{y_0}) = \overline{y_0}$, and  $\rho^c_k$ is a based map.
\end{definition}

Generally, subdivision of an interval $I_N \subseteq \Z$ gives a longer interval: $S(I_N, k) = I_{Nk+k-1} \subseteq \Z$.
We also note here that $S(I_0, k) = S(\{0\}, k) = I_{k-1}$.
We adopt the notational convention that $S(X, 1) = X$, and $\rho_1\colon  S(X, 1) \to X$ is just the identity map of $X$.

The less rigid notion of homotopy that incorporates subdivision, indicated above, is as follows.   Note that we may iterate subdivision.  Here, and in the sequel, we make identifications such as $S\big( S(X, k), k'\big) \cong S(X, kk')$.

\begin{definition}\label{def: subdn homotopic maps}
Suppose $X\subseteq \Z^m$ and $Y \subseteq \Z^n$ are based digital images.  Assume a choice of basepoint  in each subdivision $S(X, k)$  (such as $\overline{x_0}$, or $0$ in case $X$ is an interval).
Two  based maps $f \colon S(X, k) \to Y$ and $g \colon S(X, l) \to Y$ are \emph{subdivision-based homotopic} if, for some $k', l'$ with $kk' = ll' = m$, we have a based homotopy
$$f\circ \rho_{k'} \approx g\circ \rho_{l'} \colon S(X, m)  \to Y.$$
In particular,  maps $f, g\colon X \to Y$ are subdivision-based homotopic if we have a based homotopy $f\circ\rho_k \approx g\circ\rho_k\colon S(X, k) \to Y$, for some $\rho_k\colon S(X, k) \to X$.
\end{definition}

\section{The Fundamental Group}\label{sec: pi-one}

In \cite{Bo99} (see also \cite{BS16, B-S18}),     Boxer has  defined a version of the fundamental group for a digital image.  We give a  complete development of this topic here, because our development, whilst superficially similar, differs from that of \cite{Bo99} in several basic details.  We will point out these differences as we go along.  In fact, we end up with a different invariant: there are basic examples of 2D digital images for which our fundamental group is not isomorphic to that of \cite{Bo99}.  In earlier work \cite{Kong89} (cf.~ \cite{K-R-R92}), Kong has also defined a fundamental group, but that version applies to restricted types of digital image and furthermore the general approach taken there differs somewhat from our approach and that of \cite{Bo99}.  We mention also that a brief treatment of a fundamental group for digital images is given in \cite{Kha87}, but there the approach taken seems quite  different again.

Let $(Y, y_0)$ be a based digital image with $Y \subseteq \Z^n$.  For any  $N \geq 1$, a \emph{based path of length $N$ in $Y$} is a continuous map $\alpha\colon I_N \to Y$ with $\alpha(0) = y_0$.  Unlike in the ordinary (topological) homotopy setting, where any path may be taken with the fixed domain $[0, 1]$, in the digital setting we must allow paths to have different domains.

\begin{definition}\label{def: based loop}
Let $(Y, y_0)$ be a based digital image with $Y \subseteq \Z^n$.
A \emph{based loop of length $N$} in $Y$ is a path $\gamma\colon I_N \to Y$ that satisfies $\gamma(0) = \gamma(N) = y_0$.  If $\gamma\colon I_N \to Y$ is a based loop in $Y$, then for any $k$,
$$\gamma\circ \rho_k\colon S(I_N, k) \to I_N \to Y$$
is also a based loop (of length $kN + k-1$), in that we have $S(I_N, k)  = I_{kN + k-1}$ and $\gamma\circ \rho_k(0) = \gamma(0) = y_0$ and $\gamma\circ \rho_k(kN+k-1) = \gamma(N) = y_0$.
\end{definition}

Geometrically speaking, the composition $\gamma\circ \rho_k$ amounts to a reparametrization of the loop $\gamma$.  The image traced out in $Y$  is the same, but we pause at each point of the loop for an interval of length $k-1$.  This device allows us to compare loops of different lengths, and also gives much greater flexibility in deforming loops by (based) homotopies.

We specialize \defref{def: Based Homotopy} to the context of based loops as follows.

\begin{definition}\label{def: Based Homotopy of loops}
Given a based digital image $Y \subseteq \Z^n$, we say that based loops $\alpha, \beta \colon I_M \to Y$ (of the same length) are \emph{based homotopic as based loops} if there is a based homotopy $H\colon I_M \times I_N \to Y$ with $H(0, t) = H(M, t) = y_0$ for all $t \in I_N$. We refer to such a homotopy as a based homotopy of based loops.
\end{definition}

For our fundamental group, we use a slightly different way of forming equivalence classes of loops from that used in \cite{Bo99}.  Whereas \cite{Bo99} uses the notion of ``trivial extensions'' of a path, we use a path pre-composed with a standard projection $\rho\colon S(I_N, k) \to I_N$, which is a particular type of---a sort of ``regular," or evenly distributed---trivial extension.

We specialize \defref{def: subdn homotopic maps} to the context of based loops as follows.

\begin{definition}\label{def: subdn homotopic loops}
Two based loops $\alpha \colon I_M \to Y$ and $\beta \colon I_N \to Y$ (generally of different lengths) are \emph{subdivision-based  homotopic as based loops} if, for some $k, l$ with $k, l \geq 1$ and $k(M+1) = l(N+1)$, we have
$$\alpha\circ \rho_k\colon S(I_{M}, k) \to I_M \to Y \quad \text{and} \quad \beta\circ \rho_l\colon S(I_{N}, l) \to I_N \to Y$$
based-homotopic as maps $S(I_{M}, k) = I_{kM + k-1} = I_{lN + l-1} =  S(I_{N}, l) \to Y$, via a based homotopy of based loops; i.e.,  if we have a homotopy $H\colon I_{kM + k-1} \times I_R \to Y$ that satisfies $H(s, 0) = \alpha\circ \rho_k(s)$ and  $H(s, R) = \beta\circ \rho_l(s)$, and also $H(0, t) = H(kM + k-1, t) = y_0$ for all $t \in I_R$.
\end{definition}

\begin{lemma}\label{lem: subdn homotopy is equiv}
Let $Y\subseteq \Z^n$ be a based digital image.  Subdivision-based homotopy of based loops is an equivalence relation on the set of all based loops (of all lengths)  in $Y$.
\end{lemma}

\begin{proof}
This is proved in \propref{prop: subdn-bsd homotopy is equiv of loops}.
\end{proof}

Denote by $[\alpha]$ the (subdivision-based homotopy) equivalence class of based loops represented by a based loop $\alpha\colon I_N \to Y$.  Thus, we have $[\alpha] = [\alpha\circ \rho_k]$ for any standard projection $\rho_k\colon S(I_N, k) \to I_N$. More generally, we write $[\alpha] = [\beta]$ whenever $\alpha$ and $\beta$ are subdivision-based homotopic as based loops in $Y$.

From here, the development of the fundamental group follows exactly that of the topological setting (see \cite[Chap.II]{Mas91}, for example).    This plan is used in \cite{Bo99}; we  follow the same plan here.  However, we must adapt the details in a number of ways from those of \cite{Bo99}  because our basic ingredients differ somewhat: we have adopted a fixed adjacency relation on our digital images; we have defined  homotopy in a way that differs from that of \cite{Bo99} (cf.~\remref{rem: graph product homotopy}); we use subdivision-based homotopy for our equivalence relation on based loops rather than the trivial extensions of \cite{Bo99}; we concatenate loops in a slightly different way from \cite{Bo99} (see the next item).

At several points in the development, we will work in the context of paths, and not just loops.  For this reason, we give a general definition of concatentation that, in particular, may be applied to based loops.

\begin{definition}\label{def: concatenation}
Suppose $\alpha\colon I_M \to Y$ and
$\beta\colon I_N \to Y$ are paths in $Y$ that satisfy $\alpha(M) \sim_Y \beta(0)$.  Their \emph{concatenation} is  the path $\alpha\cdot\beta\colon I_{M+N+1} \to Y$ of length $M+N+1$ in  $Y$ defined by
\begin{equation}\label{eq: concat}
\alpha\cdot\beta(t) = \begin{cases} \alpha(t) & 0 \leq t \leq M\\  \beta(t-(M+1)) & M+1 \leq t \leq M+N+1.\end{cases}
\end{equation}
If $\alpha(M) = \beta(0)$, then our definition means that we pause for a unit interval when attaching the end of $\alpha$ to the start of $\beta$. In this case, we may also define the shorter version, which is commonly used in the literature but which we generally do not use except at one point in the development.  In contrast to our concatenation, and only if $\alpha(M) = \beta(0)$, we define their \emph{short concatenation} as the path  of length $M+N$ in  $Y$ given by
\begin{equation}\label{eq: short concat}
\begin{aligned}
\alpha*\beta(t) &= \begin{cases} \alpha(t) & 0 \leq t \leq M-1\\  \beta(t-M) & M \leq t \leq M+N \end{cases}\\
&= \begin{cases} \alpha(t) & 0 \leq t \leq M\\  \beta(t-(M+1)) & M+1 \leq t \leq M+N.\end{cases}
\end{aligned}
\end{equation}
\end{definition}

So given two based loops $\alpha\colon I_M \to Y$ and
$\beta\colon I_N \to Y$, we  form their \emph{product} by concatenation:
$$\alpha\cdot\beta\colon I_{M+N+1} \to Y$$
is the based loop of length $M + N +1$ defined by \eqref{eq: concat}.
We pause at  the basepoint for a unit interval when attaching the end of $\alpha$ to the start of $\beta$.  Concatenating in this way is crucial for the compatibility of subdivision homotopy and the product (part (b) of \lemref{lem: loop product basics} below).   This product of based loops is strictly associative, as is easily checked.

In the following result, and in the sequel, write $C_N \colon I_N \to Y$ for the constant loop defined by $C_N(i) = y_0$ for $0 \leq i \leq N$.  Just as in the topological setting, (the class represented by) this loop will play the role of the identity element.  Also, recall our notational convention that if $k=1$,  then $\rho_k\colon  S(X, k) \to X$ just means the identity map of $X$.

\begin{lemma}\label{lem: loop product basics}
Let $Y \subseteq \Z^n$ be any based digital image.
\begin{itemize}
\item[(a)]  Suppose we have pairs of based-homotopic loops
$\alpha \approx \alpha'\colon I_M \to Y$ and $\beta \approx \beta'\colon I_N \to Y$.
Then  we have  a based homotopy of based loops
$\alpha\cdot\beta \approx \alpha'\cdot\beta' \colon I_{M+N+1} \to Y$.
\item[(b)]  For any based loops
$\alpha \colon I_M \to Y$ and $\beta\colon I_N \to Y$ and any $k$, we have equality of based loops
$$(\alpha\cdot\beta)\circ\rho_k = (\alpha\circ\rho_k)\cdot (\beta\circ\rho_k)\colon I_{k(M+N+1) + k-1} \to Y.$$
\item[(c)]  Given a based loop $\alpha \colon I_M \to Y$ and $k \geq 2$, we have based homotopies of based loops
$$\alpha\circ \rho_k \approx (\alpha\circ\rho_{k-1})\cdot C_M\colon I_{kM + k-1} \to Y$$
and $\alpha\circ \rho_k \approx C_M\cdot (\alpha\circ\rho_{k-1})\colon I_{kM + k-1} \to Y$.
\end{itemize}
\end{lemma}

\begin{proof}
(a) Suppose we have based homotopies of based loops $H\colon I_M \times I_R \to Y$ and $G\colon I_N \times I_T \to Y$ from $\alpha$ to $\alpha'$ and from $\beta$ to $\beta'$ respectively.  We first, if necessary, adjust one of the intervals $I_R, I_T$  so that both homotopies are of the same length.  Suppose we have $R < T$ (the case in which  $R > T$ is handled similarly, and we omit it).  Then lengthen $H$ into a based homotopy $H' \colon I_M \times I_T \to Y$ defined as
$$H' (s, t) = \begin{cases} H(s, t) & 0 \leq t \leq R \\ H(s, R) & R+1 \leq t \leq T.\end{cases}$$
Allowing this to be continuous on  $I_M \times I_T$, it is clearly a based homotopy of based loops from $\alpha$ to $\alpha'$.  We just need to be a little careful to check continuity, bearing in mind our adjacencies on the product.  To this end, say we have $(s, t) \sim_{I_M \times I_T} (s', t')$.  Since $t\sim_{I_T} t'$, we must have either $\{t, t'\} \subseteq [0, R]$ or $\{t, t'\} \subseteq [R, T]$.  If $\{ (s, t),  (s', t')\} \subseteq I_M \times I_R$, then continuity of $H$ gives
$H'(s, t) \sim_{Y} H'(s', t')$.  If  $\{ (s, t),  (s', t')\} \subseteq I_M \times [R, T]$, then we have $H'(s, t) = H(s, R) \sim_{Y} H(s', R) = H'(s', t')$. It follows that $H'$ is continuous.    Now define a homotopy (with $H' = H$ in case the original $R$ and $T$ are equal)  $H'+G\colon I_{M+N+1} \times I_T \to Y$ as
$$(H' +G)(s, t) = \begin{cases} H'(s, t) & 0 \leq s \leq M \\ G(s-(M+1), t) & M+1 \leq s \leq M+N+1.\end{cases}$$
Once again, if continuous on  $I_{M+N+1} \times I_T$, this is clearly a based homotopy from $\alpha\cdot\beta$ to $\alpha'\cdot\beta'$.  To check the two homotopies assemble together continuously, we observe  that, if $(s, t) \sim_{I_{M+N+1} \times I_T} (s', t')$, then either $\{s, s'\} \subseteq [0, M+1]$ or $\{s, s'\} \subseteq [M+1, M+N]$.  Then proceeding as in the first part, and using $H(M, t) = G(0, t) = y_0$, we confirm the continuity of  $(H' +G)$.

(b) We have $\alpha\circ\rho_k \colon I_{kM + k-1} \to Y$ and $\beta\circ\rho_k \colon I_{kN + k-1} \to Y$.  Thus,
$$(\alpha\circ\rho_k)\cdot (\beta\circ\rho_k)\colon I_{(kM + k-1) + (kN + k-1) + 1} \to Y$$
is given by

\bigskip

$(\alpha\circ\rho_k)\cdot (\beta\circ\rho_k)(t) =$
$$\begin{cases} \alpha\circ\rho_k(t) & 0 \leq t \leq kM + k-1\\
\beta\circ\rho_k(t-(kM+k)) & kM+k \leq t \leq (kM + k-1) + (kN + k-1) + 1.\end{cases}$$
Note that we have $(kM + k-1) + (kN + k-1) + 1 = k(M+N+1) + (k-1)$ and that, for  $kM+k \leq t \leq k(M+N+1) + (k-1)$, we have
$\beta\circ\rho_k(t-(kM+k)) = \beta\big( \rho_k(t) - (M+1)\big)$.  Thus this product agrees with the composition  $(\alpha\cdot\beta)\circ\rho_k$.

(c)  For the first assertion, we claim that a based homotopy of based loops $H\colon I_{kM + k-1} \times I_{M+1}\to Y$ from $\alpha\circ \rho_k$ to $(\alpha\circ\rho_{k-1})\cdot C_M$
is given by
$$H(s, t) = \begin{cases} \alpha\circ\rho_{k-1}(s)& 0 \leq s \leq t(k-1) - 1\\
\alpha\circ\rho_{k}(s+t)& t(k-1) \leq s \leq kM + k-1 - t\\
y_0 & kM + k - t \leq s \leq kM + k -1.
\end{cases}
$$
The idea of the deformation is that, for each fixed $t$ as we progress through $I_{M+1}$, we pause for one fewer amount of time at the first $t$ values of the path $\alpha$, and then make up the length of the path at the endpoint by waiting there for $t$ points.  The homotopy  is illustrated in \figref{fig:homotopy alpha c}.
\begin{figure}[h!]
\centering
$$
\begin{tikzpicture}[scale=.5]
\draw[step=1cm,lightgray,very thin] (-2.2,-1.2) grid (20.2,8.2);

\draw[thick, black]  (-2.2,0)--(20.2,0);
\draw[thick, black]  (0,8.2)--(0,-1.2);

\node[inner sep=2pt, circle](25) at (0,0)[draw] {};
\node[inner sep=2pt, circle](25) at (0,1) [draw]{};
\node[inner sep=2pt, circle](25) at (0,2)[draw] {};
\node[inner sep=2pt, circle](25) at (0,3)[draw] {};
\node[inner sep=2pt, circle](25) at (0,4)[draw] {};
\node[inner sep=2pt, circle](25) at (0,5)[draw] {};

\node[inner sep=2pt, circle](25) at (1,0)[draw] {};
\node[inner sep=2pt, circle](25) at (1,1) [draw]{};
\node[inner sep=2pt, circle](25) at (1,2) [draw]{};
\node[inner sep=2pt, circle](25) at (1,3) [draw]{};
\node[inner sep=2pt, circle](25) at (1,4) [draw]{};
\node[inner sep=2pt, circle](25) at (1,5)[draw] {};

\node[inner sep=2pt, circle](25) at (2,0)[draw] {};
\node[inner sep=2pt, circle](25) at (2,1)[draw] {};
\node[inner sep=2pt, circle](25) at (2,2) [draw]{};
\node[inner sep=2pt, circle](25) at (2,3) [draw]{};
\node[inner sep=2pt, circle](25) at (2,4) [draw]{};
\node[inner sep=2pt, circle](25) at (2,5) [draw]{};

\node[inner sep=2pt, circle](25) at (3,0) [draw]{};
\node[inner sep=2pt, circle](25) at (3,1) [draw]{};
\node[inner sep=2pt, circle](25) at (3,2) [draw]{};
\node[inner sep=2pt, circle](25) at (3,3) [draw]{};
\node[inner sep=2pt, circle](25) at (3,4) [draw]{};
\node[inner sep=2pt, circle](25) at (3,5) [draw]{};

\node[inner sep=2pt, circle](25) at (4,0) [draw]{};
\node[inner sep=2pt, circle](25) at (4,1) [draw]{};
\node[inner sep=2pt, circle](25) at (4,2) [draw]{};
\node[inner sep=2pt, circle](25) at (4,3) [draw]{};
\node[inner sep=2pt, circle](25) at (4,4) [draw]{};
\node[inner sep=2pt, circle](25) at (4,5) [draw]{};

\node[inner sep=2pt, circle](25) at (5,0) [draw]{};
\node[inner sep=2pt, circle](25) at (5,1) [draw]{};
\node[inner sep=2pt, circle](25) at (5,2) [draw]{};
\node[inner sep=2pt, circle](25) at (5,3) [draw]{};
\node[inner sep=2pt, circle](25) at (5,4) [draw]{};
\node[inner sep=2pt, circle](25) at (5,5) [draw]{};

\node[inner sep=2pt, circle](25) at (6,0) [draw]{};
\node[inner sep=2pt, circle](25) at (6,1)[draw] {};
\node[inner sep=2pt, circle](25) at (6,2) [draw]{};
\node[inner sep=2pt, circle](25) at (6,3) [draw]{};
\node[inner sep=2pt, circle](25) at (6,4) [draw]{};
\node[inner sep=2pt, circle](25) at (6,5) [draw]{};

\node[inner sep=2pt, circle](25) at (7,0) [draw]{};
\node[inner sep=2pt, circle](25) at (7,1) [draw]{};
\node[inner sep=2pt, circle](25) at (7,2) [draw]{};
\node[inner sep=2pt, circle](25) at (7,3) [draw]{};
\node[inner sep=2pt, circle](25) at (7,4) [draw]{};
\node[inner sep=2pt, circle](25) at (7,5) [draw]{};

\node[inner sep=2pt, circle](25) at (8,0) [draw]{};
\node[inner sep=2pt, circle](25) at (8,1) [draw]{};
\node[inner sep=2pt, circle](25) at (8,2) [draw]{};
\node[inner sep=2pt, circle](25) at (8,3) [draw]{};
\node[inner sep=2pt, circle](25) at (8,4) [draw]{};
\node[inner sep=2pt, circle](25) at (8,5) [draw]{};

\node[inner sep=2pt, circle](25) at (9,0) [draw]{};
\node[inner sep=2pt, circle](25) at (9,1) [draw]{};
\node[inner sep=2pt, circle](25) at (9,2) [draw]{};
\node[inner sep=2pt, circle](25) at (9,3) [draw]{};
\node[inner sep=2pt, circle](25) at (9,4) [draw]{};
\node[inner sep=2pt, circle](25) at (9,5) [draw]{};

\node[inner sep=2pt, circle](25) at (10,0) [draw]{};
\node[inner sep=2pt, circle](25) at (10,1) [draw]{};
\node[inner sep=2pt, circle](25) at (10,2) [draw]{};
\node[inner sep=2pt, circle](25) at (10,3) [draw]{};
\node[inner sep=2pt, circle](25) at (10,4) [draw]{};
\node[inner sep=2pt, circle, fill=black](25) at (10,5) [draw] {};

\node[inner sep=2pt, circle](25) at (11,0) [draw]{};
\node[inner sep=2pt, circle](25) at (11,1) [draw]{};
\node[inner sep=2pt, circle](25) at (11,2) [draw]{};
\node[inner sep=2pt, circle](25) at (11,3) [draw]{};
\node[inner sep=2pt, circle, fill=black](25) at (11,4)[draw]  {};
\node[inner sep=2pt, circle, fill=black](25) at (11,5)[draw]  {};

\node[inner sep=2pt, circle](25) at (12,0) [draw]{};
\node[inner sep=2pt, circle](25) at (12,1) [draw]{};
\node[inner sep=2pt, circle](25) at (12,2) [draw]{};
\node[inner sep=2pt, circle, fill=black](25) at (12,3) [draw] {};
\node[inner sep=2pt, circle, fill=black](25) at (12,4) [draw] {};
\node[inner sep=2pt, circle, fill=black](25) at (12,5) [draw] {};

\node[inner sep=2pt, circle](25) at (13,0) [draw]{};
\node[inner sep=2pt, circle](25) at (13,1) [draw]{};
\node[inner sep=2pt, circle, fill=black](25) at (13,2) [draw] {};
\node[inner sep=2pt, circle, fill=black](25) at (13,3) [draw] {};
\node[inner sep=2pt, circle, fill=black](25) at (13,4) [draw] {};
\node[inner sep=2pt, circle, fill=black](25) at (13,5) [draw] {};

\node[inner sep=2pt, circle](25) at (14,0) [draw]{};
\node[inner sep=2pt, circle, fill=black](25) at (14,1) [draw] {};
\node[inner sep=2pt, circle, fill=black](25) at (14,2) [draw] {};
\node[inner sep=2pt, circle, fill=black](25) at (14,3) [draw] {};
\node[inner sep=2pt, circle, fill=black](25) at (14,4) [draw] {};
\node[inner sep=2pt, circle, fill=black](25) at (14,5) [draw] {};

\node[black] at (-.5,7.5) {$t$};
\node[black] at (18.5,-.5) {$s$};
\node[black] at (7,-.6) {$\alpha\circ\rho_3\colon I_{14}\to Y$};
\node[black] at (-1.5,3) {\large{$I_5$}};

\node[black] at (4,6.5) {$\alpha\circ\rho_2\colon I_{9}\to Y$};
\node[black] at (12,6.5) {$C_4\colon I_{4}\to Y$};

\draw[very thick]  (1,5)--(1,1)--(2,0);
\draw[very thick]  (3,5)--(3,2)--(5,0);
\draw[very thick]  (5,5)--(5,3)--(8,0);
\draw[very thick]  (7,5)--(7,4)--(11,0);
\draw[very thick]  (9,5)--(14,0);

\end{tikzpicture}
$$
\caption{\label{fig:homotopy alpha c}  $\alpha\circ \rho_k \approx (\alpha\circ\rho_{k-1})\cdot C_M$ (illustrated with $\alpha\colon I_4 \to Y$ and $k = 3$).}
\end{figure}
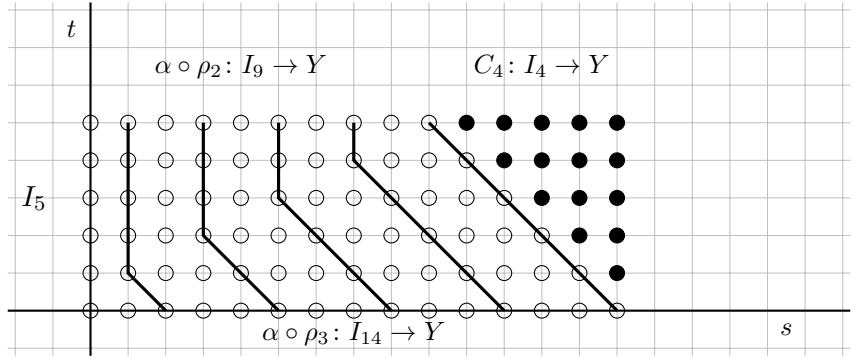
In the figure, black dots represent points sent to $y_0$ and open dots those points that are aggregated by either $\rho_k$ or $\rho_{k-1}$, depending on their $t$-coordinate, and then mapped by $\alpha$ to $Y$.  Note that some of these (open dot) points---at least those at either end of their row---will also be mapped to $y_0$.  For a fixed $t$-coordinate, open points are aggregated in groups that start after each line and continue up to the next line, including that point. The lines themselves are level sets of the homotopy $H$.  The formulas $H(0, t) = H(kM + k -1, t) = y_0$, $H(s, 0) = \alpha\circ\rho_{k}(s)$, and   $H(s, M+1) = \alpha\circ\rho_{k-1}(s)$ follow directly.  It remains to confirm continuity of $H$.  Here, the potential issue is similar to those faced in \cite[Sec.4]{LOS19a} (and especially Th.4.9 of that paper).  Namely, it is not sufficient to have $H$ continuous in each of its variables separately---it is easily seen that this is the case here.  Nonetheless, $H$ is indeed continuous.  At the risk of appearing over-cautious, we give a careful argument for this point.  For suppose we have $(s, t) \sim (s', t')$ in $I_{kM + k-1} \times I_{M+1}$.  Now $|(s+t) - (s'+t')| \leq |s-s'| + |t-t'| \leq  2$. Hence either we have both $(s, t)$ and $(s', t')$ in the region of $I_{kM + k-1} \times I_{M+1}$ given by $s+t \in [0, kM+k-1]$ (this region is indicated as the open dots in  \figref{fig:homotopy alpha c}), or we have both $(s, t)$ and $(s', t')$ in the region of $I_{kM + k-1} \times I_{M+1}$ given by $s+t \in [kM+k-2, (k+1)(M+1) - 1]$ (this region is indicated as the black dots plus the first two diagonals of open dots  \figref{fig:homotopy alpha c}).  But in the latter region, we have  $H(s, t) = y_0$ for all points, except possibly at the single point $((k-1)(M+1) - 2, M+1)$ and (only) in the case in which $k=2$, when we would have $H((M+1) - 2, M+1) = \alpha\circ\rho_1(M-1) = \alpha(M-1)$.  Now $\alpha$ is continuous, and so we have $\alpha(M-1) \sim_Y \alpha(M) = y_0$.  So in this region, the only two values that $H(s, t)$ can take are adjacent in $Y$, and so $H(s, t) \sim_Y H(s', t')$.

Now consider the region in which we have $s+t \in [0, kM+k-1]$.  Within this region, there are three possibilities: (i) Both $(s, t)$ and $(s', t')$ also satisfy $0 \leq s \leq t(k-1) - 1$ (the region to the left and above the line of---non-adjacent---open dots at  the angles of the lines, in \figref{fig:homotopy alpha c}), including that line of dots; (ii) Both $(s, t)$ and $(s', t')$ also satisfy $t(k-1) \leq s \leq kM + k-1 - t$ (the complement of the region (i) in the open dots); (iii) one in either region. Notice these cases are the same constraints that separate the formulas defining $H$.

Case (i): We have $H(s, t) =  \alpha\circ\rho_{k-1}(s)$ and $H(s', t') =  \alpha\circ\rho_{k-1}(s')$.  But $|s-s'| \leq 1$, and hence $|\rho_{k-1}(s) - \rho_{k-1}(s')| \leq 1$ (even in the case in which $k=2$).  That is, we have
$\rho_{k-1}(s) \sim_{I_M} \rho_{k-1}(s')$ and, since $\alpha$ is continuous, we have $H(s, t) \sim_Y H(s', t')$.

Case (ii): We have $H(s, t) =  \alpha\circ\rho_{k}(s+t)$ and $H(s', t') =  \alpha\circ\rho_{k}(s'+t')$.  Since $|(s+t) - (s'+t')| \leq 2$, we have $|\rho_{k}(s+t) - \rho_{k}(s'+t')| \leq 1$ \emph{because $k \geq 2$}.  Hence, again, we have $H(s, t) \sim_Y H(s', t')$.

Case (iii): We have one of $(s, t)$ and $(s', t')$ in each of the two regions.  WLOG, say that we have $(s, t)$ that satisfies $0 \leq s \leq t(k-1) - 1$ and $(s', t')$ that satisfies
$t'(k-1) \leq s' \leq kM + k-1 - t'$.  Then $H(s, t) =  \alpha\circ\rho_{k-1}(s)$ and $H(s', t') =  \alpha\circ\rho_{k}(s'+t')$.  Our goal, similarly to the other sub-cases, is to argue that
we must have $|\rho_{k-1}(s) - \rho_{k}(s'+t')| \leq 1$ and then conclude using continuity of $\alpha$.

We treat $k=2$ separately.  Here, $(s, t)$ satisfies $s \leq t-1$, or $t \geq s+1$, and $(s', t')$ satisfies  $t' \leq s'$.   If $t \geq s+3$, then $t -s\geq 3$ and any neighbour $(s', t')$ must have  $t' -s' \geq 1$, or $t' \geq s'+1$, because adjacency implies ${ (t-s) - (t'-s')} \leq 2$. Hence we have  $t = s+2$ or $t = s+1$.  If $t = s+2$, then the only possible neighbour that satisfies $t' \leq s'$ is $(s', t') = (s+1, t-1)$.  But then we have $s'+t' = s+t = 2s+2$, whence $\rho_2(s'+t') = s+1$.  If $t = s+1$, then the  neighbours that satisfies $t' \leq s'$ are
$\{ (s+1, t), (s+1, t-1), (s, t-1)\}$.  Then $s'+t' \in \{ s+t +1, s+t, s+t -1\} = \{ 2s+2, 2s+1, 2s\} $, hence $\rho_2(s'+t') \in \{ s+1, s\}$.  Meanwhile, $\rho_1(s) = s$, and so if $k=2$, we have
$|\rho_{k-1}(s) - \rho_{k}(s'+t')| \leq 1$.

Now assume that $k \geq 3$. For $(s, t)$ that satisfies $s \leq t(k-1) - 1$, consider the possible locations of a neighbour $(s', t')$ that satisfies $s' \geq  t' (k-1)$.  The point $(s, t)$ lies on some line parallel to the line $s = t(k-1) - 1$, which has slope $1/(k-1) < 1$ (in the usual ``$t$ against $s$" sense).  Any point above this line also satisfies $s \leq t(k-1) - 1$, and it follows that the only points adjacent to $(s, t)$ whose coordinates could possibly satisfy $s' \geq  t' (k-1)$ are the four points $(s+1, t)$, $(s+1, t-1)$, $(s, t-1)$, and $(s-1, t-1)$ (this last would not be possible were $k = 2$).  In particular, note that we are constrained to have
$$-2 \leq (s'+t') - (s+t) \leq 1.$$

Furthermore, if $(s, t)$ has any neighbours whose coordinates satisfy $s' \geq  t' (k-1)$, then the ``lower-right" neighbour $(s+1, t-1)$ must be one such.  But then we have $(s+1) \geq (t-1)(k-1)$, which may be translated into the lower bound of the following restriction on the range of $s$:
\begin{equation}\label{eq: range of s}
t(k -1) - k \leq  s \leq  t(k -1) - 1.
\end{equation}
For those $(s, t)$ that satisfy the part of this range
$$t(k -1) - (k-1) \leq  s \leq  t(k -1) - 1,$$
we have
$$t-1 = \rho_{k-1}(t(k -1) - (k-1)) \leq  \rho_{k-1}(s) \leq  \rho_{k-1}(t(k -1) - 1) = t-1,$$
so $\rho_{k-1}(s) = t-1$.  On the other hand, this partial range for $s$ may be re-written as
$$
tk- (k-1) \leq  s +t \leq  tk - 1,
$$
whence we have
$$tk -(k-1) - 2 \leq  s'+t' \leq  tk,$$
since $(s'+t') - (s+t) \in \{ -2, -1, 0, 1\}$.  In turn, this leads to $t-2 = \rho_k\big( (t-1)k - 1) \leq  \rho_k(s'+t') \leq \rho_k( tk) = t$, and so  $ \rho_k(s'+t') \in \{t-2, t-1, t\}$.  For all these $(s, t)$, then, we have $|\rho_{k-1}(s) - \rho_{k}(s'+t')| \leq 1$.  The only part of \eqref{eq: range of s} remaining is $s = t(k-1) - k = (t-1)(k-1)-1$.  Here we have  $\rho_{k-1}(s) = t-2$.  On the other hand, if  $s = t(k-1) - k$, then $s+t = (t-1)k$ and so $(t-1)k - 2 \leq s'+t' \leq (t-1)k + 1$, again because we have  $(s'+t') - (s+t) \in \{ -2, -1, 0, 1\}$.   Then we have
$t-2 = \rho_k\big( (t-1)k - 2\big) \leq \rho_k(s'+t') \leq \rho_k\big((t-1)k + 1\big) = t-1$.  Here also, then, we have $|\rho_{k-1}(s) - \rho_{k}(s'+t')| \leq 1$.

For the full range \eqref{eq: range of s} of possibilities for adjacent $(s, t)$ and $(s', t')$, we have $|\rho_{k-1}(s) - \rho_{k}(s'+t')| \leq 1$.  Therefore, we have
$\alpha\circ\rho_{k-1}(s) \sim_Y  \alpha\circ\rho_{k}(s'+t')$, since $\alpha$ is continuous, and the homotopy $H$ is continuous in case (iii) also.

This completes the proof of the first assertion of part (c).  The second is proved likewise, using an adaptation of the homotopy $H$.  We omit the details.
\end{proof}

\begin{proposition}\label{prop: product on classes}
Let $Y \subseteq \Z^n$ be any based digital image.
Suppose we have two pairs of subdivision-based homotopic based loops:
$\alpha\colon I_M \to Y$ and $\alpha'\colon I_{M'} \to Y$ with $[\alpha] = [\alpha']$;  and $\beta\colon I_N \to Y$ and $\beta'\colon I_{N'} \to Y$
with $[\beta] = [\beta']$.
Then   we have  $[\alpha\cdot\beta] = [\alpha'\cdot\beta']$.
\end{proposition}

\begin{proof}
First, we show that for any based loop $\alpha\colon I_M \to Y$, and any $N \geq 0$, we have $[C_N\cdot\alpha] = [\alpha]$.  In other words, $C_N\cdot\alpha\colon I_{N+M+1} \to Y$ and $\alpha$ are subdivision-based homotopic.  Recall that $C_0\colon I_0 \to Y$ denotes the constant map $\{0\} \mapsto y_0$.  It is also the case that $[\alpha\cdot C_N] = [\alpha]$.  The proof of this is entirely similar to the proof we give here, and we omit it.  These facts basically show that the subdivision-based  homotopy class of a constant loop plays the role of the identity, but we will prove this separately in the next result.

Write the constant loop $C_N \colon I_N \to Y$ as
$$C_N = C_0\cdot C_0 \cdot \overset{ (N+1)\text{-times} }{\cdots\ \ \  \cdots} \cdot C_0.$$
Then we have
$$
\begin{aligned}
(C_N\cdot\alpha)\circ\rho_{M+1} &= (C_0\circ\rho_{M+1})\cdot (C_0\circ\rho_{M+1}) \cdot \overset{ (N+1)\text{-times} }{\cdots\ \ \  \cdots} \cdot (C_0\circ\rho_{M+1}) \cdot (\alpha\circ\rho_{M+1})\\
&=C_M \cdot C_M \cdot \overset{ (N+1)\text{-times} }{\cdots\ \ \  \cdots} \cdot C_M \cdot (\alpha\circ\rho_{M+1}),
\end{aligned}
$$
from repeated use of part (b) of \lemref{lem: loop product basics} and the identification $C_0\circ\rho_{M+1} = C_M\colon S(I_0, M+1) = I_M \to Y$.
Then repeated use of part (c) of \lemref{lem: loop product basics} gives
$$
\begin{aligned}
{[C_N\cdot\alpha]} & = [(C_N\cdot\alpha)\circ\rho_{M+1}] = [C_M \cdot C_M \cdot \overset{ (N+1)\text{-times} }{\cdots\ \ \  \cdots} \cdot C_M \cdot (\alpha\circ\rho_{M+1})] \\
&= [C_M \cdot C_M \cdot \overset{ (N)\text{-times} }{\cdots} \cdot C_M \cdot (\alpha\circ\rho_{M+2})] \\
&= [C_M \cdot C_M \cdot \overset{ (N-1)\text{-times} }{\cdots\ \ \  \cdots} \cdot C_M \cdot (\alpha\circ\rho_{M+3})] \\
&= \cdots = [\alpha\circ\rho_{M+N+2}] = [\alpha].
\end{aligned}
$$

Now consider the various loops as in the hypotheses.  Since $[\alpha] = [\alpha']$ and $[\beta] = [\beta']$, we have  based homotopies of based loops  $\alpha\circ \rho_k \approx \alpha'\circ\rho_{k'}$ and $\beta\circ \rho_l \approx \beta'\circ\rho_{l'}$, for suitable $k, k', l, l'$ as in \defref{def: subdn homotopic loops}.  Using parts (a) and (b) of  \lemref{lem: loop product basics} we have
$$
\begin{aligned}
{[\alpha\cdot\beta]} &= [(\alpha\cdot\beta)\circ\rho_{kl}]  = [(\alpha\circ\rho_{kl}) \cdot (\beta\circ\rho_{kl})]\\
&= [\big((\alpha\circ\rho_{k})\circ\rho_{l}\big) \cdot \big( (\beta\circ\rho_{l})\circ\rho_{k}\big)]\\
&=  [\big((\alpha'\circ\rho_{k'})\circ\rho_{l}\big) \cdot \big( (\beta'\circ\rho_{l'})\circ\rho_{k}\big)]\\
&=  [(\alpha'\circ\rho_{k'l}) \cdot (\beta'\circ\rho_{l'k})].
\end{aligned}
$$
In the above steps, we used the second point of \lemref{lem: composition based homotopy loops}---in the form $\alpha\circ\rho_{k} \approx \alpha'\circ\rho_{k'}$ implies that $(\alpha\circ\rho_{k})\circ\rho_{l} \approx (\alpha'\circ\rho_{k'})\circ\rho_{l}$---to apply parts (a) of  \lemref{lem: loop product basics}, and likewise for the $\beta$ terms.

If $k'l = l'k$, then we may write the last term above as $[(\alpha' \cdot\beta')\circ\rho_{k'l}] = [\alpha' \cdot\beta']$, which is the desired conclusion.

Now suppose that, instead, we have $k'l > l'k$.  Write $k'l = l'k + r$ for some $r$. Then by repeatedly using the second statement in part (c) of  \lemref{lem: loop product basics}, we have
$$\alpha'\circ\rho_{k'l} \approx C_{M'} \cdot (\alpha'\circ\rho_{k'l-1})  \approx C_{M'} \cdot C_{M'} \cdot  (\alpha'\circ\rho_{k'l-2}) \approx \cdots,$$
which results in $\alpha'\circ\rho_{k'l} \approx C_{rM'+r-1} \cdot (\alpha'\circ\rho_{l'k})$, where we have identified
$$C_{M'} \cdot C_{M'} \cdot  \overset{ (r)\text{-times} }{\cdots} \cdot C_{M'} = C_{rM'+r-1}$$
and $\rho_{k'l-r} = \rho_{l'k}$.  A further application of part (a)  of  \lemref{lem: loop product basics} yields
$$[\alpha\cdot\beta] =  [C_{rM'+r-1} \cdot \big( (\alpha'\circ\rho_{l'k})  \cdot (\beta'\circ\rho_{l'k}) \big)] = [ (\alpha'\circ\rho_{l'k})  \cdot (\beta'\circ\rho_{l'k}) ],$$
with the last equality following from the first part of this proof.  Now part (b) of  \lemref{lem: loop product basics} gives
$$[\alpha\cdot\beta] = [ (\alpha' \cdot \beta')\circ\rho_{l'k}] = [ \alpha' \cdot \beta'],$$
which again is the desired conclusion.  The case in which $k'l < l'k$ is proved in an entirely similar way, using instead the first statement in part (c) of  \lemref{lem: loop product basics} together with the identity $[\alpha\cdot C_N] = [\alpha]$ mentioned at the start of this proof.  We omit the details.
\end{proof}

\propref{prop: product on classes} means that the set of subdivision-based homotopy equivalence classes of based loops inherits a well-defined product, defined by
\begin{equation}\label{eq: product of classes}
[\alpha]\cdot[\beta] = [\alpha\cdot\beta]
\end{equation}
for based loops $\alpha\colon I_M \to Y$ and $\beta\colon I_N \to Y$.    For $Y\subseteq \Z^n$ a based digital image, denote the set of subdivision-based homotopy equivalence classes of based loops in $Y$ by $\pi_1(Y, y_0)$.

Before the main result, we give a final technical result that will establish inverses for our fundamental group.  We state this result for general paths, as we have need for the more general statement in the sequel.

\begin{definition}\label{def: inverse path}
For any path $\gamma\colon I_M \to Y$, let $\overline{\gamma} \colon I_M \to Y$ denote the \emph{reverse path}
$\overline{\gamma}(t) = \gamma(M-t)$.  If $\alpha$ is a based loop in $Y$, then so too is its reverse $\overline{\alpha}$.
\end{definition}

\begin{lemma}\label{lem: inverses}
Suppose $\gamma\colon I_M \to Y$ is a path in $Y$.  Then the concatenations of paths $\gamma\cdot \overline{\gamma}, \overline{\gamma} \cdot \gamma\colon I_{2M+1} \to Y$,
as in   \defref{def: concatenation}, are loops in $Y$, based at $\gamma(0)$ and $\gamma(M)$ respectively.
We have a homotopy \emph{relative the endpoints}, which is to say that the homotopy $H$ satisfies $H(0, t) = \gamma(0) = H(2M+1)$ for all $t$,
$$\gamma\cdot \overline{\gamma} \approx C^{\gamma(0)}_{2M+1} \colon I_{2M+1} \to Y,$$
where $C^{\gamma(0)}_{2M+1} \colon I_{2M+1} \to Y$ denotes the constant loop at $\gamma(0) \in Y$.  Likewise, we have a homotopy relative the endpoints
$\overline{\gamma} \cdot \gamma\approx C^{\gamma(M)}_{2M+1}$.

Furthermore, for the short concatenations, as in  \eqref{eq: short concat} of \defref{def: concatenation}, we also have homotopies relative the endpoints
$\gamma*\overline{\gamma} \approx C^{\gamma(0)}_{2M}$ and $\overline{\gamma} * \gamma\approx C^{\gamma(M)}_{2M}$.
\end{lemma}

\begin{proof}
All assertions are justified with the same homotopy.  The main point is to check its continuity.  We begin with the homotopy $\gamma\cdot \overline{\gamma} \approx C^{\gamma(0)}_{2M+1}$.  Define a function $H \colon I_{2M+1}\times I_M \to Y$ by
$$H(s, t) = \begin{cases}
\gamma(s) & 0 \leq s \leq M-t\\
\gamma(M-t) & M-t \leq s \leq M+1+t\\
\overline{\gamma}\big(s-(M+1)\big) = \gamma(2M+1-s)& M+1+t  \leq s \leq 2M+1.\end{cases}
$$
This begins at
$$H(s, 0) = \begin{cases}
\gamma(s) & 0 \leq s \leq M\\
\overline{\gamma}\big(s-(M+1)\big)& M+1  \leq s \leq 2M+1,\end{cases}
$$
which is  the concatenation $\gamma\cdot \overline{\gamma}$.  It ends at the constant loop $H(s, M) = \gamma(0)$, and satisfies $H(0, t) =  \gamma(0)$ and $H(0, M) =  \overline{\gamma}(M) = \gamma(0)$, and so is the desired homotopy assuming continuity.

To check that $H$ is continuous,  divide $I_{2M+1} \times I_{M}$ into two overlapping regions:  Region $A$, consisting of those points $(s, t)$  that satisfy both $s+t \geq M$ and $s-t \leq M+1$; and region $B$, consisting of the points that satisfy either $s+t \leq M+1$ or $s-t \geq M$.  Now any pair of adjacent points $(s, t)\sim (s', t')$ in $I_{2M+1} \times I_{M}$ satisfies $|(s-t) - (s'-t')| \leq 2$ and $|(s+t) - (s'+t')| \leq 2$, and hence either both lie in region $A$ or both lie in region $B$.  Suppose first that both lie in region $A$.  From the formula for $H$, in this case we have $H(s, t) =  \gamma(M-t)$ and $H(s', t') = \gamma(M-t')$.  If $(s, t)\sim (s', t')$, then we have $|t' - t| \leq 1$, hence $|(M-t') - (M-t)| \leq 1$, and continuity of $\gamma$ then gives
$H(s', t') \sim_Y H(s, t)$.  On the other hand, suppose adjacent points $(s, t) \sim (s', t')$ both lie in region $B$.
First suppose that $s+t \leq M+1$ and $s'+t' \leq M+1$ (the left-hand half of region $B$).  Here we have
$$H(s, t) = \begin{cases} \gamma(s) & s +t \leq M\\
\gamma(M-t) & s +t = M+1.\end{cases}$$
If both $(s, t)$ and $(s', t')$ satisfy $s+t \leq M$ and $s'+t' \leq M$, then we have $H(s, t) =  \gamma(s)$ and $H(s', t') = \gamma(s')$, and  continuity of $\gamma$ gives
$H(s', t') \sim_Y H(s, t)$.  If both $(s, t)$ and $(s', t')$ satisfy $s+t = M+1$ and $s'+t' = M+1$, then we have $H(s, t) =  \gamma(M-t)$ and $H(s', t') = \gamma(M-t')$, and again continuity of $\gamma$ gives $H(s', t') \sim_Y H(s, t)$.  Suppose  $(s, t)$ satisfies $s+t \leq M$, so that $H(s, t) = \gamma(s)$,  and  $(s', t')$ satisfies $s'+t' = M+1$ so that $H(s', t') = \gamma(M - t') = \gamma(s'-1)$.  This is only possible if we also have $s' \geq  s$ which, assuming we have $(s, t)\sim (s', t')$, means that we have $s' = s$ or $s' = s+1$.  Then $\gamma(s) \sim \gamma(s'-1)$ in either case, so we have $H(s, t)  \sim H(s', t')$ here.  For the other remaining possibillty, when  $s+t = M+1$ and $s'+t' \leq  M$, we find that $H(s, t)  \sim H(s', t')$ follows, interchanging the roles of $(s, t)$ and $s', t')$ in the last few steps.  It remains to confirm that $H$ preserves adjacency for adjacent points $(s, t)\sim (s', t')$  that satisfy  $s-t \geq M+1$ and $s'-t' \geq M+1$ (the right-hand half of region $B$).  The same argument, \emph{mutatis mutandis}, as we used for the left-hand half of region $B$ will confirm $H(s, t)  \sim H(s', t')$ here also.  We omit the details.  This completes the check of continuity for $H$ on  $I_{2M+1}\times I_M$, and with it the proof of the first assertion.

The second assertion now follows by interchanging the roles of $\gamma$ and $\overline{\gamma}$, with  the observation that the reverse of $\overline{\gamma}$ is $\gamma$.

For the short concatenations, we modify the homotopy to
$$H' \colon I_{2M}\times I_M \to Y$$
defined by
$$H'(s, t) = \begin{cases}
\gamma(s) & 0 \leq s \leq M-t\\
\gamma(M-t) & M-t \leq s \leq M+t\\
\overline{\gamma}(s-M) = \gamma(2M-s)& M+t  \leq s \leq 2M.\end{cases}
$$
Continuity of $H'$ follows from that of $H$.  This is because we have $H' = H\colon [0, M]\times I_M \to Y$, hence $H'$ is continuous here.  Also, if we denote by $T\colon [M, 2M] \to [M+1, 2M+1]$ the translation $T(s) = s+1$, which is evidently continuous, then $T \times \mathrm{id}_{I_M} \colon [M, 2M]\times I_M \to [M+1, 2M+1]\times I_M$ is continuous (cf.~\lemref{lem: map product conts}).  Therefore
$H' = H\circ (T \times \mathrm{id}_{I_M})\colon [M, 2M]\times I_M \to Y$ is continuous on $[M, 2M]\times I_M$ also.   Now any two adjacent points $(s, t) \sim (s', t')$ in $[0, 2M]\times I_M$ must either lie both in   $[0, M]\times I_M$ or both in $[M, 2M]\times I_M$.  Since the restriction of $H'$ to either of these sub-rectangles is continuous, it follows  that $H'$ is continuous.  A direct check now confirms that $H'$ is a homotopy, relative the endpoints, from $\gamma*\overline{\gamma}$ to the constant map $C^{\gamma(0)}_{2M}$.  Just as in the first part, interchanging the roles of $\gamma$ and $\overline{\gamma}$ is sufficient to conclude for $\overline{\gamma}*\gamma \approx C^{\gamma(M)}_{2M}$.
\end{proof}

\begin{theorem}[Digital Fundamental Group]\label{thm: pi1}
The set of subdivision-based homotopy equivalence classes of based loops in a digital image $Y$, $\pi_1(Y, y_0)$, with the product \eqref{eq: product of classes}, is a group.
\end{theorem}

\begin{proof}
The product \eqref{eq: product of classes} is associative because the product of based loops is associative.   For any $N, M\geq 0$, we have
$$C_M\circ \rho_{N+1} = C_N\circ \rho_{M+1} = C_{MN + M +N} \colon I_{MN + M +N} \to Y.$$
Thus, $C_M \colon I_M \to Y$ and $C_N \colon I_N \to Y$ are subdivision-based homotopy equivalent, and we have $[C_M] = [C_N] \in \pi_1(Y, y_0)$.  Write this class as $\mathbf{e}$.  For any based loop $\alpha: I_M \to Y$, we have $[\alpha] = [\alpha\circ \rho_2]$, by definition.  Now part (c) of   \lemref{lem: loop product basics} gives
$$[\alpha] = [\alpha\circ\rho_2] = [\alpha\cdot C_M] = [\alpha]\cdot \mathbf{e}.$$
Likewise, we have $[\alpha] =  \mathbf{e}\cdot [\alpha]$, and thus $\mathbf{e} \in \pi_1(Y, y_0)$ is a two-sided identity element.

For any based loop $\alpha\colon I_M \to Y$, the reverse loop $\overline{\alpha} \colon I_M \to Y$ acts as the inverse.
From \lemref{lem: inverses} we have, in $\pi_1(Y, y_0)$,
$$[\alpha]\cdot [\overline{\alpha}] =  [ \alpha\cdot \overline{\alpha}] = [C_{2M+1}] = \mathbf{e},$$
and
$$[\overline{\alpha}]\cdot [\alpha] =  [ \overline{\alpha} \cdot \alpha] = [C_{2M+1}] = \mathbf{e}.$$
Thus, $[\overline{\alpha}]$ is a two-sided inverse element of $[\alpha] \in \pi_1(Y, y_0)$.
\end{proof}

Induced homomorphisms now follow just as in the development of these ideas in the topological setting.  Suppose we have a based map $f\colon X \to Y$ of digital images $X \subseteq \Z^m$ and $Y \subseteq \Z^n$, with basepoints $x_0 \in X$ and $y_0 \in Y$ so that $f(x_0) = y_0$.  For $\alpha \colon I_M \to X$ any based loop  in $X$, the composition $f\circ \alpha\colon I_M \to Y$ is a based loop in $Y$.  Say we have $[\alpha] = [\alpha'] \in \pi_1(X; x_0)$, for another based loop $\alpha' \colon I_N \to X$, so that there is some based homotopy of based loops
$$H \colon I_{kM+k-1} \times I_T \to X$$
from $\alpha\circ \rho_k$ to $\alpha'\circ\rho_l$, with $l$ such that $l(M +1) = l(N+1)$.  Then $f\circ H \colon I_{kM+k-1} \times I_T \to Y$ gives $[f\circ\alpha] = [f\circ\alpha'] \in \pi_1(Y; y_0)$.    So we may set $f_*([\alpha]) = [f\circ\alpha]$ to define a map
\begin{equation}\label{eq: induced hom}
f_* \colon \pi_1(X; x_0) \to \pi_1(Y; y_0).
\end{equation}

\begin{lemma}\label{lem: induced hom}
Let $f \colon X \to Y$ be any based map of based digital images $X \subseteq \Z^m$ and $Y \subseteq \Z^n$.  Then \eqref{eq: induced hom} is a homomorphism of groups.
If $g \colon Y \to Z$ is another based map, then we have $(g\circ f)_* = g_* \circ f_*$.  Furthermore, if $f \approx g\colon X \to Y$ are based-homotopic maps, then the homomorphisms $f_*$ and $g_*$ agree. Namely, we have $f_*(x) = g_*(x)$ for all $x \in \pi_1(X; x_0)$.
\end{lemma}

\begin{proof}
The first assertion follows directly from the definitions.  Indeed, we may repeatedly re-write $f_*([\alpha]\cdot[\beta])$ as
$$f_*([\alpha \cdot \beta]) =  [f\circ(\alpha \cdot \beta)] = [(f\circ \alpha) \cdot (f\circ \beta)] = [f\circ \alpha] \cdot [f\circ \beta] = f_*([\alpha])\cdot f_*([\beta]).$$
The second point is more-or-less tautological.
For the third point, suppose that $H \colon X \times I_N \to Y$ is a based homotopy from $f$ to $g$.  For any based loop $\alpha\colon I_M \to X$, the composition
$$H\circ (\alpha \times \mathrm{id}_{I_N}) \colon I_M \times I_N \to Y$$
is a based homotopy of based loops from $f\circ\alpha$ to $g\circ \alpha$.  Notice that  $\alpha \times \mathrm{id}_{I_N}\colon I_M \times I_N \to X \times I_N$ is continuous, by the observation of
\lemref{lem: map product conts}.  Thus we have $[f\circ\alpha] = [g\circ \alpha] \in \pi_1(Y; y_0)$.  The result follows.
\end{proof}

We prove a more general version of \lemref{lem: induced hom} in \thmref{thm: subdn htpic sam hom} below.  But at this point, we may draw the following conclusions.  We say that a based digital image $X$ is \emph{based-contractible} if there is a based homotopy
$$\mathrm{id}_X \approx C_{x_0}\colon X \to X$$
from the identity map of $X$ to the constant map of $X$ at $x_0$.

\begin{corollary}\label{cor: contractible pi1}
If $X$ is based-contractible, then we have $\pi_1(X; x_0) \cong \{ \mathbf{e} \}$, the trivial group.
\end{corollary}

\begin{proof}
Let $\alpha\colon I_M \to X$ be any based loop in $X$.  Then   $C_{x_0}\circ \alpha = C_M \colon I_M \to X$, and so we have $(C_{x_0})_*([\alpha]) = [C_M] = \mathbf{e}$.  Thus, the constant map induces the trivial endomorphism of $\pi_1(X; x_0)$.  The identity map, meanwhile, induces the identity endomorphsim.  If $X$ is based-contractible, then \lemref{lem: induced hom} implies these two endomorphisms agree, and it follows  that  $\pi_1(X; x_0)$ must be trivial.
\end{proof}

\begin{example}\label{ex: contractible cube}
Any interval $I_M \subseteq \Z$, more generally any $n$-cube $(I_M)^n \subseteq \Z^n$, is based-contractible.
Indeed,  the homotopy $H\colon I_M \times I_M \to I_M$ defined by
$$H(s, t) = \begin{cases} s & 0 \leq s \leq M - t\\ M-t & M - t < s \leq M\end{cases}$$
begins at  the identity $\text{id}_{I_M}$ and ends at  the constant map  $C_0$.  Furthermore, we have
$H(0, t) = 0$ for all $t=0, \ldots, M$, so the homotopy is based.
Note that  continuity of this contracting homotopy does not follow from the argument of  \cite[Ex.2.9]{Bo06a}, where the same homotopy is used to show contractibility but using the notion of homotopy that flows from the graph product.  Cf.~\remref{rem: graph product homotopy} above.  Nonetheless, this homotopy is still continuous in our sense, as a careful check reveals.  We omit details of this check.

For the $n$-cube, we may assemble a contracting homotopy using this homotopy in each coordinate.  We omit details of this, since we may conclude triviality of the fundamental group more generally from our result on products below.
\end{example}

Independence of the choice of basepoint follows exactly as in the topological setting.

\begin{theorem}\label{thm: indep basept}
Let $Y$ be any digital image.  Suppose that a path $\gamma\colon I_N \to Y$ has $\gamma(0) = y'_0$ and $\gamma(N) = y_0$.  We have an isomorphism of fundamental groups
$$\Phi \colon  \pi_1(Y; y'_0) \to \pi_1(Y; y_0),$$
defined by setting $\Phi([\alpha]) = [  \overline{\gamma}\cdot \alpha\cdot \gamma ]$  for each $\alpha\colon I_M \to Y$ with $\alpha(0) = y'_0= \alpha(M)$.
\end{theorem}

\begin{proof}
We check that $\Phi$ is well-defined.  The ingredients we need for this are implicit in parts of \lemref{lem: loop product basics}, \propref{prop: product on classes} and \lemref{lem: inverses}, and the first part of the proof of \thmref{thm: pi1}.  We make them explicit here.

Suppose that we are given based loops $\alpha\colon I_M \to Y$ and $\beta\colon I_L \to Y$, both based at $y'_0$ and with
$[\alpha] = [\beta] \in \pi_1(Y; y'_0)$.  That is, we have $\alpha\circ \rho_m \approx \beta\circ \rho_l \colon I_{mM+m-1} \to Y$ for some  $m$ and $l$ with $mM+m-1 = lL+l-1$.  Then
for the loop $\overline{\gamma}\cdot \alpha \cdot \gamma$ based at $y_0$, we have
$$[\overline{\gamma}\cdot \alpha \cdot \gamma] = [(\overline{\gamma}\cdot \alpha \cdot \gamma)\circ\rho_m]
= [(\overline{\gamma}\circ\rho_m)\cdot (\alpha\circ\rho_m) \cdot (\gamma\circ\rho_m)],$$
with the second equality following from part (b) of \lemref{lem: loop product basics}.  Actually, that result is phrased for concatenations of based loops, but the argument given clearly applies just as well to concatenations of paths (in all situations in which concatenation is suitable).   We wish to replace $(\alpha\circ\rho_m)$ with $(\beta\circ \rho_l)$ in the the middle of this concatenation.
This is a variation on part (a) of \lemref{lem: loop product basics}, with the difference being that here we are concatenating paths and loops, not just loops.
Suppose that $H\colon I_{mM+m-1}\times I_T \to Y$ is a based homotopy of based loops $\alpha\circ \rho_m \approx \beta\circ \rho_l$.  Then define $H_{-}\colon I_{mN+m-1} \times I_T \to Y$ and $H_{+}\colon I_{mN+m-1} \times I_T \to Y$ as
$$H_{-}(s, t) =  (\overline{\gamma}\circ\rho_m)(s) \quad \text{ and } \quad H_{+}(s, t) =  (\gamma\circ\rho_m)(s)$$
for all $t \in I_T$.   Then $H$, $H_{-}$ and $H_{+}$ assemble together to give a based homotopy of based loops
$$\mathcal{H} \colon I_{ m(2N+M) + 3m-1} \times I_T \to Y,$$
from $(\overline{\gamma}\circ\rho_m)\cdot (\alpha\circ\rho_m) \cdot (\gamma\circ\rho_m)$ to $(\overline{\gamma}\circ\rho_m)\cdot (\beta\circ\rho_l) \cdot (\gamma\circ\rho_m)$. Explicitly, we have
$$\mathcal{H}(s, t) = \begin{cases}
H_{-}(s, t) & 0 \leq s \leq mN+m-1\\
H\circ T_1(s, t) & mN+m \leq s \leq m(N+M) + 2m -1 \\
H_{+}\circ T_2(s, t) & m(N+M) + 2m \leq s \leq m(2N+M) + 3m-1,
\end{cases}
$$
 with $T_1$ and $T_2$ the evidently continuous translations $T_1(s, t) = (s - mN+m, t)$ and $T_2(s, t) = (s - (m(N+M) + 2m), t)$.
 Individually, $H_{-}$ is continuous on $[0, mN+m-1] \times I_T$ and $H\circ T_1$ is continuous on $[mN+m, m(N+M) + 2m-1] \times I_T$.  But we have
 $$\mathcal{H}(mN+m-1, t) = (\overline{\gamma}\circ\rho_m)(mN+m-1) = \overline{\gamma}(N) = y'_0,$$
and
 $$\mathcal{H}(mN+m, t) = H(0, t)   = y'_0,$$
for each $t \in I_T$.  Thus, for any two adjacent points $(s, t)$ and $(s', t')$ in $[0, m(N+N) + 2m-1]\times I_T$ that do not lie either both in $[0, mN+m-1] \times I_T$ or both in $[mN+m, m(N+M) + 2m-1] \times I_T$, we have $\mathcal{H}(s, t) = \mathcal{H}(s', t') = y'_0$.  Hence, $\mathcal{H}$ is continuous over  $[0, m(N+M) + 2m-1]\times I_T$.  A similar discussion applied to $H$ and $H_{+}$ confirms that, indeed, $\mathcal{H}$ is continuous over  $[0, m(2N+M) + 3m-1]\times I_T$.   Furthermore, we have
 $$\mathcal{H}(0, t) = (\overline{\gamma}\circ\rho_m)(0) = \overline{\gamma}(0) = y_0,$$
and
 $$\mathcal{H}(m(2N+M) + 3m-1, t) = H_{+}(mN+m-1, t) = (\gamma\circ\rho_m)(mN+m-1) = \gamma(N)  = y_0,$$
for each $t \in I_T$.  So $\mathcal{H}$ is a based homotopy of based loops
$$(\overline{\gamma}\circ\rho_m)\cdot (\alpha\circ\rho_m) \cdot (\gamma\circ\rho_m) \approx (\overline{\gamma}\circ\rho_m)\cdot (\beta\circ \rho_l) \cdot (\gamma\circ\rho_m).$$
Continuing with our calculation from above, then, we may now write
\begin{equation}\label{eq: alpha gamma beta}
[\overline{\gamma}\cdot \alpha \cdot \gamma] =  [(\overline{\gamma}\circ\rho_m)\cdot (\beta\circ \rho_l) \cdot (\gamma\circ\rho_m)].
\end{equation}
If $m=l$, then a further application of part (b) of \lemref{lem: loop product basics} (for concatenations of paths) yields
$$[\overline{\gamma}\cdot \alpha \cdot \gamma] =  [(\overline{\gamma}\cdot \beta \cdot \gamma)\circ\rho_m] =  [\overline{\gamma}\cdot \beta \cdot \gamma],$$
which is to say that we have  $\Phi([\alpha]) = \Phi([\beta])$.

So suppose, instead, that we have $l = m+r$, with $r \geq 1$.  Write $C_N \colon I_N \to Y$ for the constant loop of length $N$ at $y_0 \in Y$.  Also, write $(C_N)^{\cdot r}$ for  the $r$-fold concatenation of this constant  loop with itself.  We need a variant of part (c) of \lemref{lem: loop product basics} for paths.  Namely, we want a homotopy relative the endpoints
$$\gamma\circ\rho_{k} \approx (\gamma\circ\rho_{k-1})\cdot C_N  \colon I_{kN + k-1} \to Y$$
for each $k \geq 1$.  But, once again, the same argument used for  part (c) of \lemref{lem: loop product basics} applies just as well to paths: the fact that $\alpha$ in that proof is a loop rather than a path plays no role in the argument.  So the homotopy used there, namely,
$$H(s, t) = \begin{cases} \gamma\circ\rho_{k-1}(s)& 0 \leq s \leq t(k-1) - 1\\
\gamma\circ\rho_{k}(s+t)& t(k-1) \leq s \leq kN + k-1 - t\\
y_0 & kN + k - t \leq s \leq kN + k -1
\end{cases}
$$
is a homotopy $H \colon I_{kN + k-1} \times I_{N+1} \to Y$ from $\gamma\circ\rho_{k}$ to $(\gamma\circ\rho_{k-1})\cdot C_N$ that is relative the endpoints, because we have
$$H(0, t) = \begin{cases} \gamma\circ\rho_{k}(0) = \gamma(0) = y'_0 & t=0\\
\gamma\circ\rho_{k-1}(0) = \gamma(0) = y'_0 & 1 \leq t \leq N+1
\end{cases}
$$
and
$$H(kN + k-1, t) = \begin{cases} \gamma\circ\rho_{k}(kN + k-1) = \gamma(N) = y_0 & t=0\\
y_0& 1 \leq t \leq N+1.
\end{cases}
$$
This homotopy is continuous for the same reasons as in the proof of  part (c) of \lemref{lem: loop product basics}---again, the fact that $\alpha$ is a loop there plays no role in the continuity part of the argument.  By repeatedly using this variant of part (c) of \lemref{lem: loop product basics}, we obtain a homotopy relative the endpoints
$$H\colon I_{(m+r)N+m+r-1} \times I_{N+1} \to Y$$
from $\gamma\circ \rho_{m+r}$ to    $(\gamma\circ \rho_{m}) \cdot (C_N)^{\cdot r}$.  A similar discussion, using a variant of part (c) of \lemref{lem: loop product basics} for a homotopy  $\gamma\circ\rho_{k}$ to $C_N \cdot (\overline{\gamma}\circ\rho_{k-1})$ leads to a homotopy relative the endpoints
$$H'\colon I_{(m+r)N+m+r-1} \times I_{N+1} \to Y$$
from $\overline{\gamma}\circ \rho_{m+r}$ to  $(C_N)^{\cdot r}  \cdot (\overline{\gamma}\circ \rho_{m})$.  Now, because the homotopies $H$ and $H'$ are relative their endpoints, we may assemble them together to give a based homotopy of based loops
$$(\overline{\gamma}\circ \rho_{m+r}) \cdot  (\beta\circ \rho_l) \cdot  (\gamma\circ\rho_{m+r})
\approx
(C_N)^{\cdot r}  \cdot (\overline{\gamma}\circ \rho_{m}) \cdot  (\beta\circ \rho_l) \cdot  (\gamma\circ\rho_m) \cdot (C_N)^{\cdot r}.$$
In fact, we may use $G\colon I_{2(m+r)N+2m+2r + lL+l -1}\times I_{N+1} \to Y$ defined by
$$G(s, t) = \begin{cases}
H'(s, t) & 0 \leq s \leq (m+r)(N+1)-1\\
\\
& (m+r)(N+1)\\
(\beta\circ \rho_l)(s-((m+r)(N+1))) & \hbox{\hskip0.3truein} \leq s \leq \\
&\hbox{\hskip.6truein}  (m+r)(N+1) + lL+l-1 \\
\\
 & (m+r)(N+1) + lL+l-1\\
H\circ T(s, t) & \hbox{\hskip0.3truein} \leq s \leq \\
& \hbox{\hskip0.6truein}(m+r)(2N+2) + lL+l -1,
\end{cases}
$$
 with $T$ the evidently continuous translation $T(s, t) = (s - ((m+r)N+m+r + lL+l-1), t)$.  A direct check confirms that this satisfies $G(s, 0) = (\overline{\gamma}\circ \rho_{m+r}) \cdot  (\beta\circ \rho_l) \cdot  (\gamma\circ\rho_{m+r}) (s)$ and $G(s, N+1) = (C_N)^{\cdot r}  \cdot (\overline{\gamma}\circ \rho_{m}) \cdot  (\beta\circ \rho_l) \cdot  (\gamma\circ\rho_m) \cdot (C_N)^{\cdot r}(s)$.  Also, we have $G(0, t) = G(2(m+r)N+2m+2r + lL+l -1, t) = y_0$.  Continuity of $G$  follows because the defining formulas are continuous on their domains, and where their domains abut, namely at $\{  mN+m-1, mN+m\}\times I_{N+1}$ and at   $\{  (m+r)N+m+r + lL+l-1, (m+r)N+m+r + lL+l\}\times I_{N+1}$, $G$ takes the single value $y'_0$.
Any pair of adjacent points in the domain of $G$ is either in one of the separate domains of the defining formulas or, if not, in one of these two vertical strips on which $G$ takes a constant value. Hence we have $G(s, t) \sim_Y G(s', t')$, and $G$ is continuous.

So, return to \eqref{eq: alpha gamma beta}, and recall that we are considering the case in which we have $l = m+r$ with $r \geq 1$.  Then after the above discussion, we may continue with \eqref{eq: alpha gamma beta} and write
$$
\begin{aligned}
{[\overline{\gamma}\cdot \alpha \cdot \gamma] }&=  [(\overline{\gamma}\circ\rho_m)\cdot (\beta\circ \rho_{m+r}) \cdot (\gamma\circ\rho_m)]\\
&=  [(C_N)^{\cdot r}  \cdot(\overline{\gamma}\circ\rho_m)\cdot (\beta\circ \rho_{m+r}) \cdot (\gamma\circ\rho_m)\cdot (C_N)^{\cdot r}]\\
&=  [(\overline{\gamma}\circ\rho_{m+r})\cdot (\beta\circ \rho_{m+r}) \cdot (\gamma\circ\rho_{m+r})].
\end{aligned}
$$
The second equality here is a direct application of  the facts established at the start of the proof of \propref{prop: product on classes}:
For any based loop $\alpha\colon I_M \to Y$, and any $N \geq 0$, we have $[C_N\cdot\alpha] = [\alpha] = [\alpha\cdot C_N]$.  Notice that, for this step, we do actually have concatenations of based loops.  Then the third equality above follows from the discussion above.  But now, just as we did for the case in which $l = m$ above, we may make
a further application of part (b) of \lemref{lem: loop product basics} (for concatenations of paths) to obtain
$$[(\overline{\gamma}\circ\rho_{m+r})\cdot (\beta\circ \rho_{m+r}) \cdot (\gamma\circ\rho_{m+r})] =  [(\overline{\gamma}\cdot \beta \cdot \gamma)\circ\rho_{m+r}] =  [\overline{\gamma}\cdot \beta \cdot \gamma],$$
and in this case also we have  $\Phi([\alpha]) = \Phi([\beta])$.

Finally, referring to \eqref{eq: alpha gamma beta}, suppose that we have $m = l+r$, with $r \geq 1$.   From the previous case, we have homotopies relative the endpoints
$\gamma\circ \rho_{l+r}\approx(\gamma\circ \rho_{l}) \cdot (C_N)^{\cdot r}$ and $\overline{\gamma}\circ \rho_{l+r}$ to  $(C_N)^{\cdot r}  \cdot (\overline{\gamma}\circ \rho_{l})$.  Just as in the previous case, we incorporate these into a based homotopy of based loops
$$(\overline{\gamma}\circ \rho_{l+r}) \cdot  (\beta\circ \rho_l) \cdot  (\gamma\circ\rho_{l+r})
\approx
(C_N)^{\cdot r}  \cdot (\overline{\gamma}\circ \rho_{l}) \cdot  (\beta\circ \rho_l) \cdot  (\gamma\circ\rho_l) \cdot (C_N)^{\cdot r}.$$
Then we have $\Phi([\alpha]) = \Phi([\beta])$ here also, again just as in the previous case.

So we have shown that $\Phi$ is well-defined.  That $\Phi$ is a homomorphism follows easily.  As in the first part of the proof of \thmref{thm: pi1}, with $\mathbf{e} = [C_{2N+1}] \in \pi_1(Y; y'_0)$, we may write $[\alpha]\cdot[\beta] =  [\alpha]\cdot \mathbf{e} \cdot [\beta] =  [\alpha]\cdot [\gamma\cdot \overline{\gamma}] \cdot [\beta] = [\alpha\cdot (\gamma\cdot \overline{\gamma}) \cdot \beta ] \in  \pi_1(Y; y'_0)$, with the last identities following from  \lemref{lem: inverses} and \propref{prop: product on classes} (directly, for based loops).  Thus we have
$$
\begin{aligned}
\Phi([\alpha]\cdot[\beta]) & = \Phi( [\alpha\cdot (\gamma\cdot \overline{\gamma}) \cdot \beta ])\\
&= [\overline{\gamma} \cdot \alpha\cdot \gamma\cdot \overline{\gamma} \cdot \beta \cdot \gamma]\\
&= [(\overline{\gamma} \cdot \alpha\cdot \gamma)\cdot (\overline{\gamma} \cdot \beta \cdot \gamma)]\\
&= [(\overline{\gamma} \cdot \alpha\cdot \gamma)]\cdot [(\overline{\gamma} \cdot \beta \cdot \gamma)]\\
&= \Phi([\alpha])\cdot \Phi([\beta]),\end{aligned}
$$
with all identities after the first being identities of equivalence classes of based loops in $\pi_1(Y; y_0)$.

Since $\gamma$ was an arbitrary path in $Y$ from $y'_0$ to $y_0$, we may replace it with the reverse path $\overline{\gamma}$ from $y_0$ to $y'_0$ and obtain a homomorphism
$$\Phi' \colon  \pi_1(Y; y_0) \to \pi_1(Y; y'_0),$$
defined by setting $\Phi'([\alpha]) = [  \gamma \cdot \alpha\cdot \overline{\gamma} ]$.  For each $[\alpha] \in  \pi_1(Y; y'_0)$, we have
$$
\begin{aligned}
\Phi'\circ \Phi ([\alpha]) & = [  \gamma \cdot (\overline{\gamma} \cdot \alpha \cdot  \gamma)\cdot \overline{\gamma}]\\
& = [ ( \gamma \cdot \overline{\gamma} )\cdot \alpha \cdot  (\gamma \cdot \overline{\gamma})] = [  \gamma \cdot \overline{\gamma}]\cdot [\alpha] \cdot [\gamma \cdot \overline{\gamma} ] \\
& = \mathbf{e}\cdot [\alpha] \cdot \mathbf{e} = [\alpha],
\end{aligned}
$$
again from  \lemref{lem: inverses} and \propref{prop: product on classes} directly.  Likewise, we have  $\Phi\circ \Phi' = \mathrm{id} \colon  \pi_1(Y; y_0) \to \pi_1(Y; y_0)$, and so $\Phi$ and $\Phi'$ are inverse homomorphisms.  It follows that each is an isomorphism.
\end{proof}

 Another standard result from the topological setting that translates well into this digital setting concerns products.  See \lemref{lem: product} for some of the notation and terminology used here.

 \begin{theorem}\label{thm: products}
Let $X$ and $Y$ be any digital images.  Let $p_1\colon X \times Y \to X$ and $p_2\colon X \times Y \to Y$ denote the projections onto either factor.  Define a map
$$\Psi \colon  \pi_1\big(X \times Y; (x_0, y_0)\big) \to \pi_1(X; x_0)\times \pi_1(Y; y_0),$$
by setting $\Psi([\alpha]) = \big( (p_1)_*([\alpha]), (p_2)_*([\alpha])\big)$ for each $[\alpha] \in  \pi_1\big(X \times Y; (x_0, y_0)\big)$.
Then $\Psi$ is an isomorphism.
\end{theorem}

\begin{proof}
First note that $\Psi$ is well-defined, and a homomorphism, because both $(p_1)_*$ and $(p_2)_*$ are (see \lemref{lem: induced hom} and the discussion above it).
Suppose we have $([\alpha], [\beta]) \in \pi_1(X; x_0)\times \pi_1(Y; y_0)$,  with $[\alpha]$ represented by a based loop $\alpha\colon I_M \to X$ and $[\beta]$ represented by a based loop $\beta\colon I_N \to Y$.  Then $(\alpha, C_M) \colon I_M \to X \times Y$ and $(C_N, \beta) \colon I_N \to X \times Y$ represent elements of $\pi_1\big(X \times Y; (x_0, y_0)\big)$.  We have
$$\Psi \big( [(\alpha, C_M)] \cdot [(C_N, \beta)]\big) = \Psi \big( [(\alpha, C_M)] \big) \Psi \big( [(C_N, \beta)]\big) = ([\alpha], \mathbf{e}) (\mathbf{e}, [\beta]) =
([\alpha], [\beta]).$$
Hence $\Psi$ is onto.

Suppose we have $[\alpha] \in \pi_1\big(X \times Y; (x_0, y_0)\big)$ represented by a based loop $\alpha \colon I_M \to X \times Y$, such that $\Psi([\alpha]) = (\mathbf{e}, \mathbf{e}) \in \pi_1(X; x_0)\times \pi_1(Y; y_0)$.  Then $p_1\circ\alpha$ and $p_2\circ\alpha$ are subdivision-based homotopic, as based loops, to constant loops.  Suppose we have
$p_1\circ\alpha\circ \rho_k \approx C_{kM+k-1}$ via a homotopy of based loops $H \colon I_{kM+k-1} \times I_T \to X$  and $p_2\circ\alpha\circ\rho_l \approx C_{lM+l-1}$ via a homotopy of based loops $G \colon I_{lM+l-1} \times I_S \to Y$.  Then $H\circ (\rho_l \times \mathrm{id}_{I_T}) \colon I_{klM+kl-1} \times I_T \to X$ is a based homotopy of based loops $p_1\circ\alpha\circ \rho_{kl} \approx C_{kM+k-1}\circ \rho_l = C_{klM+kl-1}$ and $G\circ (\rho_k \times \mathrm{id}_{I_S}) \colon I_{klM+kl-1} \times I_S \to Y$ is a based homotopy of based loops $p_2\circ\alpha\circ \rho_{kl} \approx C_{lM+l-1}\circ \rho_k = C_{klM+kl-1}$.  If $S \not= T$, then we may lengthen the shorter of these two homotopies, exactly as we did at the start of our proof of part (a) of \lemref{lem: loop product basics}, so that they are of equal length.  So assume, without loss of generality,  that we have $S=T$.
Then we have a based homotopy of based loops
$$\big( H\circ (\rho_l \times \mathrm{id}_{I_T}),  G\circ (\rho_k \times \mathrm{id}_{I_T}) \big) \colon  I_{klM+kl-1} \times I_T \to X \times Y$$
from $\big( p_1\circ\alpha\circ \rho_{kl}, p_2\circ\alpha\circ \rho_{kl}\big)$ to $\big( C_{klM+kl-1}, C_{klM+kl-1}\big)$.  Because we may write
$$\alpha\circ \rho_{kl} = \big( p_1\circ\alpha\circ \rho_{kl}, p_2\circ\alpha\circ \rho_{kl}\big) \colon I_{klM+kl-1} \to X \times Y$$
and
$$C_{klM+kl-1} = \big( C_{klM+kl-1}, C_{klM+kl-1}\big) \colon I_{klM+kl-1} \to X \times Y,$$
it follows that we have $[\alpha] = \mathbf{e} \in \pi_1\big(X \times Y; (x_0, y_0)\big)$.  Hence $\Psi$ is injective.  The result follows.
\end{proof}

Thus far, we have succeeded in  implementing the standard development of ideas but from a point of view that emphasizes the use of subdivision.
With our last few results of this section, we illustrate the advantage of building subdivision into our constructions.  In these results, we go significantly beyond what has been shown for other versions of the fundamental group already in the digital literature.

\begin{theorem}\label{thm: rho induces iso on pi1}
Each standard projection $\rho_k \colon S(X, k) \to X$ induces an isomorphism of fundamental groups.
\end{theorem}

\begin{proof}
Suppose we have $[\alpha] \in \pi_1(X; x_0)$ represented by a based loop $\alpha\colon I_M \to X$.  First consider an odd subdivison $\rho_{2k+1} \colon S(X, 2k+1) \to X$.
From \thmref{thm: path odd subdivision map}, we have a based loop $\widehat{\alpha} \colon S(I_M, 2k+1) \to S(X, 2k+1)$ that satisfies $\rho_{2k+1}\circ \widehat{\alpha} = \alpha\circ \rho_{2k+1}$.   Then in $\pi_1(X; x_0)$, we have
$$[\alpha] = [\alpha\circ\rho_{2k+1}] =  [\rho_{2k+1}\circ \widehat{\alpha}] = (\rho_{2k+1})_*([\widehat{\alpha}]),$$
with $[\widehat{\alpha}] \in \pi_1( S(X, 2k+1) ; \overline{x_0})$. Hence the homomorphism $(\rho_{2k+1})_* \colon \pi_1( S(X, 2k+1) ; \overline{x_0}) \to \pi_1(X; x_0)$ is onto for an odd subdivision.  Now suppose we have an even subdivision $\rho_{2k} \colon S(X, 2k) \to X$.  Factor the standard projection $\rho_{2k+1} \colon S(X, 2k+1) \to X$ as
$$\rho_{2k+1} = \rho_{2k}\circ \rho^c_{2k+1} \colon S(X, 2k+1) \to S(X, 2k) \to X,$$
where $\rho^c_{2k+1} \colon S(X, 2k+1) \to S(X, 2k)$ denotes the ``partial projection" defined in \defref{def: rho^c}.  As discussed in \defref{def:subdn basepoints}, $\rho^c_{2k+1}$ is a based map.
Then we have the following commutative diagram:
$$\xymatrix{ S(I_M, 2k+1) \ar[r]^-{\widehat{\alpha}} \ar[dd]_{\rho_{2k+1}} & S(X, 2k+1) \ar[dd]^-{\rho_{2k+1}} \ar[rd]^{\rho^c_{2k+1}} \\
&  & S(X, 2k) \ar[ld]^{\rho_{2k}}\\
I_M \ar[r]_-{\alpha}   & X}$$
Here, the based loop $\widehat{\alpha}$ is given by \thmref{thm: path odd subdivision map}.  Then $\rho^c_{2k+1}\circ \widehat{\alpha}$ is a based loop in $S(X, 2k)$ and the class in
$\pi_1( S(X, 2k) ; \overline{x_0})$ that it represents satisfies
$$(\rho_{2k})_*([\rho^c_{2k+1}\circ \widehat{\alpha}]) =  (\rho_{2k}\circ \rho^c_{2k+1})_*([\widehat{\alpha}])  = (\rho_{2k+1})_*([\widehat{\alpha}])
= [\alpha\circ\rho_{2k+1}] = [\alpha].$$
For even subdivisions the homomorphism $(\rho_{2k})_* \colon \pi_1( S(X, 2k) ; \overline{x_0}) \to \pi_1(X; x_0)$ is also onto.  For any subdivision, then, we have shown that the standard projection $\rho_k \colon S(X, k) \to X$ induces a surjection of fundamental groups.

To show that $(\rho_k)_*$ is injective requires somewhat more argument, most of which we defer to \appref{sec: technical}.  First, notice that it is sufficient to show this for $k$ odd.  For suppose that we have $\alpha\colon I_M \to S(X, 2k)$, with $k \geq 1$,  representing $[\alpha] \in \pi_1( S(X, 2k); \overline{x_0})$.   The based loop
$$\widehat{\alpha} \colon S(I_M, 3) = I_{3M+2} \to S( S(X, 2k), 3) = S(X, 6k)$$
that covers $\alpha$ from \thmref{thm: path odd subdivision map} satisfies $\alpha\circ\rho_3 = \rho_3\circ \widehat{\alpha}  \colon I_{3M+2} \to S(X, 2k)$, and thus we have
$$[\alpha] = [\alpha\circ\rho_3] = [\rho_3\circ \widehat{\alpha} ] \in \pi_1( S(X, 2k) ; \overline{x_0}).$$
Now we may factor $\rho_3\colon S(X, 6k) \to S(X, 2k)$ as
$$\rho_3 = \rho\circ \rho^c_{6k} \colon S(X, 6k) \to S(X, 6k-1) \to S(X, 2k),$$
by repeatedly using \defref{def: rho^c}, where $\rho^c_{6k} \colon S(X, 6k) \to S(X, 6k-1)$ is one of the partial projections defined there, and
$$\rho= \rho^c_{2k+1}\circ\cdots \circ \rho^c_{6k-1}\colon  S(X, 6k-1) \to S(X, 6k-2) \to \cdots \to S(X, 2k)$$
is a composition of such. As discussed in \defref{def:subdn basepoints}, these partial projections, and hence $\rho$, are all based maps.  The constructions thus far are represented in the following commutative diagram:
$$\xymatrix{I_{3M+2} \ar[r]^-{\widehat{\alpha}} \ar[dd]_{\rho_3} & S(X, 6k) \ar[rd]^{\rho^c_{6k}} \ar[dd]^{\rho_3}\\
 & & S(X, 6k-1) \ar[ld]^{\rho}\\
 I_M \ar[r]_-{\alpha} & S(X, 2k) \ar[d]^{\rho_{2k}}\\
  & X.}$$
Continuing with the above identifications, we have
$$[\alpha] =  [\alpha\circ\rho_3] = [ \rho\circ\rho^c_{6k}\circ \widehat{\alpha}] = \rho_*( [\rho^c_{6k}\circ \widehat{\alpha}]) \in \pi_1( S(X, 2k) ; \overline{x_0}).$$
Now assume that $(\rho_{6k-1})_*\colon \pi_1( S(X, 6k-1) ; \overline{x_0}) \to \pi_1( X ; x_0)$ is injective, and that we have
$(\rho_{2k})_*([\alpha]) = \mathbf{e} \in \pi_1(X, x_0)$.
We may write
$$(\rho_{6k-1})_*( [\rho^c_{6k}\circ \widehat{\alpha}] ) = (\rho_{2k})_*\circ \rho_*( [\rho^c_{6k}\circ \widehat{\alpha}] ) = (\rho_{2k})_*( [\alpha] ) = \mathbf{e},$$
since $\rho_{6k-1} = \rho_{2k}\circ\rho\colon S(X, 6k-1) \to S(X, 2k) \to X$.  But with our assumption of injectivity of $(\rho_{6k-1})_*$, this gives
$[\rho^c_{6k}\circ \widehat{\alpha}]  = \mathbf{e} \in  \pi_1( S(X, 6k-1) ; \overline{x_0})$, whence we have $[\alpha] =  \rho_*( \mathbf{e}) \in \pi_1( S(X, 2k) ; \overline{x_0})$.
Thus, injectivity of $(\rho_{2k})_*$ follows from that of $(\rho_{6k-1})_*$. It suffices to show $(\rho_{2k+1})_* \colon \pi_1( S(X, 2k+1) ; \overline{x_0}) \to \pi_1(X; x_0)$ is injective for each odd $2k+1$.

So suppose that we have $x \in \pi_1( S(X, 2k+1); \overline{x_0})$ with $(\rho_{2k+1})_*(x) = \mathbf{e} \in \pi_1( X; x_0)$.  Then $x$ is  represented by a based loop $\alpha \colon I_M \to S(X, 2k+1)$.  Since $[\rho_{2k+1} \circ \alpha] = \mathbf{e}$, we have a based homotopy of based loops $\rho_{2k+1} \circ \alpha\circ \rho_{k'} \approx C_K \colon I_K \to X$, with $K = k'M + k'-1$, for some suitable $k'$.
Let $\widehat{\big(\rho_{2k+1} \circ \alpha\circ \rho_{k'}\big)}$ be the standard cover of $\rho_{2k+1} \circ \alpha\circ \rho_{k'}\colon I_K \to X$ (cf.~\thmref{thm: path odd subdivision map}).
Then \corref{cor: homotopy alpha-rho cover} (with $\alpha$ replaced by $\alpha\circ \rho_{k'}$) implies that in the following diagram, the top-left triangle commutes up to a based homotopy of based loops:
$$\xymatrix{ S(I_K, 2k+1) \ar[rr]^-{ \widehat{\big(\rho_{2k+1} \circ \alpha\circ \rho_{k'}\big)} }  \ar[d]_{\rho_{2k+1}} & &S(X, 2k+1) \ar[d]^{\rho_{2k+1}} \\
I_K \ar[rru]_{\alpha\circ \rho_{k'}} \ar[rr]_-{\rho_{2k+1} \circ \alpha\circ \rho_{k'}} & & X.}$$
Since $\rho_{2k+1} \circ \alpha\circ \rho_{k'} \approx C_K$, it follows from \thmref{thm: 2-D subdivision map rectangle} that
we have a based homotopy of based loops $\widehat{\big(\rho_{2k+1} \circ \alpha\circ \rho_{k'}\big)}  \approx C_{(2k+1)K+2k} \colon I_{(2k+1)K+2k} \to S(X, 2k+1)$.
Hence we may write
$$x = [\alpha] = [\alpha \circ \rho_{k'(2k+1)}] = [(\alpha \circ \rho_{k'})\circ \rho_{2k+1}] =  [\widehat{\big(\rho_{2k+1} \circ \alpha\circ \rho_{k'}\big)} ] = \mathbf{e}.$$
Thus  $(\rho_{2k+1})_* \colon \pi_1( S(X, 2k+1) ; \overline{x_0}) \to \pi_1(X; x_0)$ is injective.  As observed above, this is sufficient to conclude that every
standard projection $\rho_k$ induces an injection of fundamental groups.  The result follows.
\end{proof}

We postpone the proof of the technical result used to establish injectivity of $(\rho_{2k+1})_*$ in the above proof.  It appears as   \corref{cor: homotopy alpha-rho cover}
in Appendix~\ref{sec: technical}

\begin{corollary}
Each partial projection $\rho^c_k\colon S(X, k) \to S(X, k-1)$ (from \defref{def: rho^c}) induces an isomorphism of fundamental groups.
\end{corollary}

\begin{proof}
We may factor the projection $\rho_k\colon S(X, k) \to X$ as
$$\rho_k = \rho_{k-1}\circ \rho^c_k\colon S(X, k) \to S(X, k-1) \to X,$$
with $\rho_{k-1}\colon S(X, k-1) \to X$ the standard projection and $\rho^c_k\colon S(X, k) \to S(X, k-1)$ the partial projection.  Then both $(\rho_{k})_*$ and
$(\rho_{k-1})_*$ are isomorphisms, from \thmref{thm: rho induces iso on pi1}, and hence $(\rho^c_{k})_* = \big((\rho_{k-1})_*\big)^{-1}\circ (\rho_{k})_*$ is an isomorphism.
\end{proof}

In \lemref{lem: induced hom}, we observed that based-homotopic maps induce the same homomorphism of fundamental groups.   A consequence of
 \thmref{thm: rho induces iso on pi1} is that maps that are subdivision-based homotopic also induce ``essentially" the same homomorphism of fundamental groups.
In the following result, the case in which $k = l = 1$ yields the statement that subdivision-based homotopic maps $f, g\colon X \to Y$ induce the same homomorphism of fundamental groups.

\begin{theorem}\label{thm: subdn htpic sam hom}
Suppose we have subdivision-based homotopic maps
$f \colon S(X, k) \to Y$ and $g \colon S(X, l) \to Y$.  If $k = l$, then we have
$$f_*  = g_* \colon \pi_1( S(X, k); \overline{x_0} ) \to \pi_1(Y; y_0).$$
If $k>l$, respectively $l > k$, then we have $f_*  = g_*\circ \rho_* \colon \pi_1( S(X, k); \overline{x_0} ) \to \pi_1(Y; y_0)$, respectively $f_*\circ \rho_*  = g_* \colon \pi_1( S(X, l); \overline{x_0} ) \to \pi_1(Y; y_0)$, with $\rho_*$ a canonical isomorphism induced by a partial projection $\rho\colon S(X, k) \to S(X, l)$, respectively $\rho\colon S(X, l) \to S(X, k)$.
\end{theorem}

\begin{proof}
As in \defref{def: subdn homotopic maps}, we have a diagram
$$\xymatrix{ S(X, m) \ar[r]^-{\rho_{l'}} \ar[d]_{\rho_{k'}} & S(X, l) \ar[d]^{g}\\
S(X, k) \ar[r]_-{f} & Y}$$
that commutes up to based homotopy, with $m = kk'=ll'$.  If $k = l$, then $k' = l'$ and we have $f_*\circ (\rho_{k'})_* = g_*\circ (\rho_{k'})_*$ as homomorphisms, by \lemref{lem: induced hom}.  But $(\rho_{k'})_*$ is an isomorphism, by \thmref{thm: rho induces iso on pi1}, so we may cancel to obtain $f_* = g_*$.  Suppose that $k >l$.  Use the partial projections of
\defref{def: rho^c} to write $\rho_{l'} = \rho\circ\rho_{k'}\colon S(X, m) \to S(X, k) \to S(X, l)$, where
$$\rho = \rho^c_{l+1}\circ \cdots \circ \rho^c_{k}\colon S(X, k) \to S(X, k-1) \to \cdots \to S(X, l).$$
Here, then, we have a diagram
$$\xymatrix{ S(X, m) \ar[r]^-{\rho_{k'}} \ar[dd]_{\rho_{k'}} & S(X, k) \ar[d]^-{\rho} \\
 & S(X, l) \ar[d]^{g}\\
S(X, k) \ar[r]_-{f}   & Y}$$
that commutes up to based homotopy, and it follows as in the first part that we have $f_*  = g_*\circ \rho_*$.
Since $(\rho_{l'})_* = \rho_*\circ(\rho_{k'})_*$, and both $(\rho_{l'})_*$ and $(\rho_{k'})_*$ are isomorphisms by \thmref{thm: rho induces iso on pi1}, then so too is $\rho_*$ (and each of the homomorphisms $(\rho^c_j)_*$ induced by the partial projections) an isomorphism.

The case in which $l>k$ is handled likewise.  We omit the details.
\end{proof}

It follows easily from \lemref{lem: induced hom} that the fundamental group is preserved by based homotopy equivalence (Cf.~\defref{def: based h.e.} for the definition).  However, this notion of equivalence is so rigid that it is hard to make effective use of the fact.  By involving subdivision, we now obtain a result that says the fundamental group is preserved by a much more relaxed version of ``sameness" than homotopy equivalence.   Refer to \defref{def: subdn homotopy equiv} for the notion of subdivision-based homotopy equivalence.

\begin{theorem}\label{thm: subdn htpy equiv iso pi1}
If $X$ and $Y$ are subdivision-based homotopy equivalent, then they have isomorphic fundamental groups.
\end{theorem}

\begin{proof}
Refer to the data from \defref{def: subdn homotopy equiv}.  We have a map $f\colon S(X, k) \to Y$ for some $k$ that satisfies $f\circ G = \rho_{kl}$ for a certain auxiliary map $G$.  Then $f_*\circ G_* = (\rho_{kl})_*$ and, since $(\rho_{kl})_*$ is an isomorphism, it follows that $f_*$ is onto.  We also assume a cover $F$ of $f$, that satisfies $g\circ F = \rho_{kl}$ for some map $g$, which implies that $F_*$ is injective.  However, as a cover of $f$, we have $\rho_l\circ F =  f\circ \rho_l$.  Since each $(\rho_l)_*$ here is an isomorphism, this relation implies $F_*$ is injective exactly when $f_*$ is injective.  Thus $f_*$ is both onto and injective:  $f\colon S(X, k) \to Y$ induces an isomorphism of fundamental groups.  Then
$$f_*\circ \big( (\rho_k)_*\big)^{-1} \colon \pi_1(X, x_0) \to  \pi_1(S(X, k), \overline{x_0}) \to \pi_1(Y, y_0)$$
is an isomorphism.
\end{proof}

We may deduce the following more general version of \corref{cor: contractible pi1}.   We say that a based digital image $X\subseteq \Z^m$ is \emph{subdivision-based contractible} if $X$ is subdivision-based homotopy equivalent to a point, namely to some $\{y_0\} \subseteq \Z^n$ for a singleton point $y_0 \in \Z^n$.

\begin{corollary}\label{cor: subd contr pi1 triv}
If $X$ is subdivision-based contractible, then $\pi_1(X, x_0) \cong \{\mathbf{e}\}$.
\end{corollary}

\begin{proof}
This follows directly from \thmref{thm: subdn htpy equiv iso pi1}.  In fact, for a subdivision-based contractible $X$,
\defref{def: subdn homotopy equiv} may be shown to be equivalent to the existence, for some $K$, of
a based homotopy
$$\rho_K \approx C_{x_0}\colon S(X, K) \to X$$
from some standard projection to the constant map at $x_0$.   So the conclusion could equally well be obtained from \thmref{thm: subdn htpic sam hom} and \thmref{thm: rho induces iso on pi1}.
\end{proof}

In our final main result, we give a digital version of the calculation of the fundamental group of a circle.   Recall from \exref{ex:basic digital images} that $D$ denotes our prototypical digital circle.   For our calculation, we use a notion of \emph{winding number} for loops in $D$ that we introduced in \cite{LOS19a}.    We note, again, that in  the framework of \cite{Bo99}, the Diamond $D$ is contractible
and thus \cite{Bo99} would have  $\pi_1(D) = \{\mathbf{e}\}$ (cf.~\remref{rem: graph product homotopy}).   Thus our calculation here marks an essential difference between our approach and that used previously in the literature.

\begin{theorem}\label{thm: pi1D}
With $D = \{ (1, 0), (0, 1), (-1, 0), (0, -1) \} \subseteq \Z^2$ the digital circle of \exref{ex:basic digital images}, we have $\pi_1\big(D; (1, 0)\big) \cong \Z$.
\end{theorem}

\begin{proof}
In Section 6 of \cite{LOS19a}, we establish a winding number for loops in $D$.  The notion may be applied to based loops in $D$ as follows.  Suppose $\alpha\colon I_M \to D$ is a based loop.  For any choice of initial point $n_0 \in \Z$, there is a unique path $\overline{\alpha} \colon I_M \to [n_0-M, n_0+M]\subseteq \Z$ that starts at $n_0$ and for which the following diagram commutes
$$\xymatrix{
\{0\} \ar[r]^-{f} \ar[d]_-{i}  & [n_0-M, n_0+M] \ar[d]^-{p}\\
I_N \ar[r]_-{\alpha}  \ar[ru]^-{\overline{\alpha}}  & D.    } $$
Here, $p$ denotes the (restriction of the) projection $p \colon \Z \to D$
defined by $p(n) =  e^{n\pi i/2}$ (we identify a complex number $a + ib$ with the point $(a, b)$ in $\R^2$), and $f(0) = n_0$ is simply the choice of initial point for the lift.  We show this in Lemma 6.1 of \cite{LOS19a}, and in Corollary 6.2 of the same, we conclude that each based loop $\alpha$ has a winding number $w(\alpha) \in \Z$, which is independent of the choice of initial point of the lift, and which may be defined by $w(\alpha) = \overline{\alpha}(M) - \overline{\alpha}(0)$.

First, we claim that this winding number is an invariant of a subdivision-based homotopy equivalence class of based loops in $D$.  That is, we claim that if $[\alpha] = [\beta] \in \pi_1\big(D; (1, 0) \big)$, then we have $w(\alpha) = w(\beta)$.   For suppose that $\alpha\colon I_M \to D$ and $\beta\colon I_N \to D$ are based loops, and that we have suitable $m, n$ and a based homotopy of based loops $\alpha\circ \rho_m \approx \beta\circ \rho_n\colon I_{mM+m-1} = I_{nN+n-1} \to D$.  Suppose $\overline{\alpha}\colon I_M \to [-M, M]$ is the unique lift of $\alpha$ that starts at $0 \in [-M, M] \subseteq \Z$.
Then $\overline{\alpha}\circ\rho_m \colon  I_{mM+m-1} \to [-M, M]$ is the unique lift of  $\alpha\circ \rho_m$ that starts at $0 \in \Z$.
Hence we have
$$
\begin{aligned}
w(\alpha) &= \alpha(M) - \alpha(0) = \overline{\alpha}\circ\rho_m(mM+m-1) - \overline{\alpha}\circ\rho_m(0) \\
&= \overline{ \alpha\circ\rho_m}(mM+m-1) - \overline{ \alpha\circ\rho_m}(0)=w(\alpha\circ\rho_m).
\end{aligned}$$
Similarly, we have $w(\beta) = w(\beta\circ\rho_n)$.  Now Proposition 6.4 of \cite{LOS19a} shows that homotopic loops in $D$ have the same winding number.  Therefore, we have
$$w(\alpha) = w(\alpha\circ\rho_m) = w(\beta\circ\rho_n) = w(\beta),$$
and the winding number is invariant on elements of $\pi_1\big(D; (1, 0)\big)$.

So---bearing in mind that the winding number $w(\alpha)$, as constructed in \cite{LOS19a}, is actually a multiple of $4$---define a function
$$h\colon \pi_1\big(D; (1, 0)\big) \to \Z$$
by setting $h([\alpha]) = (1/4) w(\alpha)$.  This is well-defined by the preceding paragraph.  Furthermore, $h$ is also a homomorphism.  For suppose we have elements $[\alpha], [\beta] \in \pi_1\big(D; (1, 0)\big)$  represented by based loops $\alpha \colon I_M \to D$ and $\beta\colon I_N \to D$.  Suppose we lift $\alpha$ to the (unique) path $\overline{\alpha}$  in $\Z$ that starts at $0$.  Then let $\overline{\beta} \colon I_N \to \Z$ be the (unique) path that lifts $\beta$ and starts at $\overline{\alpha}(M)$, so that we have $\overline{ \beta}(0) =   \overline{  \alpha}(M)$.  Then the concatenation of paths
$\overline{\alpha}\cdot \overline{\beta} \colon I_{M+N+1} \to \Z$ is the lift of the concatenation of loops $\alpha\cdot\beta\colon I_{M+N+1} \to D$ (the unique such lift  that starts at $0$).  Hence, we may write
$$
\begin{aligned}
h([\alpha]\cdot[\beta])  &= h([\alpha \cdot \beta]) = \frac{1}{4} \,w( \alpha \cdot \beta)  = \frac{1}{4} \big(  \overline{ \alpha \cdot \beta}(M+N+1) -   \overline{ \alpha \cdot \beta}(0) \big)\\
&= \frac{1}{4} \big(   \overline{ \beta}(N) -   \overline{  \alpha}(0) \big)     =  \frac{1}{4} \big(  \overline{ \beta}(N)  +\big( - \overline{ \beta}(0) +  \overline{  \alpha}(M) \big)   -    \overline{ \alpha}(0) \big)        \\
&= \frac{1}{4} \big(  \overline{ \beta}(N)   - \overline{ \beta}(0)\big)  +  \frac{1}{4} \big(   \overline{  \alpha}(M)    -    \overline{ \alpha}(0) \big)\\
& = \frac{1}{4} w( \alpha) +  \frac{1}{4} w( \beta) = h(\alpha) + h(\beta).
\end{aligned}$$

Finally, we show that $h$ is actually an isomorphism.  It is easily seen to be surjective.  Indeed, the based loop $\alpha_n\colon I_n \to D$ defined as
$$\alpha_n(t) = p(n) =  e^{n\pi i/2}$$
has, more-or-less tautologically, winding number $w(\alpha_n) = 4n$ for each $n \in \Z$.  Therefore, we have $h([\alpha_n]) = n$, and it follows that $h$ is surjective.

For injectivity of $h$, suppose we have a based loop $\alpha \colon I_M \to D$ for which $h([\alpha] = 0$.  Then let $\overline{\alpha}\colon I_M \to [-M, M] \subseteq \Z$ be the lift of $\alpha$ that starts at $\overline{\alpha}(0) = 0$.  Since $h([\alpha] = (1/4) (\overline{\alpha}(M) - 0) = 0$, we must have  $\overline{\alpha}(M) = 0$ as well, so that the lift $\overline{\alpha}$ is a based loop in $[-M, M]$. Now this interval is contractible, and furthermore is contractible to $\{0\}$ via a homotopy that keeps $\{0\}$ fixed.  Explicitly, we may define a contracting homotopy $H \colon [-M, M] \times I_M \to [-M, M]$ as
$$
H(s, t) = \begin{cases}  -M + t & -M \leq s \leq -M + t\\
s & -M + t \leq s \leq M - t\\
M-t & M-t \leq s \leq M.\end{cases}
$$
A direct check using the formulas confirms that we have $H(s. 0) = s$ and $H(s, M) = 0$ for $s \in [-M, M]$, and also $H(0, t) = 0$ for each $t \in I_M$.

Continuity of $H$ may be confirmed with an argument similar to that used in the proof of \lemref{lem: inverses}.
As there, we divide  $[-M, M] \times I_M$
into two regions:  Region $A$, consisting of those points $(s, t)$  that satisfy both $t-s \leq M$ and $t+s \leq M$; and region $B$, consisting of the points that satisfy either $t-s \geq M$ or $t+s \geq M$.  The shapes of these regions (triangular) are the same as the shapes of the regions in the proof of \lemref{lem: inverses}, although here they are reversed with respect to the $t$-coordinate and also the horizontal coordinate here has range $-M \leq s \leq M$ rather than $[0, 2M+1]$ as in \lemref{lem: inverses}.  The regions here do not overlap quite as much as those of the proof of \lemref{lem: inverses}, but the basic steps in confirming continuity of $H$ remain the same.     We indicate some of these details.
Given a pair of adjacent points $(s, t)\sim (s', t')$ in  region $A$ of $[-M, M] \times I_M$, the formula for $H$ gives us $H(s, t) =  s$ and $H(s', t') = s'$.  If $(s, t)\sim (s', t')$, then we have $|s' - s| \leq 1$, hence we have $H(s', t') \sim H(s, t)$.

On the other hand, suppose adjacent points $(s, t) \sim (s', t')$ both lie in region $B$.
First suppose that $t-s \geq M+1$ and $t'-s' \geq M$ (the left-hand half of region $B$).  Here we have
$H(s, t) = -M+t$ and $H(s', t') = -M+t'$, and so $|H(s', t') - H(s, t)| = |t' - t| \leq 1$ since $(s, t) \sim (s', t')$.  But if we have $t+s \geq M$ and $t'+s' \geq M$ (the right-hand half of region $B$), then we have
$H(s, t) = M-t$ and $H(s', t') = M-t'$, and again $|H(s', t') - H(s, t)| = |-t' + t| \leq 1$ since $(s, t) \sim (s', t')$.  So for adjacent points that lie either both in $A$ or both in $B$, we see that $H$ preserves adjacency.  It remains to check that $H$ also preserves adjacency when one of $(s, t)$ and $(s', t')$ lies in $A$ and the other in $B$. Although straightforward,  we omit these details on the grounds that they are similar to some of those in the proof of \lemref{lem: inverses}, and are a little long-winded to go through.  We proceed on  the basis that $H$ is indeed a (continuous) contracting homotopy.

Recall from above the definition of $H$ that $\overline{\alpha}\colon I_M \to [-M, M]$ is the  based loop in $[-M, M]$ that lifts $\alpha$.
Then setting $\mathcal{H} = p\circ H\circ (\overline{\alpha} \times \mathrm{id}_{I_M}) \colon I_M \times I_M \to I_M$ gives a (continuous) homotopy that starts at
$$\mathcal{H}(i, 0) = p\circ H\big( \overline{\alpha}(i), 0\big) =  p\circ \overline{\alpha}(i) = \alpha(i),$$
and ends at  $\mathcal{H}(i, M) = p\circ H\big( \overline{\alpha}(i), M\big) =  p(0) = (1, 0)$, for $i \in I_M$.  Furthermore, we have
$$\mathcal{H}(0, t) = p\circ H\big( \overline{\alpha}(0), t\big) =  p\circ H\big( 0, t\big) = p(0) = (1, 0),$$
as well as
$$\mathcal{H}(M, t) = p\circ H\big( \overline{\alpha}(M), t\big) =  p\circ H\big( 0, t\big) = p(0) = (1, 0),$$
for each $t \in I_M$.  So $\mathcal{H}$ is a based homotopy of based loops $\alpha \approx C_M\colon I_M \to D$.  We have $[\alpha] = \mathbf{e} \in  \pi_1\big(D; (1, 0)\big)$; this is sufficient to conclude that $h$ is injective, as we have already shown it to be a homomorphism.  The result follows.
\end{proof}

We finish the main body with a discussion of the two digital circles $D$ and $C$ from \exref{ex:basic digital images}.

\begin{example}\label{ex: pi1 D and C}
First, we confirm that  $D$ and $C$ from \exref{ex:basic digital images} are not (based) homotopy equivalent. As we observed in \exref{ex:basic digital images}, the issue is that we cannot wrap $D$ around $C$.  Formally, this plays out as follows.  Suppose we have any map $f \colon D \to C$.  Assume $f$ is based, so that we have $f(1, 0) = (2, 0)$.  Because we must have $f(0, 1) \sim f(1, 0)$ and $f(0, -1) \sim f(1, 0)$, and then also $f(1, 1) \sim f(0, 1)$, it follows that the image $f(D)$ is contained in the subset of $C$
$$U = \{ (x_1, x_2) \in C \mid x_1 \geq 0\} \subseteq C.$$
Now this subset $U$ is (based) contractible (in the rigid sense given above \corref{cor: contractible pi1}).  We may justify this assertion carefully as follows.  Define a ``coordinate-centring" function $\lambda \colon [-2, 2] \to [-2, 2]$ (see \eqref{eq: function C} below for some comments) by setting
$$\lambda(x) = \begin{cases} x+1 & x=-2, -1\\
0 & x=0\\
x-1 & x=1, 2.\end{cases}$$
Then, use this function to define a contracting homotopy $H \colon U \times I_2 \to U \subseteq C$ by
$$H\big( (x_1, x_2), t \big) = \begin{cases}  \big( \lambda^t(x_1), 2 - \lambda^t(x_1)\big) & x_2 \geq 0\\
  \big( \lambda^t(x_1), \lambda^t(x_1)-2\big) & x_2 \leq 0.\end{cases}$$
 This is a continuous function, because each of its coordinate functions is: each one is a function only of $x_1$, and is continuous in that variable because $\lambda$ is.   A direct check confirms that we have $H\big( (x_1, x_2), 0 \big) = (x_1, x_2)$ and $H\big( (x_1, x_2), 2 \big) = (2, 0)$ for $(x_1, x_2) \in U$.    Furthermore, we have
  $H\big( (2, 0), t \big) = (2, 0)$ for each $t \in I_2$, so  this is a based homotopy of based maps $\mathrm{incl.} \approx *\colon U \to C$, from the inclusion of $U$ in $C$ to the constant map $*$ of $U$ at $(2, 0)$.    Then $H \circ (f\times \mathrm{id}_{I_2}) \colon D \times I_2 \to  C$ is a based  homotopy  from $f$ to the constant map $*$ of $D$ at $(2, 0)$.  Now it follows that $f$ cannot have a left  inverse up to based homotopy:  If there were a based map $g \colon C \to D$ with $g\circ f \approx \mathrm{id}_D \colon D \to D$, it would follow that $D$ is contractible, which we know is not so by  the calculation of \thmref{thm: pi1D} together with  \corref{cor: contractible pi1} (and we also show this directly in \cite{LOS19a}).   In particular, $D$ and $C$ cannot be based-homotopy equivalent.

However,  $D$ and $C$ \emph{are} subdivision-based homotopy equivalent.  Actually, we have $S(D, 2)$ and $C$ based-homotopy equivalent, here, which is a particular way in which digital images may be subdivision-based homotopy equivalent. But we give data for $D$ and $C$ that matches  \defref{def: subdn homotopy equiv}.  First, we have maps $f\colon S(D, 2) \to C$ and $g\colon C \to D$, defined as
$$f(2, 0) = f(3, 0) = (2, 0), \quad f(2, 1) = f(3, 1) = (1, 1),$$
$$f(1, 2) = f(1, 3) = (0, 2), \quad f(0, 2) = f(0, 3) = (-1, 1),$$
$$f(-2, 1) = f(-1, 1) = (-2, 0), \quad f(-2, 0) = f(-1, 0) = (-1, -1),$$
$$f(0, -1) = f(0, -2) = (0, -2), \quad f(1, -1) = f(1, -2) = (1, -1),$$
and
$$g(2, 0) = g(1, 1) = (1, 0), \quad g(0, 2) = g(-1, 1) = (0, 1),$$
$$g(-2, 0) = g(-1, -1) = (-1, 0), \quad g(0, -2) = g(1, -1) = (0, -1).$$
In \defref{def: subdn homotopy equiv}, then, $D$ and $C$ play the roles of $X$ and $Y$, respectively, with $k = 2$ and $l=1$.  Then $F=f$, and the composition $g\circ F = g\circ f = \rho_2\colon S(D, 2) \to D$, which is subdivision-based homotopic to $\mathrm{id}_D$.  For the cover $G$ of $g$,  since any cover $G$ is acceptable in  \defref{def: subdn homotopy equiv}, define a map $g'\colon C \to S(D, 2)$ by
$$g'(2, 0) =  (2, 0), \quad g'(1, 1) = (2, 1), \quad g'(0, 2) = (1, 2)\quad g'(-1, 1) = (0, 2),$$
$$g'(-2, 0) =  (-1, 1), \quad g'(-1, 1) = (0, -1), \quad g'(0, -2) = (0, -1)\quad g'(1, -1) = (1, -1),$$
 and then set $G = g'\circ \rho_2\colon S(C, 2) \to S(D, 2)$ (in fact, this $g'$ is a based-homotopy inverse of $f$, giving the based-homotopy equivalence between $C$ and $S(D, 2)$ hinted at above). Checking continuity of $g'$ and the covering property of $G$ is direct. Finally, to complete the data of  \defref{def: subdn homotopy equiv}, we require the composition
 $f\circ G \colon S(C, 2) \to C$ to be subdivision-based homotopy equivalent to the identity $\mathrm{id}_Y$.  But $f\circ g' = \mathrm{id}_C \colon C \to C$, as is easily checked, and so we have  $f\circ G = f\circ g' \circ \rho_2 = \rho_2\colon S(C, 2) \to C$, so the data are complete.

In summary, by \thmref{thm: subdn htpy equiv iso pi1},  $D$ and $C$ have isomorphic fundamental groups (isomorphic to $\Z$, by \thmref{thm: pi1D}),   despite not being based-homotopy equivalent.
\end{example}

Although $D$ and $C$ are relatively simple examples of digital images, nonetheless we feel \exref{ex: pi1 D and C} represents significant progress. We have directly calculated the fundamental group in a single, basic case and then used a relation less rigid than that of (based) homotopy equivalence to deduce its value for a different digital image.  The general strategy represented here suggests that our approach is  more suitable for homotopy theory in the digital setting than other developments in the literature.

\section{Directions for Future Work}\label{sec: Future Work}

Some our main results here rely on the results of \cite{LOS19b} about subdivisions of maps.  Those results are for maps with 1D or 2D domains.  Whilst they suffice for our needs here, we are working to try and extend those results to higher-dimensional domains.  If successful, it should then be possible to develop digital higher homotopy groups, and establish basic properties of them, in much the same spirit as our fundamental group results here.

It would be interesting to develop ways of effectively computing the fundamental group. For example, is it possible to establish a Seifert-Van Kampen theorem in this setting?  Also, is it possible to identify certain structures a digital image may have that impact or determine its fundamental group (H-space, co-H-space, wedge, etc.)?

In \cite{LOS19a} and \cite{LOS19b} we have begun to develop versions of homotopy invariants (such as L-S category) that are much less rigid  than versions of the same that exist in the digital topology  literature.  It would be interesting, where possible, to make connections between these notions and the fundamental group comparable to those that exist in the topological setting.

\appendix

\section{A Technical Result}\label{sec: technical}

We prove the technical result \corref{cor: homotopy alpha-rho cover} that we relied upon to establish injectivity of $(\rho_{2k+1})_*$  in the proof of
 \thmref{thm: rho induces iso on pi1}.  The  results here should be viewed in that context.  Unfortunately, the proof is rather lengthy, with many technical details.  Furthermore, there is considerable notation introduced for the purpose of the proof (some of which is drawn from the following   \appref{Appx: technical} below).  Including all this material in the main body would interrupt the flow of ideas there, hence we have placed it here. 

Suppose $\alpha\colon I_M \to S(X, 2k+1)$ is a based loop in $S(X, 2k+1)$, for  $X \subseteq \Z^n$ any based digital image.  Denote the basepoint of $X$ by $\mathbf{x_0} \in X$; we suppose $\alpha$ is based at $\overline{\mathbf{x_0}} \in S(X, 2k+1)$.  We use $\alpha$ to construct a new based loop in $S(X, 2k+1)$ as follows.
For this, we use a device similar to what we called the  \emph{coordinate-centring function} in \cite{LOS19b}.   Namely, define a function $C\colon I_{2k} \to I_{2k}$ by
\begin{equation}\label{eq: function C}
C(r) = \begin{cases}  r+ 1 & 0 \leq r \leq k-1\\
k & r = k \\
 r- 1 & k+1 \leq r \leq 2k.\end{cases}
\end{equation}
Then applying $C$ repeatedly will increase or decrease an integer in $I_{2k}$ to $k$, according as the integer is less than $k$, or greater than $k$, and stabilize at $k$ once there.
For any $r$ with $0 \leq r \leq 2k$, we have $C^p(r) = k$ for all $p \geq k$.

\emph{Notation (leading up to \lemref{lem: homotopy alpha-rho beta})}:
For any $i \in I_M$, suppose  that $\rho_{2k+1} \circ \alpha(i) = \mathbf{x_i} = (x_1, \ldots, x_n) \in X$, so that  $\alpha(i) \in S( \mathbf{x_i}, 2k+1)$.  Then we have
$$\alpha(i) = \big( (2k+1)x_1 + r_1, \ldots, (2k+1)x_n + r_n \big),$$
for some $r_j$  with $0 \leq r_j \leq 2k$, for $j = 1, \ldots, n$.   The centre of $S( \mathbf{x_i}, 2k+1)$ is
$$\overline{ \mathbf{x_i}}  = \big( (2k+1)x_1 + k, \ldots, (2k+1)x_n + k \big).$$
For each $i \in I_M$, then, define a path
$$\gamma_i\colon I_k \to  S( \mathbf{x_i}, 2k+1),$$
that goes from $\alpha(i)$ to $\overline{\mathbf{x_i}}$ then remains at $\overline{\mathbf{x_i}}$  for any remaining time in $I_k$, as follows:
 \begin{equation}\label{eq: gamma}
 \gamma_i(t) =  \big( (2k+1)x_1 + C^t(r_1), \ldots, (2k+1)x_n + C^t(r_n) \big),
 \end{equation}
for $0 \leq t \leq k$,  where $C^0(r) = r$.  (Notice that the  maximum difference $|r_j - k|$ in any one coordinate of $\alpha(i)$ and $\overline{\mathbf{x_i}}$ cannot be greater than $k$, so we have $C^k(r_j) = k$ for all $j$.)
If $i = 0$ or $i=M$, we have  $\alpha(0) = \alpha(M) = \overline{x_0}$,  the centre of $S( \mathbf{x_0}, 2k+1)$.  If we write the coordinates of $\mathbf{x_0}$ as $(x_1, \ldots, x_n)$, then we have
$$\alpha(0) = \alpha(M) =\big( (2k+1)x_1 + k, \ldots, (2k+1)x_n + k \big),$$
so each $r_j = k$, and $\gamma_0\colon I_k \to   S( \mathbf{x_0}, 2k+1)$ and $\gamma_M\colon I_k \to   S( \mathbf{x_0}, 2k+1)$ are both the constant path at $\mathbf{x_0}$.  Then, string these paths and their reverses together to define $\beta$. Decompose $S(I_M, 2k+1)$ using the centres  $\overline{i} = (2k+1)i + k$ for $i = 0, \ldots, M$, so we have
$$S(I_M, 2k+1)  = [0, \overline{0}] \cup \bigcup_{i=0}^{i=M-1}\ [\overline{i}, \overline{i+1}]\ \cup [\overline{M}, (2k+1)M + 2k].$$
(This is not a disjoint union;  the subintervals overlap at their endpoints.)
Then define $\beta$ as the constant path at $\overline{\mathbf{x_0}}$ on $[0, \overline{0}] \cup [\overline{M}, (2k+1)M + 2k]$.  On each subinterval between centres
 $[\overline{i}, \overline{i+1}]$, for $0 \leq i \leq M$, use \eqref{eq: gamma} to define
\begin{equation}\label{eq: beta}
\beta(\overline{i}+s) =  \begin{cases}  \gamma_i(k-s) & 0 \leq s \leq  k\\ \gamma_{i+1}\big(s -(k+1)\big) & k+1 \leq s \leq 2k+1.\end{cases}
\end{equation}
On each subinterval $[\overline{i}, \overline{i+1}]$, then, $\beta$ restricts to the concatenation of paths $\overline{\gamma_i}\cdot \gamma_{i+1}$, in the sense we defined concatenation of paths in \defref{def: concatenation}.

\begin{lemma}\label{lem: homotopy alpha-rho beta}
Suppose $\alpha\colon I_M \to S(X, 2k+1)$ is a based loop in $S(X, 2k+1)$, for  $X \subseteq \Z^n$ any based digital image.  With
$\beta\colon S(I_M, 2k+1) \to S(X, 2k+1)$ the based loop defined as in \eqref{eq: beta} above, we have a based homotopy of based loops
$$\beta \approx \alpha\circ \rho_{2k+1}  \colon S(I_M, 2k+1) \to S(X, 2k+1).$$
\end{lemma}

\begin{proof}
To define a homotopy $H$ over the whole of $S(I_M, 2k+1) \times I_k$, it is sufficient to define it and check continuity on each subrectangle $[(2k+1)i, (2k+1)(i + 1)]\times I_k$, as well as the subrectangle $[(2k+1)M, (2k+1)M + 2k]\times I_k$,  separately, so long as the separate definitions agree on the overlaps $\{(2k+1)i\} \times I_k$.  This is because any two points of $S(I_M, 2k+1)\times I_k$ that are adjacent lie in one or the other of these subrectangles.   We begin by defining $H$ on $S(i, 2k+1) \times I_k$ for each $i \in I_M$.

For $i=0$, we have $S(0, 2k+1) = [0, 2k]$.  Now here, we have $\alpha\circ\rho_{2k+1}(s) = \alpha(0) = \overline{\mathbf{x_0}}$ for each $s \in [0, 2k]$.  Also, above \eqref{eq: beta} we specified that $\beta(s) = \overline{\mathbf{x_0}}$ for $s\in [0, \overline{0}]$, whilst \eqref{eq: beta} has $\beta(\overline{0}+s) =  \gamma_0(k-s)$ for $0 \leq s \leq k$.  But  the $\gamma_i$ are defined such that we have $\gamma_0$ the constant path at $\overline{\mathbf{x_0}}$ also (see the comments following \eqref{eq: gamma}).  That is, we also have $\beta(s) =  \overline{\mathbf{x_0}}$ for each $s \in [0, 2k]$. Since $\alpha\circ\rho_{2k+1}(s) = \beta(s) =  \overline{\mathbf{x_0}}$ for each $s \in [0, 2k]$, we may define $H$ to be the constant map $H(s, t) =  \overline{\mathbf{x_0}}$ on $[0, 2k]\times I_k$.  A similar state of affairs pertains at the other end of $S(I_M, 2k+1)$: we have $\alpha\circ\rho_{2k+1}(s) = \beta(s) =  \overline{\mathbf{x_0}}$ for each $s \in [(2k+1)M, (2k+1)M+2k]$, and we define $H$ to be the constant map $H(s, t) =  \overline{\mathbf{x_0}}$ on $[(2k+1)M, (2k+1)M+2k]\times I_k$.

Now consider $i$ for $1 \leq i \leq M-1$.  Here, we have $\alpha\circ\rho_{2k+1}(s) = \alpha(i)$ for each $s \in S(i, 2k+1)$. We write $S(i, 2k+1) = [(2k+1)i, (2k+1)i + 2k]$ as
$$
\begin{aligned}
S(i, 2k+1) & = [(2k+1)i, (2k+1)i + k-1] \sqcup [(2k+1)i+k, (2k+1)i + 2k] \\
& = [\overline{i-1}+k+1, \overline{i-1}+2k] \sqcup [\overline{i}, \overline{i}+k] \\
& = \{ \overline{i-1}+k+1+s \mid 0 \leq s \leq k-1\} \sqcup \{ \overline{i}+s-k \mid k \leq s \leq 2k\}.
\end{aligned}
$$
According to \eqref{eq: beta}, then, we have $\beta$ restricted to $S(i, 2k+1) = [(2k+1)i, (2k+1)i + 2k]$ given by
$$
\begin{aligned}
\beta\big( (2k+1)i + s \big)&= \begin{cases} \beta( \overline{i-1}+k+1+s) & 0 \leq s \leq k-1\\
\beta(  \overline{i}+s-k) & k \leq s \leq 2k\end{cases}\\
&= \begin{cases} \gamma_{ (i-1)+1}\big( (k+1+s) -(k+1)\big) & 0 \leq s \leq k-1\\
\gamma_i\big( k-(s-k)\big) & k \leq s \leq 2k\end{cases}\\
&= \begin{cases} \gamma_{i}(s) & 0 \leq s \leq k-1\\
\gamma_i( 2k-s) & k \leq s \leq 2k\end{cases}\\
&= (\gamma_i*\overline{\gamma_i})(s),
\end{aligned}
$$
where $\gamma_i*\overline{\gamma_i}$ denotes the short concatenation of paths as in \eqref{eq: short concat} of \defref{def: concatenation}. By \lemref{lem: inverses}, we have a homotopy relative the endpoints $H \colon I_{2k}\times I_k \to S(X, 2k+1)$  from $\gamma_i*\overline{\gamma_i}$ to $C^{\gamma_i(0)}_{2k} = C^{\alpha(i)}_{2k}$, which translates to a homotopy relative the endpoints  $H \circ (T\times \mathrm{id}_{I_k})\colon [(2k+1)i, (2k+1)i + 2k] \times I_k \to S(X, 2k+1)$ from the restriction of $\beta$ to the restriction of  $\alpha\circ\rho_{2k+1}$, where $T \colon  [(2k+1)i, (2k+1)i + 2k]  \to [0, 2k]$ is the translation $T(s) = s-2k$.

Then, we extend $H$ over the subrectangle $[(2k+1)i, (2k+1)(i + 1)]\times I_k$ by setting
$H\big((2k+1)(i + 1), t\big) = \alpha(i+1)$ for all $t \in I_k$.  Notice that this is consistent with how we would define $H$ over the next subinterval $S(i+1, 2k+1)$.  Since $H$ is already constant at $\alpha(i)$ on $\{(2k+1)i + 2k\} \times I_k$, and we have $\alpha(i) \sim_{S(X, 2k+1)} \alpha(i+1)$ since $\alpha$ is continuous, it follows that this extends $H$ continuously over  $[(2k+1)i, (2k+1)(i + 1)]\times I_k$.
As already discussed, these homotopies piece together to give a homotopy
$$H\colon S(I_M, 2k+1) \times I_M \to S(X, 2k+1)$$
from $\beta$ to $\alpha\circ \rho_{2k+1}$. Since we have
$H\big( 0, t\big) =  \alpha(0) = \overline{x_0}$ and $H\big( (2k+1)M + 2k, t\big) =  \alpha(M) = \overline{x_0}$, this is a based homotopy of based loops.
\end{proof}

The next step is to show $\beta$ and the standard cover $\widehat{\rho_{2k+1} \circ \alpha}$ (cf.~\thmref{thm: path odd subdivision map}) are homotopic.  We first give a careful description of
$\widehat{\rho_{2k+1} \circ \alpha}$.
As observed following \eqref{eq: beta}, when restricted to the subinterval $[\overline{i}, \overline{i+1}]$, we have $\beta = \overline{\gamma_i}\cdot \gamma_{i+1}$.
Our strategy is to view $\beta$ and $\widehat{\rho_{2k+1} \circ \alpha}$ piecewise, over the subintervals $[\overline{i}, \overline{i+1}]$ of $S(I_M, 2k+1)$, and define a homotopy from one to the other on each of these subintervals separately.

Now the definition of $\widehat{\rho_{2k+1} \circ \alpha}$ on each subinterval $[\overline{i}, \overline{i+1}]$ depends on both the values $\rho_{2k+1} \circ \alpha(i) \in X$ and $\rho_{2k+1} \circ \alpha(i+1) \in X$.
We augment the notation leading up to \lemref{lem: homotopy alpha-rho beta} as follows:

\emph{Notation (leading up to \lemref{lem: cover rho-alpha})}:
For any $i \in I_M$ with $0 \leq i \leq M-1$, suppose  that $\rho_{2k+1} \circ \alpha(i) = \mathbf{x_i} = (x_1, \ldots, x_n) \in X$ and $\rho_{2k+1} \circ \alpha(i+1) = \mathbf{x_{i+1}} = (x'_1, \ldots, x'_n) \in X$, so that  $\alpha(i+1) \in S( \mathbf{x_{i+1}}, 2k+1)$.  Then we have
$$\alpha(i) = \big( (2k+1)x_1 + r_1, \ldots, (2k+1)x_n + r_n \big),$$
and
$$\alpha(i+1)  = \big( (2k+1)x'_1 + r'_1, \ldots, (2k+1)x'_n + r'_n \big)$$
for some $r_j$ and $r'_j$ with $0 \leq r_j, r'_j \leq 2k$ for each $j = 1, \ldots, n$.
The centres of $S( \mathbf{x_i}, 2k+1)$ and $S( \mathbf{x_{i+1}}, 2k+1)$ are
$$\overline{ \mathbf{x_i}}  = \big( (2k+1)x_1 + k, \ldots, (2k+1)x_n + k \big)$$
and
$$\overline{ \mathbf{x_{i+1}}}  = \big( (2k+1)x'_1 + k, \ldots, (2k+1)x'_n + k \big)$$
respectively.  In the following result, and in the sequel, we denote the $j$th coordinate of
$\widehat{\rho_{2k+1} \circ \alpha}(t)$, for any $t \in S(I_M, 2k+1)$, by
$$\widehat{\rho_{2k+1} \circ \alpha}(s)_j,$$
for $j = 1, \ldots, n$.

\begin{lemma}\label{lem: cover rho-alpha}
Suppose $\alpha\colon I_M \to S(X, 2k+1)$ is a based loop in $S(X, 2k+1)$, for  $X \subseteq \Z^n$ any based digital image, and $\widehat{\rho_{2k+1} \circ \alpha}\colon S(I_M, 2k+1) \to S(X, 2k+1)$ is the standard cover of   $\rho_{2k+1} \circ \alpha\colon I_M \to X$ as in \thmref{thm: path odd subdivision map}.
\begin{itemize}
\item[(1)] For $s \in [0, \overline{0}] \sqcup [\overline{M}, \overline{M} + 2k] \subseteq S(I_M, 2k+1)$, we have
$$\widehat{\rho_{2k+1} \circ \alpha}(s) = \overline{\mathbf{x_0}}.$$
\item[(2)]  For $i \in I_M$ with $0 \leq i \leq M-1$, write $[\overline{i}, \overline{i+1}]  = \{ \overline{i} + s \mid 0 \leq s \leq 2k+1\}$.  When restricted to $[\overline{i}, \overline{i+1}]$,  we may write the $j$th coordinate of the standard cover as
$$\widehat{\rho_{2k+1} \circ \alpha}(\overline{i} + s )_j = \begin{cases}
(2k+1)x_j + C^{k-s}\big( k + k(x'_j - x_j)\big) & 0 \leq s \leq k\\
(2k+1)x'_j + C^{s-(k+1)}\big( k - k(x'_j - x_j)\big) & k+1 \leq s \leq 2k+1.\end{cases}$$
\end{itemize}
\end{lemma}

\begin{proof}
The first item is part of the construction of the standard cover.  Now consider the second item, over a subinterval $[\overline{i}, \overline{i+1}]$.   From the construction of the standard cover, we have
$$\widehat{\rho_{2k+1} \circ \alpha}(\overline{i} + s )_j = (2k+1)x_j + k + s[x'_j - x_j]  \text{ for } 0 \leq s \leq 2k+1.$$
For $k+1 \leq s \leq 2k+1$, we may write
$$\begin{aligned}\widehat{\rho_{2k+1} \circ \alpha}(\overline{i} + s )_j &= (2k+1)x_j + k + s[x'_j - x_j] \\
&=(2k+1)x'_j - (2k+1)x'_j + k + (2k+1)x_j  + sx'_j - sx_j \\
&= (2k+1)x'_j + k -\big( (2k+1) - s\big)[ x'_j  -x_j].\end{aligned}$$
Thus, we have
$$\widehat{\rho_{2k+1} \circ \alpha}(\overline{i} + s )_j = \begin{cases}
(2k+1)x_j + k + s(x'_j - x_j) & 0 \leq s \leq k\\
(2k+1)x'_j + k -\big( (2k+1) - s\big)[ x'_j  -x_j] & k+1 \leq s \leq 2k+1.\end{cases}$$
In the notation above, we have $\rho_{2k+1} \circ \alpha(i) = \mathbf{x_i}$ and $\rho_{2k+1} \circ \alpha(i+1) = \mathbf{x_{i+1}}$.  Now
$\rho_{2k+1} \circ \alpha$ is continuous, and so we have $\mathbf{x_{i+1}} \sim_X \mathbf{x_i}$.  Therefore, $|x'_j - x_j| \leq 1$ for $j = 1, \ldots, n$.  We divide and conquer, considering $x'_j - x_j = 0$, $x'_j - x_j = +1$, and $x'_j - x_j= -1$ for each $j$ separately.

If $x'_j - x_j =0$, then the $j$th coordinate of the standard cover reduces to
\begin{equation}\label{eq: x' = x cover}
\widehat{\rho_{2k+1} \circ \alpha}(\overline{i} + s )_j = (2k+1)x_j + k  \text{ for } 0 \leq s \leq 2k+1.
\end{equation}
Item (2) of the enunciation (with $x'_j = x_j$) reduces to
\begin{equation}\label{eq: x' = x alt cover}
\widehat{\rho_{2k+1} \circ \alpha}(\overline{i} + s )_j = \begin{cases}
(2k+1)x_j + C^{k-s}( k) & 0 \leq s \leq k\\
(2k+1)x_j + C^{s-(k+1)}( k) & k+1 \leq s \leq 2k+1.\end{cases}
\end{equation}
But \eqref{eq: x' = x cover} and \eqref{eq: x' = x alt cover} agree, since we have $C^{k-s}( k) = k$ for each $s = 0, \ldots, k$ and  $C^{s-(k+1)}( k) = k$ for each $s = k+1, \ldots, 2k+1$.
Notice that this case includes the case in which we have  $\rho_{2k+1} \circ \alpha(i) = \rho_{2k+1} \circ \alpha(i+1) = \mathbf{x_i}$, where we have
$$\widehat{\rho_{2k+1} \circ \alpha}(s) = \overline{\mathbf{x_i}}$$
for each $s \in [\overline{i}, \overline{i+1}]$.

Next, suppose  that we have $x'_j - x_j =+1$.  Then the $j$th coordinate of the standard cover reduces to
\begin{equation}\label{eq: x' - x=1 cover}
\widehat{\rho_{2k+1} \circ \alpha}(\overline{i} + s )_j = (2k+1)x_j + k +s \text{ for } 0 \leq s \leq 2k+1.
\end{equation}
Item (2) of the enunciation (with $x'_j = x_j + 1$) reduces to
\begin{equation}\label{eq: x' - x=1 alt cover}
\widehat{\rho_{2k+1} \circ \alpha}(\overline{i} + s )_j = \begin{cases}
(2k+1)x_j + C^{k-s}( 2k) & 0 \leq s \leq k\\
(2k+1)(x_j + 1)+C^{s-(k+1)}( 0) & k+1 \leq s \leq 2k+1.\end{cases}
\end{equation}
Now $C^{k-s}(2k) = 2k - (k-s) = k+s$ for $s=0, \dots, k$, whereas $C^{s-(k+1)}( 0) = 0+\big( s-(k+1) \big) = s-(k+1)$ for $s=k+1, \dots, 2k+1$.  Hence the second term of \eqref{eq: x' - x=1 alt cover} also gives $(2k+1)x_j  + k+s$ for $s=k+1, \dots, 2k+1$, and \eqref{eq: x' - x=1 cover} agrees with \eqref{eq: x' - x=1 alt cover} on $[\overline{i}, \overline{i+1}]$.

Finally, suppose we have  $x'_j - x_j =-1$.  Then the $j$th coordinate of the standard cover reduces to
\begin{equation}\label{eq: x' - x=-1 cover}
\widehat{\rho_{2k+1} \circ \alpha}(\overline{i} + s )_j = (2k+1)x_j + k -s \text{ for } 0 \leq s \leq 2k+1.
\end{equation}
Item (2) of the enunciation (with $x'_j = x_j - 1$) reduces to
\begin{equation}\label{eq: x' - x=-1 alt cover}
\widehat{\rho_{2k+1} \circ \alpha}(\overline{i} + s )_j = \begin{cases}
(2k+1)x_j + C^{k-s}( 0) & 0 \leq s \leq k\\
(2k+1)(x_j - 1)+C^{s-(k+1)}( 2k) & k+1 \leq s \leq 2k+1.\end{cases}
\end{equation}
Here, we have $C^{k-s}(0) = 0 + (k-s) = k-s$ for $s=0, \dots, k$, and $C^{s-(k+1)}( 2k) = 2k-\big( s-(k+1) \big) = (2k+1) + k-s$ for $s=k+1, \dots, 2k+1$. Here, the second term of \eqref{eq: x' - x=1 alt cover} then gives $(2k+1)x_j +k-s$ for $s=k+1, \dots, 2k+1$, and in this last case also, \eqref{eq: x' - x=-1 cover} agrees with \eqref{eq: x' - x=-1 alt cover} on $[\overline{i}, \overline{i+1}]$.
\end{proof}

\begin{lemma}\label{lem: homotopy cover-rho-alpha beta}
Suppose $\alpha\colon I_M \to S(X, 2k+1)$ is a based loop in $S(X, 2k+1)$, for  $X \subseteq \Z^n$ any based digital image.  With
$\beta\colon S(I_M, 2k+1) \to S(X, 2k+1)$ the based loop defined as in \eqref{eq: beta} above, and $\widehat{\rho_{2k+1} \circ \alpha}\colon S(I_M, 2k+1) \to S(X, 2k+1)$ the standard cover of   $\rho_{2k+1} \circ \alpha\colon I_M \to X$ as in \thmref{thm: path odd subdivision map},
we have a based homotopy of based loops
$$\beta \approx \widehat{\rho_{2k+1} \circ \alpha}  \colon S(I_M, 2k+1) \to S(X, 2k+1).$$
\end{lemma}

\begin{proof}  We will construct a based homotopy of based loops
$$G \colon S(I_M, 2k+1) \times I_k\to S(X, 2k+1)$$
from $\beta$ to $\widehat{\rho_{2k+1} \circ \alpha}$.
As discussed in the proof of \lemref{lem: homotopy alpha-rho beta}, we have
$\beta(s) = \overline{\mathbf{x_0}}$ for (at least)  $s\in [0, \overline{0}]$ and  $s \in [(2k+1)M, (2k+1)M+2k]$.  Also, the standard cover is defined as
$\widehat{\rho_{2k+1} \circ \alpha}(s) = \overline{\mathbf{x_0}}$ on the same subintervals.
So we set $G$ to be the constant map $G(s, t) =  \overline{\mathbf{x_0}}$ on $[0, \overline{0}]\times I_k$ and  $[\overline{M}, \overline{M}+k]\times I_k$.

Refer to the notation leading up to \lemref{lem: homotopy alpha-rho beta}.
We have
$$\widehat{\rho_{2k+1} \circ \alpha}(\overline{i}) = \overline{  \rho_{2k+1} \circ \alpha(i)   } = \overline{  \mathbf{x_i}  }$$
for each centre $\overline{i} \in S(I_M, 2k+1)$.  From \eqref{eq: beta}, $\beta$ also has the property that
$$\beta(\overline{i}) = \gamma_i(k) = \overline{  \mathbf{x_i}  }.$$
For each $i \in I_M$ with $0 \leq i \leq M-1$, we will define a homotopy
$$G\colon [\overline{i}, \overline{i+1}] \times I_k \to S(X, 2k+1)$$
\emph{relative the endpoints}, which is to say that we will have $G(\overline{i}, t) = \beta( \overline{i}) = \widehat{\rho_{2k+1} \circ \alpha}( \overline{i}) = \overline{  \mathbf{x_i}  }$ for each $i \in I_M$ with $0 \leq i \leq M-1$ and each $t \in I_k$, from the restriction of $\beta$ to the restriction of $\widehat{\rho_{2k+1} \circ \alpha}$.
 Together with the constant homotopies  on $[0, \overline{0}] \times I_k$ and $[\overline{M}, (2k+1)M + 2k]\times I_k$, these will piece together continuously to give the desired homotopy.

Refer to the notation leading up to \lemref{lem: cover rho-alpha}.
From  \eqref{eq: gamma} and \eqref{eq: beta},
we have the $j$th coordinate function of $\beta$ described as
\begin{equation}\label{eq: beta coordinates}
\beta(\overline{i}+s)_j = \begin{cases}
(2k+1)x_j + C^{k-s} (r_j) & 0 \leq s \leq k\\
(2k+1)x'_j + C^{s-(k+1)} (r'_j) & k+1 \leq s \leq 2k+1.
\end{cases}
\end{equation}
Item (2) of \lemref{lem: cover rho-alpha} gives our description of the $j$th coordinate of the standard cover that we use.
Now $\alpha$ is continuous, so we have $\alpha(i) \sim_{S(X, 2k+1)} \alpha(i+1)$.  This means that in each coordinate we have
\begin{equation}\label{eq: x-r relation}
| (2k+1)x'_j + r'_j -((2k+1)x_j + r_j)| \leq 1
\end{equation}
for each $j$.  Also, because $\mathbf{x_{i+1}} \sim_X \mathbf{x_{i}}$, we have $|x'_j - x_j| \leq 1$ for each $j$.

We will work with each coordinate separately, proceeding differently according as  $x'_j - x_j = 0$ or not.  We will define $G\colon [\overline{i}, \overline{i+1}]\times I_{k}\to S(X, 2k+1)$ in terms of its coordinate functions, thus:
$$G(s, t) = \big( G_1(s, t), \ldots, G_n(s, t)\big),$$
for $(s, t) \in [\overline{i}, \overline{i+1}]\times I_{k}$.  In fact, we will construct the $G_j$ so that they map
$$G_j\colon [\overline{i}, (2k+1)i+ 2k]\times I_{k}\to S(x_j, 2k+1) \subseteq \Z,$$
and
$$G_j\colon [(2k+1)i+ 2k+1, \overline{i+1}]\times I_{k}\to S(x'_j, 2k+1) \subseteq \Z.$$
Then the coordinate functions $G_j$ will assemble into a homotopy $G$ that maps
$$G\colon [\overline{i}, (2k+1)i+ 2k]\times I_{k}\to S(\mathbf{x_{i}}, 2k+1) \subseteq S(X, 2k+1)$$
and
$$G\colon [(2k+1)i+ 2k+1, \overline{i+1}]\times I_{k}\to S(\mathbf{x_{i+1}}, 2k+1) \subseteq S(X, 2k+1).$$
Thus, $G$ will be continuous on $[\overline{i}, \overline{i+1}]\times I_{k}$ if and only if each of its coordinate functions is.

We treat the cases in which $x'_j - x_j = \pm1$ and $x'_j - x_j = 0$ separately.

\bigskip

\noindent\textbf{Case (i) $\mathbf{x'_j - x_j = \pm1}$:}   Suppose that in the $j$th coordinates of $\mathbf{x_{i}}$ and $\mathbf{x_{i+1}}$ we have
$x'_j - x_j = +1$.  Then \eqref{eq: x-r relation} implies that we have  $| (2k+1) + r'_j -r_j| \leq 1$ which, given that $0 \leq r_j, r'_j \leq 2k$, implies that $r_j = 2k$ and $r'_j = 0$.
Then the formulas for the $j$th coordinates of $\beta$ and of $\widehat{\rho_{2k+1} \circ \alpha}$ agree.  Namely, because we have $x'_j - x_j = +1$, item (2) of \lemref{lem: cover rho-alpha} reduces to
\begin{equation}\label{eq: cover x'j - xj = pm1}
\widehat{\rho_{2k+1} \circ \alpha}(\overline{i} + s )_j
= \begin{cases}
(2k+1)x_j + C^{k-s} (2k) & 0 \leq s \leq k\\
(2k+1)x'_j + C^{s-(k+1)} (0) & k+1 \leq s \leq 2k+1.
\end{cases}
\end{equation}
But with $r_j = 2k$ and $r'_j = 0$, this is exactly the formula for  $\beta(\overline{i}+s)_j$ from  \eqref{eq: beta coordinates}.  On the other hand,
suppose that we have  $x'_j - x_j = -1$.   Then \eqref{eq: x-r relation} implies that we have $| (2k+1) - (r'_j -r_j)| \leq 1$ which, given that $0 \leq r_j, r'_j \leq 2k$, implies that $r_j = 0$ and $r'_j = 2k$.  In a similar way, item (2) of \lemref{lem: cover rho-alpha}  and  \eqref{eq: beta coordinates} reduce to the same formula here, too.  So, in the case in which
$x'_j - x_j = \pm1$, we have
$$
\widehat{\rho_{2k+1} \circ \alpha}(\overline{i} + s )_j
= \beta(\overline{i}+s)_j,$$
for each $s$ with $0 \leq s \leq 2k+1$.  In this case, then, we define $G_j(\overline{i}+s, t)$  independently of $t$ as
\begin{equation}\label{eq: beta homotopy A}
G_j(\overline{i}+s, t)  = \widehat{\rho_{2k+1} \circ \alpha}(\overline{i} + s )_j = \beta(\overline{i} + s )_j
\end{equation}
for $0 \leq t \leq k$.  The value for each $s$ is given by \eqref{eq: cover x'j - xj = pm1}.  Continuity in the $j$th coordinate in this case follows from the continuity of $\beta$ (or of
$\widehat{\rho_{2k+1} \circ \alpha}$).  Indeed, suppose we have $(\overline{i} + s, t) \sim (\overline{i} + s', t')$ in $[\overline{i}, \overline{i+1}]\times I_{k}$.  Then, in particular, we have $\overline{i} + s \sim \overline{i} + s'$ in $[\overline{i}, \overline{i+1}]$.    Since $G_j(\overline{i} + s, t)$ is independent of $t$, we have $|G_j(\overline{i} + s', t') - G_j(\overline{i} + s, t)| = |\beta(\overline{i} + s')_j - \beta(\overline{i} + s)_j| \leq 1$, since $\beta$ is continuous.

\bigskip

\noindent\textbf{Case (ii) $\mathbf{x'_j - x_j = 0}$:}
If $x'_j - x_j = 0$, then item (2) of \lemref{lem: cover rho-alpha} (with $x'_j = x_j$) gives  the $j$th coordinate of  $\widehat{\rho_{2k+1} \circ \alpha}(\overline{i} + s)$ as $(2k+1)x_j + k$ for each $s$ (recall that we have $C^p(k) = k$ for all $p \geq 0$).   Also, if $x'_j - x_j = 0$, then \eqref{eq: x-r relation} implies that we have $|r'_j -r_j| \leq 1$.
  Unfortunately, describing a suitable homotopy in this case involves a further divide-and-conquer, according as how $r'_j$ and $r_j$ compare with each other.

\medskip

\textbf{Case (ii) (A) $\mathbf{x'_j - x_j = 0}$ and $\mathbf{r'_j = r_j}$:}  If  $r'_j = r_j$, then we have $C^p(r'_j) = C^p(r_j)$ for $p \geq 0$.  Then set
\begin{equation}\label{eq: beta homotopy iiA}
G_j(\overline{i}+s, t) = \begin{cases}
(2k+1)x_j + C^{k-s} (r_j) & 0 \leq s < k-t\\
(2k+1)x_j + C^{t} (r_j) & k-t \leq s \leq k+1+t\\
(2k+1)x_j + C^{s-(k+1)} (r_j) & k+1+t < s \leq 2k+1.\end{cases}
\end{equation}
When $t=0$, this formula reduces to that for $\beta(\overline{i}+s)$ from \eqref{eq: beta coordinates} (with $x'_j = x_j$ and $r'_j = r_j$).
When $t = k$, this formula reduces to
$$G_j(\overline{i}+s, k)  =   (2k+1)x_j + k $$
for $0 \leq s \leq 2k+1$, since $C^p(r_j) = k$ for $p \geq k$.  Namely, it reduces to the formula for $\widehat{\rho_{2k+1} \circ \alpha}(\overline{i} + s)_j$ from  item (2) of \lemref{lem: cover rho-alpha} (with $x'_j = x_j$).

We now check continuity of \eqref{eq: beta homotopy iiA}; this follows an argument very similar to that used for \lemref{lem: inverses}.
To check that $G_j$ is continuous,  divide $[\overline{i}, \overline{i+1}] \times I_{k}$ into two overlapping regions:  Region $A$, consisting of those points $(\overline{i}+s, t)$  that satisfy both $s+t \geq k$ and $s-t \leq k+1$; and region $B$, consisting of the points that satisfy either $s+t \leq k+1$ or $s-t \geq k$.  Now any pair of adjacent points $(\overline{i}+s, t)\sim (\overline{i}+s', t')$ in $[\overline{i}, \overline{i+1}] \times I_{k}$ satisfies $|(s-t) - (s'-t')| \leq 2$ and $|(s+t) - (s'+t')| \leq 2$, and hence either both lie in region $A$ or both lie in region $B$.  Suppose first that both lie in region
$A$.  From \eqref{eq: beta homotopy iiA}, for these points we have $G_j(s, t) =  C^t(r_j)$ and $G_j(s', t') = C^{t'}(r_j)$.  Since $(\overline{i}+s, t)\sim (\overline{i}+s', t')$, we have $|t' - t| \leq 1$, and hence $|C^{t'}(r_j) - C^{t}(r_j)| \leq 1$.  Thus we have
$| G_j(\overline{i}+s', t') - G_j(\overline{i}+s, t)| \leq 1$ when both adjacent points lie in $A$.  On the other hand, suppose adjacent points $(\overline{i}+s, t) \sim (\overline{i}+s', t')$ both lie in region $B$.
First suppose that $s+t \leq k+1$ and $s'+t' \leq k+1$ (the left-hand half of region $B$).  Here,  \eqref{eq: beta homotopy iiA} specializes to
$$G_j(\overline{i}+s, t) = \begin{cases} (2k+1)x_j +  C^{k-s}(r_j) & s +t \leq k-1\\
(2k+1)x_j +  C^{t}(r_j) & k \leq s +t \leq k+1.\end{cases}$$
If we have  $s+t \leq k-1$ and $s'+t' \leq k-1$, then we have $| G_j(\overline{i}+s', t') - G_j(\overline{i}+s, t)| = |C^{k-s'}(r_j) - C^{k-s}(r_j)| \leq 1$, since we have $|(k-s') - (k-s)| = |s - s'| \leq 1$.   If we have $k \leq s+t \leq  k+1$ and $k \leq s'+t' \leq k+1$, then we have $G_j(\overline{i}+s, t) =  C^{t}(r_j)$ and $G_j(\overline{i}+s', t') =  C^{t'}(r_j)$, and again we have $| G_j(\overline{i}+s', t') - G_j(\overline{i}+s, t)| \leq 1$ since $|t' - t| \leq 1$.  Finally, for adjacent points in the left-hand part of region B, suppose that we have $(s, t)$ with $s+t \leq k-1$ and $(s', t')$ with $k \leq s'+t' \leq k+1$.   Now this is only possible if we also have $s' \geq  s$ which, assuming we have $(\overline{i}+s, t)\sim (\overline{i}+s', t')$, means that we have $s' = s$ or $s' = s+1$.  Then $| G_j(\overline{i}+s', t') - G_j(\overline{i}+s, t)| = |C^{k-s}(r_j) - C^{t'}(r_j)|$.  If $s' = s$, then the inequality $k \leq s'+t' \leq k+1$ yields $k-s \leq t' \leq k-s+1$, and if $s' = s+1$, then the same inequality yields $k-s-1 \leq t' \leq k-s$.  In either case we have $|t' - (k-s)| \leq 1$, and thus $| G_j(\overline{i}+s', t') - G_j(\overline{i}+s, t)|  \leq 1$.    For the other remaining possibility, when  $k \leq s+t \leq k+1$ and $s'+t' \leq  k-1$, we find that $| G_j(\overline{i}+s', t') - G_j(\overline{i}+s, t)| \leq 1$ follows by interchanging the roles of $(s, t)$ and $(s', t')$ in the last few steps.  It remains to confirm that $G_j$ preserves adjacency for adjacent points $(\overline{i}+s, t)\sim (\overline{i}+s', t')$  that satisfy  $s-t \geq k$ and $s'-t' \geq k$ (the right-hand half of region $B$).  The same argument, \emph{mutatis mutandis}, as we just used for the left-hand half of region $B$ will confirm that we have $| G_j(\overline{i}+s', t') - G_j(\overline{i}+s, t)| \leq 1$ here also.  We omit the details.  This completes the check of continuity for $G_j$ on  $[\overline{i}, \overline{i+1}]\times I_k$ in this sub-case.

\medskip

\textbf{Case (ii) (B) $\mathbf{x'_j - x_j = 0}$ and $k\leq r'_j < r_j$ or $r_j < r'_j \leq k$:}
Since $|r'_j -r_j| \leq 1$ (going back to the start of discussion for Case (ii)), we have $C(r_j) = r'_j$ here (and hence, $C^{p+1}(r_j) = C^p(r'_j)$ for $p \geq 0$).
Roughly speaking, in this case our homotopy is the same as the previous case on the left half of $[\overline{i}, \overline{i+1}] \times I_k$ but, on the right half, we pause at the start for one unit of time, to allow the values $G(k, t)_j$ to ``catch-up" with those of $G(k+1, t)_j$, as it were.  This effect may be achieved as follows.  We denote the homotopy in this sub-case by $G^B_j(\overline{i}+s, t)$, for $(s, t) \in [0, 2k+1]\times I_k$, and define it in terms of the homotopy  $G_j(\overline{i}+s, t)$   from \eqref{eq: beta homotopy iiA} of the previous sub-case.  Here, because we must have $k\leq r'_j < r_j \leq 2k$ or $0 \leq r_j < r'_j \leq k$, it follows that we have $C^{k-1}(r'_j) = C^k(r'_j) = k$, and thus, referring to \eqref{eq: beta coordinates} (with $x'_j = x_j$), we have
$$\beta(\overline{i}+s)_j =  (2k+1)x_j + k =  \widehat{\rho_{2k+1} \circ \alpha}(\overline{i} + s)_j,$$
(at least) for $s = 2k, 2k+1$ (recall that in this case, in which we have $x'_j = x_j$,  item (2) of \lemref{lem: cover rho-alpha}  gives  the $j$th coordinate of  $\widehat{\rho_{2k+1} \circ \alpha}(\overline{i} + s)$ as $(2k+1)x_j + k$ for each $s$).  Therefore, we may take the homotopy here to be constant for $s = 2k$ and $s= 2k+1$, and not just relative the endpoint $s = 2k+1$.  We define
\begin{equation}\label{eq: beta homotopy iiB1}
G^B_j(\overline{i}+s, t) = \begin{cases}
G_j(\overline{i}+s, t) & 0 \leq s \leq k\\
G_j(\overline{i}+s+1, t) & k+1 \leq s \leq 2k\\
G_j(\overline{i}+2k+1, t) =  \beta(\overline{i}+2k+1)_j& s = 2k+1,\end{cases}
\end{equation}
where $G_j(\overline{i}+s, t)$  is given by \eqref{eq: beta homotopy iiA}.
Then for $t=0$, referring to \eqref{eq: beta homotopy iiA}, this yields
$$G^B_j(\overline{i}+s, 0) = \begin{cases}
(2k+1)x_j + C^{k-s}(r_j) & 0 \leq s < k\\
(2k+1)x_j + C^{0}(r_j) & s = k\\
(2k+1)x_j + C^{(k+1)-(s+1)}(r_j) & k+1 \leq s \leq 2 k\\
(2k+1)x_j + k & s = 2 k+1,
\end{cases}
$$
which may be re-written as
$$G^B_j(\overline{i}+s, 0) = \begin{cases}
(2k+1)x_j + C^{k-s}(r_j) & 0 \leq s \leq k\\
(2k+1)x_j + C^{k-(s+1)}(r'_j) & k+1 \leq s \leq 2 k+1,
\end{cases}
$$
and thereby recognized as agreeing with $\beta(\overline{i}+s)_j$ from \eqref{eq: beta coordinates} (with $x'_j = x_j$).  A similar direct check confirms  that  we have
$G^B_j(\overline{i}+s, k) = \widehat{\rho_{2k+1} \circ \alpha}(\overline{i} + s) = (2k+1)x_j + k$ for each $s$.
Since $G_j$ is continuous, from the previous sub-case, it follows that this $G^B_j$ is continuous on the separate parts of its domain $[\overline{i}, \overline{i}+k] \times I_k$,
$[\overline{i}+k+1, \overline{i}+2k] \times I_k$, and $\{\overline{i}+2k+1\} \times I_k$.  Where the latter two abut, the homotopy is a constant map and it is evident that $G^B_j$ is continuous on the whole of   $[\overline{i}+k+1, \overline{i}+2k+1] \times I_K$.   Then, for $(s, t) \in \{\overline{i}+k, \overline{i}+k+1\} \times I_k$, \eqref{eq: beta homotopy iiB1}(and \eqref{eq: beta homotopy iiA})  gives
$$G^B_j(\overline{i}+k, t) = G_j(\overline{i}+k, t) = (2k+1)x_j + C^{t}(r_j)$$
and
$$G^B_j(\overline{i}+k+1, t) = G_j(\overline{i}+k+2, t) = \begin{cases}
(2k+1)x_j + C^{1}(r_j) & t=0\\
(2k+1)x_j + C^{t}(r_j) & 1 \leq t \leq k.
\end{cases}
$$
Now, if we have $(s, t) \sim (s', t')$ in $\{\overline{i}+k, \overline{i}+k+1\} \times I_k$, then either $\{t, t'\} \subseteq \{0, 1\}$, or $|t' - t| \leq 1$ and $1\leq t, t' \leq k$.  If  $\{t, t'\} \subseteq \{0, 1\}$, then
$$ \{ G^B_j(\overline{i}+s, t),  G^B_j(\overline{i}+s', t') \} \subseteq \{ (2k+1)x_j + C^{1}(r_j), (2k+1)x_j + C^{0}(r_j) \},$$
and since $|(2k+1)x_j + C^{1}(r_j)-\big( (2k+1)x_j + C^{0}(r_j)\big)|  = |C(r_j)-r_j| \leq 1$, it follows that we have  $| G^B_j(\overline{i}+s, t)-  G^B_j(\overline{i}+s', t')| \leq 1$.  But if
$1\leq t, t' \leq k$, then we have
$$| G^B_j(\overline{i}+s, t)-  G^B_j(\overline{i}+s', t')| = | C^{t}(r_j) - C^{t'}(r_j)| \leq 1,$$
because we have $|t' - t| \leq 1$.  It follows that $G^B_j$ is continuous when restricted to $\{\overline{i}+k, \overline{i}+k+1\} \times I_k$, and this is now sufficient to conclude that $G^B_j$ is continuous over the whole of $[ \overline{i}, \overline{i+1}] \times I_k$.

\medskip

\textbf{Case (ii) (C) $\mathbf{x'_j - x_j = 0}$ and $k\leq r_j < r'_j$ or $r'_j < r_j \leq k$:}
This last case needs a variation on Case (ii) (A) similar to that which we just made for Case (ii) (B).
Since $|r'_j -r_j| \leq 1$ (going back to the start of discussion for Case (ii)), we have $C(r'_j) = r_j$ here (and hence, $C^{p+1}(r'_j) = C^p(r_j)$ for $p \geq 0$).
Here, our homotopy is the same as in Case (ii) (A) on the right half of $[\overline{i}, \overline{i+1}] \times I_k$ but, on the left half, we pause at the start for one unit of time, to allow the values $G(k+1, t)_j$ to ``catch-up" with those of $G(k, t)_j$, as it were.  We denote the homotopy in this sub-case by $G^C_j(\overline{i}+s, t)$, for $(s, t) \in [0, 2k+1]\times I_k$, and again define it in terms of the homotopy  $G_j(\overline{i}+s, t)$   from \eqref{eq: beta homotopy iiA}.  Here, because we must have $k\leq r_j < r'_j \leq 2k$ or $0 \leq r'_j < r_j \leq k$, it follows that we have $C^{k-1}(r_j) = C^k(r_j) = k$, and thus, referring to \eqref{eq: beta coordinates} (with $x'_j = x_j$), we have
$$\beta(\overline{i}+s)_j =  (2k+1)x_j + k =  \widehat{\rho_{2k+1} \circ \alpha}(\overline{i} + s)_j,$$
(at least) for $s = 0, 1$ (recall again that when we have $x'_j = x_j$,  item (2) of \lemref{lem: cover rho-alpha}  gives  the $j$th coordinate of  $\widehat{\rho_{2k+1} \circ \alpha}(\overline{i} + s)$ as $(2k+1)x_j + k$ for each $s$).  Therefore, we may take the homotopy here to be constant for $s = 0$ and $s= 1$, and not just relative the endpoint $s = 0$.
Finally, for set-up, we need to use $r'_j$ in place of $r_j$ in \eqref{eq: beta homotopy iiA} (in that sub-case, $r_j$ and $r'_j$ were equal, but generally $r_j$ pertains to the left half of
$[\overline{i}, \overline{i+1}] \times I_k$ and $r'_j$ to the right half, and it is on the right half that we wish to preserve the homotopy here).  That is, we could equally well write the homotopy of  \eqref{eq: beta homotopy iiA} as

\begin{equation}\label{eq: alt beta homotopy iiA}
G'_j(\overline{i}+s, t) = \begin{cases}
(2k+1)x_j + C^{k-s} (r'_j) & 0 \leq s < k-t\\
(2k+1)x_j + C^{t} (r'_j) & k-t \leq s \leq k+1+t\\
(2k+1)x_j + C^{s-(k+1)} (r'_j) & k+1+t < s \leq 2k+1.\end{cases}
\end{equation}
Then, with reference to \eqref{eq: alt beta homotopy iiA},  define here
\begin{equation}\label{eq: beta homotopy iiC}
G^C_j(\overline{i}+s, t) = \begin{cases}
G'_j(\overline{i}, t) =  \beta(\overline{i})_j& s = 0\\
G'_j(\overline{i}+s-1, t) & 1 \leq s \leq k\\
G'_j(\overline{i}+s, t) & k+1 \leq s \leq 2k+1.
\end{cases}
\end{equation}
As in the previous Case (ii) (B), a direct check shows that this is a homotopy from the $j$th coordinate of the restriction of $\beta$ to the $j$th coordinate of the restriction of
$\widehat{\rho_{2k+1} \circ \alpha}$.  Continuity on  $[ \overline{i}, \overline{i+1}] \times I_k$ follows from that of $G'_j$ (which, recall, is identical with $G_j$ in Case (ii) (A)).  We omit these details.

Across Case (i) and Cases (ii) (A), (B), and (C) of the last several pages, then, we have defined coordinate homotopies $G_j$ for each $j = 1, \ldots, n$. As discussed before we entered Case (i) above,  these coordinate homotopies assemble into a homotopy
$$G(\overline{i} + s, t) = \big( G_1(\overline{i} + s, t),  \ldots, G_n(\overline{i} + s, t)\big)  \colon [ \overline{i}, \overline{i+1}] \times I_k \to S(X, 2k+1),$$
from $\beta$ restricted to  $[ \overline{i}, \overline{i+1}]$ to
$\widehat{\rho_{2k+1} \circ \alpha}$ restricted to  $[ \overline{i}, \overline{i+1}]$.
By checking each of the formulas \eqref{eq: beta homotopy A}, \eqref{eq: beta homotopy iiA}, \eqref{eq: beta homotopy iiB1},  and \eqref{eq: beta homotopy iiC} when $s=0$ and $s = 2k+1$, we see that, in each case in each coordinate, the homotopy is relative the endpoints.  Indeed, we have (refer to the notation leading up to \lemref{lem: cover rho-alpha})
$$G(\overline{i}, t) = \overline{\mathbf{x_i}} \quad \text{ and } \quad  G(\overline{i}+2k+1, t) = \overline{\mathbf{x_{i+1}}},$$
for all $t \in I_k$.  Because $G$ is continuous  on each  $[ \overline{i}, \overline{i+1}] \times I_k$, and is well-defined where these intervals overlap, these homotopies---as well as the constant homotopies on $[ 0, \overline{0}] \times I_k$ and $[ \overline{M}, \overline{M}+ k] \times I_k$ as we defined them at the start of this proof---assemble into a based homotopy of based loops
$G\colon S(I_M, k) \times I_k \to S(X, 2k+1)$ from $\beta$ to   $\widehat{\rho_{2k+1} \circ \alpha}$.
\end{proof}

Finally, we arrive at the technical result we relied upon to establish injectivity of $(\rho_{2k+1})_*$ in the proof of  \thmref{thm: rho induces iso on pi1}.

\begin{corollary}\label{cor: homotopy alpha-rho cover}
Suppose $\alpha\colon I_M \to S(X, 2k+1)$ is a based loop in $S(X, 2k+1)$, for  $X \subseteq \Z^n$ any based digital image.  With $\widehat{\rho_{2k+1} \circ \alpha}\colon S(I_M, 2k+1) \to S(X, 2k+1)$ the standard cover of   $\rho_{2k+1} \circ \alpha\colon I_M \to X$, we have a based homotopy of based loops
$$\widehat{\rho_{2k+1} \circ \alpha} \approx \alpha\circ \rho_{2k+1}\colon S(I_M, 2k+1) \to S(X, 2k+1).$$
\end{corollary}

\begin{proof}
The ingredients are represented in the following diagram:
$$\xymatrix{ S(I_M, 2k+1) \ar[r]^-{ \widehat{\rho_{2k+1} \circ \alpha} }  \ar[d]_{\rho_{2k+1}} & S(X, 2k+1) \ar[d]^{\rho_{2k+1}} \\
I_M \ar[ru]_{\alpha} \ar[r]_-{\rho_{2k+1} \circ \alpha} & X.}$$
The lower-right triangle commutes tautologically; the outer rectangle commutes from the construction of the standard cover, as in \thmref{thm: path odd subdivision map}.
The assertion here is that the upper-left triangle commutes up to based homotopy (generally it does not commute).

The homotopy we want follows directly from \lemref{lem: homotopy alpha-rho beta} and \lemref{lem: homotopy cover-rho-alpha beta}, assuming symmetricity and transitivity of based homotopy of based loops. Since we have not given  arguments for these points, we provide an explicit argument here.   \lemref{lem: homotopy alpha-rho beta} constructs a based homotopy of based loops
$$H\colon S(I_M, 2k+1) \times I_M \to S(X, 2k+1)$$
from $\beta$ to $\alpha\circ \rho_{2k+1}$, and \lemref{lem: homotopy cover-rho-alpha beta} constructs a based homotopy of based loops
$$G \colon S(I_M, 2k+1) \times I_k\to S(X, 2k+1)$$
from $\beta$ to $\widehat{\rho_{2k+1} \circ \alpha}$. Define a map $\mathcal{H}\colon S(I_M, 2k+1) \times I_{2k+1} \to S(X, 2k+1)$ by
$$\mathcal{H}(s, t) = \begin{cases}  H(s, k-t) & 0 \leq t \leq k\\
G\big(s, t-(k+1)\big) & k+1 \leq t \leq 2k+1.\end{cases}$$
The only possible issue with continuity of $\mathcal{H}$ occurs where the two halves of the domain abut, namely, where $t = k$ and $t=k+1$.  But here, we have
$$\mathcal{H}(s, k) =  H(s, 0) = \beta(s) = G\big(s, 0\big) =  \mathcal{H}(s, k+1).$$
Therefore, if we have $(s, t) \sim (s', t')$ in $S(I_M, 2k+1) \times [k, k+1]$, then
$$|\mathcal{H}(s', t') - \mathcal{H}(s, t)| =  |\beta(s') - \beta(s)| \leq 1,$$
since we have $|s' - s| \leq 1$ and $\beta$ is continuous.  Now any pair of adjacent points in  $S(I_M, 2k+1) \times I_{2k+1}$  both lie in one of the sub-rectangles   $S(I_M, 2k+1) \times [0, k]$, $S(I_M, 2k+1) \times [k+1, 2k+1]$, or $S(I_M, 2k+1) \times [k, k+1]$, and it follows that $\mathcal{H}$ is continuous on the whole of $S(I_M, 2k+1) \times I_{2k+1}$.
A direct check confirms that $\mathcal{H}(s, 0) = \alpha\circ \rho_{2k+1}(s)$ and $\mathcal{H}(s, 2k+1) = (\widehat{\rho_{2k+1} \circ \alpha})(s)$, and also that
$\mathcal{H}(0, t) = \overline{\mathbf{x_0}} =  \mathcal{H}((2k+1)M+2k, t)$, so that $\mathcal{H}$ is the desired based homotopy of based loops.
\end{proof}

\section{Based maps and homotopies}\label{Appx: technical}

Here we present some basic material on maps and homotopies in the based setting. As we pointed out in the Introduction, whereas \cite{LOS19a, LOS19b} are concerned with  un-based maps and homotopies,  we need based versions of all definitions and results.  
Whilst some of the items given here extend or amplify those reviewed in \secref{sec: basics}, they are nonetheless useful in the development of ideas in the main body.  We refer to this appendix from numerous points in the main body, as well as from the material in \appref{sec: technical}.  But including all this material in the main body would slow the progression of ideas there.   So we have elected to collect it here, rather than disperse it throughout the main body.      

Also in this appendix, and for the convenience of the reader, we give the statements of two results from \cite{LOS19b} that are used in some of our main results.

\subsection{Products and Homotopy}

We record a number of elementary observations that are used, implicitly or explicitly, in the main body.

\begin{applemma}\label{lem: product}
For based digital images $(X, x_0)$ and $(Y, y_0)$, the projections onto either factor $p_1\colon X \times Y \to X$ and $p_2\colon X \times Y \to Y$ are based maps.  Suppose given based maps of digital images $f\colon (A, a_0) \to (X, x_0)$ and $g\colon (A, a_0) \to (Y, y_0)$.  Then there is a unique based map, which we write $(f, g)\colon (A, a_0) \to \big(X \times Y, (x_0, y_0)\big)$ that satisfies $p_1\circ (f, g) = f$ and $p_2\circ (f, g) = g$.
\end{applemma}

\begin{proof}
The first assertion follows immediately from the definitions.   The map $(f, g)$ is defined as $(f, g)(a) = \big( f(a), g(a) \big)$.  It is immediate from the definitions that this map is continuous and based.  This is evidently the unique map with the suitable coordinate functions.
\end{proof}

Because of the rectangular nature of the digital setting, it is often convenient to consider the \emph{product of maps}, as follows.

\begin{appdefinition}\label{def: map product}
Given functions of digital images $f_i \colon X_i \to Y_i$ for $i = 1, \ldots, n$, we define their \emph{product} function
$$f_1 \times \cdots \times f_n \colon X_1 \times \cdots \times  X_n \to Y_1 \times \cdots \times Y_n$$
as $(f_1 \times \cdots \times f_n) (x_1, \ldots, x_n) = \big(f_1(x_1), \ldots, f_n(x_n) \big)$.
\end{appdefinition}

\begin{applemma}\label{lem: map product conts}
Given based maps of digital images $f_i \colon X_i \to Y_i$ for $i = 1, \ldots, n$, their product $f_1 \times \cdots \times f_n$ is a (continuous) based map.
\end{applemma}

\begin{proof}
This follows directly from the definitions.
\end{proof}

We defined based homotopy in \defref{def: Based Homotopy}. 

\begin{applemma}\label{lem:composition and based homotopy}
Suppose based maps $f, f'\colon X \to Y$ of based digital images $X \subseteq \Z^m$ and  $Y \subseteq \Z^n$ are based-homotopic, and also based maps $g, g'\colon Y \to Z$ of based digital images with  $Z \subseteq \Z^p$ are based-homotopic.    Then $g\circ f, g'\circ f'\colon X \to Z$ are based-homotopic.
\end{applemma}

\begin{proof}
Suppose $H \colon X \times I_N \to Y$ is a based homotopy from $f$ to $f'$ and $G \colon Y \times I_M \to Z$ is a based homotopy from $g$ to $g'$.
Define $P \colon X \times I_{N+M} \to Z$ by
$$P(x, t) = \begin{cases}  g\circ H(x, t) & 0 \leq t \leq N\\ G\circ(f' \times T)(x, t) & N \leq t \leq N+M,\end{cases}$$
where $T \colon [N, N+M] \to I_M$ is the translation $T(t) = t_N$.   We check that $P$ defines a continuous map on $X \times I_{N+M}$.  For this, $T$ is clearly continuous and so
$f' \times T$ is  continuous by \lemref{lem: map product conts}.  As compositions of continuous maps, the two parts of $P$ are continuous on $X \times I_N$ and $X \times [N, N+M]$. They agree on their overlap $X \times \{N\}$, as is easily checked.  Furthermore, they piece together to give a continuous map.  This is because any two adjacent points $(x, t) \sim (x', t')$ of $X \times I_{N+M}$ must have $|t - t'| \leq 1$, and thus either both are in $X \times I_N$ or both are in $X \times [N, N+M]$.  From the continuity of the two parts, we have $P(x, t) \sim_Z P(x', t')$, so $P$ is continuous on $X \times I_{N+M}$.   Now $P$ is a homotopy from $g\circ f$ to $g'\circ f'$, as is easily checked.  Furthermore, we have
$$P(x_0, t) = \begin{cases}  g\circ H(x_0, t) = g(y_0) = z_0 & 0 \leq t \leq N\\  G\big(f'(x_0), T(t)\big) = G\big(y_0, T(t)\big) = z_0& N \leq t \leq N+M,\end{cases}$$
since $H$ and $G$ are both based homotopies.  Hence $P$ is a based homotopy from $g\circ f$ to $g'\circ f'$.
\end{proof}

We defined based homotopy of based loops in \defref{def: Based Homotopy of loops}.

\begin{applemma}\label{lem: composition based homotopy loops}
Suppose based loops $\alpha, \beta\colon I_L \to Y$ in a based digital image $Y \subseteq \Z^n$ are based homotopic as based loops, and based maps $g, g'\colon Y \to Z$ of based digital images with  $Z \subseteq \Z^p$ are based-homotopic.  Then the based loops $g\circ \alpha, g'\circ \beta\colon I_L \to Z$ are based homotopic as based loops in $Z$.  Furthermore, considering
$\alpha\circ \rho_k, \beta\circ \rho_k\colon S(I_L, k) \to Y$ as based loops $I_{kL+k-1} \to Y$ in $Y$, they are based homotopic as based loops in $Y$.
\end{applemma}

\begin{proof}
These are basically two special cases of \lemref{lem:composition and based homotopy}; we just need to be sure that both ends of the loop are preserved through the homotopy.
For the first point, suppose $H \colon I_L \times I_N \to Y$ is a based homotopy of based loops from $\alpha$ to $\beta$ and $G \colon Y \times I_M \to Z$ is a based homotopy from $g$ to $g'$.
Define $P \colon I_L \times I_{N+M} \to Z$ as in the proof of \lemref{lem:composition and based homotopy} by
$$P(s, t) = \begin{cases}  g\circ H(s, t) & 0 \leq t \leq N\\ G\circ(\beta \times T)(x, t) & N \leq t \leq N+M,\end{cases}$$
Then $P$ is a (continuous)  homotopy from $g\circ \alpha$ to $g'\circ \beta$ that satisfies $P(0, t) = z_0$ for $t \in I_{N+M}$, by \lemref{lem:composition and based homotopy}.  In addition, here we have
$$P(L, t) = \begin{cases}  g\circ H(L, t) = g(y_0) = z_0 & 0 \leq t \leq N\\  G\big(\beta(L), T(t)\big) = G\big(y_0, T(t)\big) = z_0& N \leq t \leq N+M,\end{cases},$$
since $\beta$ is a based loop in $Y$ and $H$ is a based homotopy of based loops.  Thus, $P$ is a based homotopy of based loops.

For the second point, recall that $\rho_k\colon S(I_L, k) \to I_L$ satisfies $\rho_k(0) = 0$ and $\rho_k(kL+k-1) = L$.  A special case of the argument of  \lemref{lem:composition and based homotopy}, in which  $H$ is redundant, may be used here.  Namely, define  $P \colon S(I_L, k) \times I_{N} \to Y$  by
$$P(s, t) = G\circ(\rho_k \times \mathrm{id}_{I_N})(s, t),$$
where $G \colon I_L \times I_N \to Y$ is a based homotopy of based loops from $\alpha$ to $\beta$.  Continuity follows directly from \lemref{lem: map product conts} here. A direct check confirms that $P$ is a based homotopy of based loops from  $\alpha\circ \rho_k$ to $\beta\circ \rho_k$.
\end{proof}

\begin{appdefinition}[Based Homotopy Equivalence]\label{def: based h.e.}
Let $f \colon X \to Y$ be a based map of based digital  images.  If there is a based map $g \colon Y \to X$ such that $g\circ f \approx \text{id}_X$ and $f \circ g \approx \text{id}_Y$, then $f$ is a \emph{based homotopy equivalence}, and $X$ and $Y$ are said to be \emph{based homotopy equivalent}, or to have the same \emph{based homotopy type}.
\end{appdefinition}

Based homotopy equivalence is an equivalence relation on based digital images.  This may be shown with an argument identical to that used to show the topological  counterpart of this fact.  We leave the proof as an exercise.

\subsection{Subdivision}

Subdivision behaves well with respect to products.  For any digital images $X \subseteq \Z^m$ and $Y \subseteq \Z^n$ and any $k \geq 2$ we may identify
$$S(X \times Y, k) \cong S(X, k) \times S(Y, k)$$
and, furthermore, the standard projection $\rho_k\colon S(X \times Y, k) \to X \times Y$ may be identified with the product of the standard projections on $X$ and $Y$, thus:
$$\rho_k = \rho_k \times \rho_k\colon  S(X, k) \times S(Y, k) \to X \times Y.$$

The projection $\rho_k \colon S(X, k) \to X$ may be factored in various ways.  For example, if $k = pq$, then we may write
$$\rho_k = \rho_p\circ \rho_q \colon S(X, k) \to S(X, p) \to X.$$
A different sort of ``partial projection" that may also be used to factor $\rho_k$ is as follows.

\begin{appdefinition}\label{def: rho^c}
For any $x \in \Z$ and any  $k \geq 2$, recall that the subdivision $S(x, k)$ may be described as $S(x, k) = [kx, kx + k-1]$.    Then, for $k \geq 3$, define a function
$$\rho^c_{k} \colon S(x, k) \to S(x, k-1)$$
as
$$\rho^c_{k}(kx + j) = \begin{cases}  (k-1)x + j & 0 \leq j \leq  \lfloor k/2 \rfloor -1\\
(k-1)x + j-1 &  \lfloor k/2 \rfloor \leq j \leq  k-1.\end{cases}
$$
Next, for any $x = (x_1, \ldots, x_n) \in \Z^n$, with the identifications
$$S(x, k) = S(x_1, k) \times \cdots \times S(x_n, k)$$
and
$$S(x, k-1) = S(x_1, k-1) \times \cdots \times S(x_n, k-1),$$
in which each $S(x_i, k)$ and $S(x_i, k-1)$ are viewed as 1D,
define $\rho^c_{k} \colon S(x, k) \to S(x, k-1)$ as the product of functions
$$\rho^c_{k}\times \cdots \times \rho^c_{k} \colon S(x_1, k) \times \cdots \times S(x_n, k) \to  S(x_1, k-1) \times \cdots \times S(x_n, k-1).$$
Finally, for any digital image $X \subseteq \Z^n$, define
$$\rho^c_{k} \colon S(X, k) \to S(X, k-1)$$
by viewing each subdivision as a (disjoint) union
$$ S(X, k) = \coprod_{x \in X}\ S(x, k) \qquad \text{and} \qquad S(X, k-1) = \coprod_{x \in X}\ S(x, k-1)$$
and assembling a global  $\rho^c_{k}$ on $S(X, k)$ from the individual $\rho^c_{k} \colon S(x, k) \to S(x, k-1)$ as just defined.  In \cite{LOS19b}, we show that, for each $k\geq 3$,
the function $\rho^c_{k} \colon S(X, k) \to S(X, k-1)$ is continuous.
\end{appdefinition}

These partial projections are often useful, because they allow us to factor the projection $\rho_k\colon S(X, k) \to X$ as
$$\rho_k = \rho_{k-1}\circ \rho^c_k\colon S(X, k) \to S(X, k-1) \to X,$$
with $\rho_{k_1}\colon S(X, k-1) \to X$ the standard projection and $\rho^c_k\colon S(X, k) \to S(X, k-1)$ the map from \defref{def: rho^c}.  Indeed, we make crucial use of these partial projections in several results in the main body.

\subsection{Subdivision-Based Homotopy}

Recall from \defref{def:subdn basepoints} our conventions on basepoints vis-{\`a}-vis subdivision. Also, recall our notational convention that when $k=1$, $S(X, k) = S(X, 1) = X$ and $\rho_1\colon S(X, 1) \to X$ is just the identity map.

We defined subdivision-based homotopy of maps in \defref{def: subdn homotopic maps}.

\begin{appproposition}\label{prop: subdn-bsd homotopy is equiv}
Suppose we have based digital images $X \subseteq \Z^m$ and  $Y \subseteq \Z^n$. Consider the set
$$\mathcal{S} = \left\{ f\colon S(X, k) \to Y \mid f \text{ is a based map and } k \geq 1\right\}$$
of all based maps from any subdivision of $X$ to $Y$.
Subdivision-based homotopy is an equivalence relation on the set $\mathcal{S}$.
\end{appproposition}

\begin{proof}
The proofs of reflexivity and symmetry are more-or-less tautological from \defref{def: subdn homotopic maps}.  We omit their details.  For transitivity, we argue as follows.  Suppose we have based maps $f \colon S(X, k) \to Y$,  $g \colon S(X, l) \to Y$, and $h \colon S(X, m) \to Y$, with $f$ and $g$ subdivision-based homotopic, and $g$ and $h$ subdivision-based homotopic.  We wish to show that $f$ and $h$ are subdivision-based homotopic.

For $f$ and $g$ , we have $k', l'$ with $kk' = ll'$ and a based homotopy $H \colon S(X, kk')\times I_N = S(X, ll')\times I_N \to Y$ from
$f\circ \rho_{k'}$ to $g\circ \rho_{l'}$.  Likewise, for $g$ and $h$ , we have $l'', m'$ with $ll'' = mm'$ and a based homotopy $G \colon S(X, ll'')\times I_M = S(X, mm')\times I_M \to Y$ from
$g\circ \rho_{l''}$ to $h\circ \rho_{m'}$.  From the special case of \lemref{lem:composition and based homotopy} in which the first homotopy is redundant,
$$P = H\circ(\rho_{l''} \times \mathrm{id}_{I_N})\colon S(X, ll'l'')\times I_N = S(X, kk'l'')\times I_N \to  Y$$
is a based homotopy of based maps from $f\circ \rho_{k'}\circ\rho_{l''}$ to $g\circ\rho_{l'}\circ\rho_{l''} =  g\circ\rho_{l'l''}$.  Also,
$$P'= G\circ(\rho_{l''} \times \mathrm{id}_{I_M})\colon S(X, ll''l')\times I_M = S(X, mm'l')\times I_M \to  Y$$
is a based homotopy of based maps from $g\circ\rho_{l''}\circ\rho_{l'} =  g\circ\rho_{l''l'}$ to $h\circ\rho_{m'}\circ\rho_{l'}$.
Then we piece these homotopies together into $P''\colon S(X, kk'l'')\times I_{N+M} =  S(X, mm'l')\times I_{N+M}\to  Y$ defined by
$$P''(x, t) = \begin{cases}  P(x, t) & 0 \leq t \leq N\\ P'(x, t-N) & N \leq t \leq N+M.\end{cases}$$
The two parts of this homotopy agree on their overlap (when $t = N$) and assemble into a continuous whole by the same argument as was used in the proof of \lemref{lem:composition and based homotopy}.  It is straightforward to check that $P''$ is a based homotopy of based maps from  $f\circ \rho_{k'l''}$ to $h\circ\rho_{l'm'}$  The result follows.
\end{proof}

We defined  subdivision-based homotopy of based loops in \defref{def: subdn homotopic loops}.

\begin{appproposition}\label{prop: subdn-bsd homotopy is equiv of loops}
For a based digital image $Y \subseteq \Z^n$, consider the set
$$\mathcal{S} = \left\{ \alpha\colon I_N \to Y \mid \alpha \text{ is a based loop of length } N \geq 1\right\}$$
of all based loops in $Y$ of any length.
Subdivision-based homotopy of based loops is an equivalence relation on the set $\mathcal{S}$.
\end{appproposition}

\begin{proof}
This is basically a special case of \propref{prop: subdn-bsd homotopy is equiv}; we just need to be sure that both ends of a loop are preserved through  homotopies.
The proofs of reflexivity and symmetry here are more-or-less tautological from \defref{def: subdn homotopic loops}.  We omit their details, and focus on transitivity.  Suppose we have based loops $\alpha \colon I_L \to Y$,  $\beta \colon I_M \to Y$, and $\gamma \colon I_N \to Y$, with $\alpha$ and $\beta$ subdivision-based homotopic as based loops, and $\beta$ and $\gamma$ subdivision-based homotopic as based loops.  We wish to show that $\alpha$ and $\gamma$ are subdivision-based homotopic as based loops.

For $\alpha$ and $\beta$ , we have $l, m$ with $l(L+1) = m(M+1)$ and a based homotopy of based loops $H \colon I_{lL + l-1}\times I_S = I_{mM + m-1}\times I_S \to Y$ from
$\alpha\circ \rho_l$ to $\beta\circ \rho_m$.  Likewise, for $\beta$ and $\gamma$, we have $\mu, \nu$ with $\mu(M+1) = \nu(N+1)$ and a based homotopy of based loops $G \colon I_{\mu M + \mu-1}\times I_T = I_{\nu N + \nu-1}\times I_T \to Y$ from $\beta\circ \rho_{\mu}$ to $\gamma\circ \rho_{\nu}$.

From the (proof of the) second point of \lemref{lem: composition based homotopy loops},
$$P = H\circ(\rho_{\mu} \times \mathrm{id}_{I_S})\colon S(I_L, l\mu)\times I_S = S(M, m\mu)\times I_S \to  Y$$
is a based homotopy of based loops from $\alpha\circ\rho_l\circ\rho_{\mu}$ to $\beta\circ\rho_m\circ\rho_{\mu} =  \beta\circ\rho_{m\mu}$.  Also,
$$P'= G\circ(\rho_{m} \times \mathrm{id}_{I_T})\colon S(I_M, \mu m)\times I_T = S(I_N, \nu m)\times I_T \to  Y$$
is a based homotopy of based loops from $\beta\circ\rho_{m\mu} = \beta\circ\rho_{\mu}\circ\rho_{m}$ to $\gamma\circ\rho_{\nu}\circ\rho_{m}$.
Then we piece these homotopies together into $P''\colon S(I_L, l\mu)\times I_{S+T} =  S(I_N, \nu m)\times I_{S+T}\to  Y$ defined by
$$P''(s, t) = \begin{cases}  P(s, t) & 0 \leq t \leq S\\ P'(s, t-S) & S \leq s \leq S+T.\end{cases}$$
The two parts of this homotopy agree on their overlap (when $t = S$) and assemble into a continuous whole by the same argument as was used in the proof of \lemref{lem:composition and based homotopy}.  It is straightforward to check that $P''$ is a homotopy  from
$\alpha\circ\rho_{l\mu}$ to  $\gamma\circ\rho_{m\nu}$, and also that $P''(0, t) = P''(l\mu L + l \mu -1, t) = P''(m\nu N + m \nu -1, t) = y_0$, so that $P''$ is a based homotopy of based loops.  The result follows.
\end{proof}

In this last item of this collection of material on based maps and homotopies, we define  a relation between digital images that is less rigid than the relation of homotopy equivalence (\emph{per}  \defref{def: based h.e.}) and to which we refer from several places in the Introduction and the main body.

\begin{appdefinition}[Subdivision-Based Homotopy Equivalence]\label{def: subdn homotopy equiv}
For based digital images $X \subseteq \Z^m$ and $Y\subseteq \Z^n$  and some $k, l \geq 1$, suppose that we have the following data:
\begin{itemize}
\item[(a)] Based maps $f\colon S(X, k) \to Y$ and $g\colon S(Y, l) \to X$;

\item[(b)] Based maps (coverings) $F$ and $G$ that make  the following diagrams commute
$$\xymatrix{ S(X, kl) \ar[d]_{\rho_{l}} \ar[r]^-{F} & S(Y, l) \ar[d]^{\rho_{l}}\\
S(X, k) \ar[r]_-{f} & Y} \quad \text{and} \quad \xymatrix{ S(Y, kl) \ar[d]_{\rho_{k}} \ar[r]^-{G} & S(X, k) \ar[d]^{\rho_{k}}\\
S(Y, l) \ar[r]_-{g} & X};$$
We say that $F$ is a \emph{cover} of $f$ and $G$ is a \emph{cover} of $g$.

\item[(c)]  $f\circ G\colon S(Y, kl) \to Y$ subdivision-based homotopic to $\mathrm{id}_Y\colon Y \to Y$ and $g\circ F\colon S(X, kl) \to X$ subdivision-based homotopic to $\mathrm{id}_X\colon X \to X$.
\end{itemize}
Then we say that $X$ and $Y$ are \emph{subdivision-based homotopy equivalent}.
\end{appdefinition}

\begin{remark}
With $k  = l = 1$ and $F=f$, $G=g$ and the subdivision-based homotopies of (c) ordinary based homotopies, the notion of subdivision-based  homotopy equivalence reduces to that of based homotopy equivalence.
\end{remark}

\subsection{Results about Subdivision of Maps from \cite{LOS19b}}

For the convenience of the reader, we state two results from \cite{LOS19b} that we use in the proof of \thmref{thm: rho induces iso on pi1} and its annexe \appref{sec: technical}.  We refer to \cite{LOS19b} for all details of the proofs, and simply comment on some of the ingredients used.

\begin{apptheorem}[Th.4.1 and Cor.6.2.(C) of \cite{LOS19b}]\label{thm: path odd subdivision map}
Suppose we are given $\alpha\colon I_N \to Y$, a path of length $N$ in any digital image $Y \subseteq \Z^n$.  For any odd $2k+1 \geq 3$,  there is  a canonical map of subdivisions
$$\widehat{\alpha}\colon S(I_N, 2k+1) = I_{N(2k+1)+2k}  \to S(Y, 2k+1),$$
or \emph{standard cover of} $\alpha$,
that covers the given path, in the sense that the following diagram commutes:
$$\xymatrix{ S(I_N, 2k+1) \ar[d]_{\rho_{2k+1}} \ar[r]^-{\widehat{\alpha}} & S(Y, 2k+1) \ar[d]^{\rho_{2k+1}}\\
I_N \ar[r]_-{\alpha} & Y}$$
For $Y$ a based digital image and $\alpha$ a based loop, with $\alpha(0) = \alpha(N) = y_0$, then the standard cover $\widehat{\alpha}$ is a based loop: we have
$\widehat{\alpha}(0) = \widehat{\alpha}\big( (2k+1)N + 2k \big) = \overline{y_0}$.
\end{apptheorem}

\begin{proof}[Proof (Gloss on the proof of Theorem 4.1 of \cite{LOS19b})]
Refer to \defref{def:subdn basepoints} for our conventions and notation on basepoints and centres.  For instance,  
for $i \in I_N \subseteq \Z$ we  write $\overline{i} = (2k+1)i+k \in S(i, 2k+1)$.
Then the standard cover may be described as
\begin{equation}\label{eq: standard cover}
\widehat{\alpha}(j) =\begin{cases}  \ \  \overline{\alpha(0)} & 0 \leq  j < k \\
\\
\widehat{\alpha}(\overline{i} + t) =  \overline{\alpha(i)} + t\big[   \alpha(i+1) - \alpha(i)   \big]  & 0 \leq i \leq N, 0 \leq t \leq 2k\\
\\
\overline{\alpha(N)} & \overline{N} \leq  j \leq \overline{N} + 2k  \\
 \end{cases}
 \end{equation}
In the middle line, we use coordinate-wise (vector) addition and scalar multiplication.   If $\alpha(i) = (y_1, \ldots, y_n)$ and  $\alpha(i+1) = (y'_1, \ldots, y'_n)$, then the coordinates of the middle line of  \eqref{eq: standard cover} may be specified as
\begin{equation}\label{eq: coordinate standard cover}
\widehat{\alpha}(\overline{i} + t) =  \big( (2k+1)y_1 +k + t[  y'_1 - y_1  ], \ldots, (2k+1)y_n + k + t[  y'_n - y_n  ]\big)
 \end{equation}
for $0 \leq t \leq 2k$.
  A basic feature of the way in which we construct $\widehat{\alpha}$ is that it satisfies $\widehat{\alpha}\big(\overline{i}\big) = \overline{ \alpha(i)}$.  In particular, if $\alpha$ is a based loop, then we have $\widehat{\alpha}\big(\overline{0}\big) = \widehat{\alpha}\big(\overline{N}\big) = \overline{ y_0}$.  But also, in the construction, we define $\widehat{\alpha}(j) = \overline{ \alpha(0)}$ each $j$ with $0 \leq j \leq \overline{0} = k$, and $\widehat{\alpha}(j) = \overline{ \alpha(N)}$ for each $j$ with $(2k+1)N + k = \overline{N} \leq j \leq (2k+1)N +2k$.  With $0$ the basepoint of  $S(I_N, 2k+1) = I_{(2k+1)N + 2k}$ and $\overline{ y_0}$ the basepoint of $S(Y, 2k+1)$, this justifies the last part of the theorem.
\end{proof}

The following  result from  \cite{LOS19b} functions as a homotopy lifting result for based paths or loops.

\begin{apptheorem}[Th.5.1 and Cor.6.2.(C) of \cite{LOS19b}]\label{thm: 2-D subdivision map rectangle}
Suppose we are given a map $H \colon I_M \times I_N \to Y$ with $Y \subseteq \Z^n$ any  digital image.  For any $k \geq 1$,  there is  a canonical choice of map $\widehat{H}\colon S(I_M, 2k+1) \times S(I_N, 2k+1) \to S(Y, 2k+1)$ that makes the following diagram commute:
$$\xymatrix{ S(I_M, 2k+1) \times S(I_N, 2k+1) \ar[d]_{\rho_{2k+1} \times \rho_{2k+1} =\rho_{2k+1}} \ar[r]^-{\widehat{H}} & S(Y, 2k+1) \ar[d]^{\rho_{2k+1}}\\
I_M \times I_N \ar[r]_-{H} & Y}$$
Furthermore, if $H$ is a based homotopy of based loops from $\alpha\colon I_M \to Y$ to $\beta\colon I_M \to Y$, then  $\widehat{H}\colon I_{(2k+1)M+ 2k} \times I_{(2k+1)N+ 2k} \to S(Y, 2k+1)$ is a based homotopy of based loops from $\widehat{\alpha}\colon I_{(2k+1)M+ 2k} \to S(Y, 2k+1)$ to $\widehat{\beta} \colon I_{(2k+1)M+ 2k} \to S(Y, 2k+1)$, the standard covers of $\alpha$ and $\beta$ as in \thmref{thm: path odd subdivision map}.
\end{apptheorem}

\begin{proof}[Proof (Gloss on the proof of Theorem 5.1 of \cite{LOS19b})]
In the proof of \cite[Th.5.1]{LOS19b}, we begin by defining the covering homotopy $\widehat{H}$ around the edges of $S(I_M, 2k+1) \times S(I_N, 2k+1)$ using the standard covers of the paths in $Y$ gotten by restricting $H$ to the edges of $I_M \times I_N$.  In particular, this construction yields $\widehat{H}(s, 0) = \widehat{\alpha}(s)$ and
$\widehat{H}(s, (2k+1)N+ 2k) = \widehat{\beta}(s)$ for $s \in I_{(2k+1)M+ 2k}$.  Also, along the vertical edges, since $H$ is a based homotopy of based loops, we have $H(0, t)  = y_0 = H((2k+1)M+ 2k, t)$ for $t \in I_{(2k+1)N+ 2k}$.  Now for the constant map $C_N \colon I_N \to Y$, the standard cover is again a constant map (at $\overline{y_0} \in S(Y, 2k+1)$): we have
$\widehat{C_N} = C_{(2k+1)N+ 2k}\colon I_{(2k+1)N+ 2k} \to S(Y, 2k+1)$.  Hence, again from the construction of  the covering homotopy $\widehat{H}$, we have
$$\widehat{H}(0, t) =  \overline{y_0} =  \widehat{H}((2k+1)M+ 2k, t)$$
for $t \in I_{(2k+1)N+ 2k}$.  Hence $\widehat{H}$ is indeed a based homotopy of based loops.
\end{proof}


\begin{thebibliography}{10}

\bibitem{Bo99}
L.~Boxer, \emph{A classical construction for the digital fundamental group}, J.
  Math. Imaging Vision \textbf{10} (1999), no.~1, 51--62. \MR{1692842}

\bibitem{Bo05}
\bysame, \emph{Properties of digital homotopy}, J. Math. Imaging Vision
  \textbf{22} (2005), no.~1, 19--26. \MR{2138582}

\bibitem{Bo06}
\bysame, \emph{Digital products, wedges, and covering spaces}, J. Math. Imaging
  Vision \textbf{25} (2006), no.~2, 159--171. \MR{2267137}

\bibitem{Bo06a}
\bysame, \emph{Homotopy properties of sphere-like digital images}, J. Math.
  Imaging Vision \textbf{24} (2006), no.~2, 167--175. \MR{2226398}

\bibitem{Bo18}
\bysame, \emph{Alternate product adjacencies in digital topology}, Appl. Gen.
  Topol. \textbf{19} (2018), no.~1, 21--53. \MR{3784715}

\bibitem{BS16}
L.~Boxer and P.~C. Staecker, \emph{Fundamental groups and {E}uler
  characteristics of sphere-like digital images}, Appl. Gen. Topol. \textbf{17}
  (2016), no.~2, 139--158.

\bibitem{B-S18}
\bysame, \emph{Remarks on pointed digital homotopy}, Topology Proc. \textbf{51}
  (2018), 19--37. \MR{3633236}

\bibitem{Evako2006}
A.~V. Evako, \emph{Topological properties of closed digital spaces: One method
  of constructing digital models of closed continuous surfaces by using
  covers}, Computer Vision and Image Understanding \textbf{102} (2006),
  134--144.

\bibitem{Kha87}
E.~Khalimsky, \emph{Motion, deformation and homotopy in finite spaces},
  Proceedings of the 1987 IEEE International Conference on Systems, Man and
  Cybernetics \textbf{87CH2503-l} (1987), 227--234.

\bibitem{Kong89}
T.~Y. Kong, \emph{A digital fundamental group}, Computers and Graphics
  \textbf{13} (1989), 159--166.

\bibitem{K-R-R92}
T.~Y. Kong, A.~W. Roscoe, and A.~Rosenfeld, \emph{Concepts of digital
  topology}, Topology Appl. \textbf{46} (1992), no.~3, 219--262, Special issue
  on digital topology. \MR{1198732}

\bibitem{LOS19a}
G.~Lupton, J.~Oprea, and N.~Scoville, \emph{Homotopy theory in digital
  topology}, arXiv:1905.07783 [math.AT], 2019.

\bibitem{LOS19b}
\bysame, \emph{Subdivision of maps in digital topology}, arXiv:1906.03170
  [math.AT], 2019.

\bibitem{Mas91}
William~S. Massey, \emph{A basic course in algebraic topology}, Graduate Texts
  in Mathematics, vol. 127, Springer-Verlag, New York, 1991. \MR{1095046}

\bibitem{Ro86}
A.~Rosenfeld, \emph{`{C}ontinuous' functions on digital pictures}, Pattern
  Recognition Letters \textbf{4} (1986), 177--184.

\end{thebibliography}

\providecommand{\bysame}{\leavevmode\hbox to3em{\hrulefill}\thinspace}
\providecommand{\MR}{\relax\ifhmode\unskip\space\fi MR }
\providecommand{\MRhref}[2]{%
  \href{http://www.ams.org/mathscinet-getitem?mr=#1}{#2}
}
\providecommand{\href}[2]{#2}

\end{document}